\documentclass{amsart}
\usepackage[implicit,naturalnames,hyperfootnotes=false,final]{hyperref}
\usepackage[pdftex]{color}
\usepackage{lmodern}
\usepackage{xcolor}
\usepackage{enumitem}
\usepackage{amsmath,amssymb,amsthm}
\usepackage{bbm,amsfonts,mathrsfs,calligra,dsfont}
\usepackage{nicefrac}
\usepackage{esint}
\usepackage{mathtools}
\usepackage{ifthen}
\usepackage{xparse}
\usepackage{xspace}
\usepackage{calc}
\usepackage[nostandardtensors]{simpletensors}
\usepackage{partial}
\usepackage{bracket-hack}
\usepackage{auto-eqlabel}
\usepackage{refcount}
\usepackage{varioref}

\vrefwarning%
\usepackage{tikz}\usetikzlibrary{calc,intersections,shapes,decorations.markings,external}

\renewcommand\thesupercolon{\mathrel{\vcenter{\baselineskip0.4ex \lineskiplimit0pt\hbox{\small.}\hbox{\small.}}}}%

\NewDocumentCommand\SAdS{s}{(Schwarz\-schild-\nolinebreak)\nolinebreak Anti-de Sitter\IfBooleanTF{#1}\relax\xspace}

\newenvironment{itemizeequation}%
{\par\noindent\minipage[r]{\textwidth-3.5em}\csname equation*\endcsname}%
{\csname endequation*\endcsname\endminipage\par\noindent}

\newtheoremstyle{mytheorem}{3pt}{}{\itshape}{}{\bfseries}{\nopagebreak\newline}{.5em}{}
\newtheoremstyle{mydefinition}{3pt}{}{}{}{\bfseries}{\nopagebreak\newline}{.5em}{}
\gdef\mytheorem{\theoremstyle{mytheorem}}
\gdef\mydefinition{\theoremstyle{mydefinition}}
\gdef\mytheoremcounter{theorem}
\mytheorem

\makeatletter
\gdef\pretomacro#1#2{\bgroup%
 \def\pretomacro@tmp{#1}%
 \expandafter\expandafter\expandafter\def%
 \expandafter\expandafter\expandafter\pretomacro@tmp%
 \expandafter\expandafter\expandafter{%
 \expandafter\pretomacro@tmp#2}%
 \global\let#2\pretomacro@tmp\egroup}
\gdef\cspretomacro#1#2{\bgroup%
 \expandafter\let\expandafter\cspretomacro@tmp\csname#2\endcsname%
 \pretomacro{#1}\cspretomacro@tmp%
 \expandafter\global\expandafter\let\csname #2\endcsname\cspretomacro@tmp%
 \undef\cspretomacro@tmp\egroup}
\gdef\@thlabelnew#1#2{%
 \expandafter\gdef\csname @thlabel@#1\endcsname{#2~\ref{#1}}%
 \expandafter\gdef\csname @thlabel@#1*\endcsname{#2~\ref{#1} (\vpageref*{#1})}}
\gdef\labelth{\expandafter\@th@label@snd\expandafter{\lasttheoremType}}
\gdef\@th@label@snd#1#2{\label{#2}\protected@write\@auxout{}{\string\@thlabelnew{#2}{#1}}}
\gdef\@thlabel@IfLaTeXLabelNotExistsTF#1{%
 \expandafter\ifx\csname r@#1\endcsname\relax%
 \expandafter\@firstoftwo\else\expandafter\@secondoftwo\fi}%
\gdef\@thlabel@IfLabelNotExistsTF#1{%
 \expandafter\ifx\csname @thlabel@#1\endcsname\relax%
 \expandafter\@firstoftwo\else\expandafter\@secondoftwo\fi}%
\NewDocumentCommand\refth{sm}{%
 \@thlabel@IfLabelNotExistsTF{#2}%
 {{\normalfont\@thlabel@IfLaTeXLabelNotExistsTF{#2}%
  {\textbf{[??]~(??)}\message{LaTeX Warning: Undefined theorem label: #2 on page \thepage}}%
  {\textbf{[??]}~\ref{#1}\message{LaTeX Warning: Referenced on theorem #2 on page \thepage which was labeled with label instead of labelth}}}}%
 {\IfBooleanTF{#1}{\csname @thlabel@#2*\endcsname}{\csname @thlabel@#2\endcsname}}}
\makeatother

\NewDocumentCommand\Newtheorem{oomomo}
{\IfValueTF{#4}%
 {\IfValueTF{#6}{\newtheorem{#3}[#4]{#5}[#6]}{\newtheorem{#3}[#4]{#5}}}%
 {\IfValueTF{#6}{\newtheorem{#3}{#5}[#6]}{\newtheorem{#3}{#5}}}%
 \cspretomacro{%
  \global\edef\lasttheoremtype{\IfValueTF{#1}{#1}{#5}}
	\global\edef\lasttheoremType{\IfValueTF{#2}{#2}{#5}}}{#3}}

\Newtheorem[theorem]{theorem}{Theorem}[section]
\Newtheorem[theorem][Theorem]{maintheorem}{Theorem}
\Newtheorem[corollary]{corollary}[\mytheoremcounter]{Corollary}
\Newtheorem[lemma]{lemma}[\mytheoremcounter]{Lemma}
\Newtheorem[example]{example}[\mytheoremcounter]{Example}
\Newtheorem[proposition]{proposition}[\mytheoremcounter]{Proposition}
\Newtheorem[assumptions]{assumptions}[\mytheoremcounter]{Assumptions}
\mydefinition
\Newtheorem[definition]{definition}[\mytheoremcounter]{Definition}
\Newtheorem[definition and lemma][Definition and lemma]{defiundlemma}[\mytheoremcounter]{Definition and Lemma}
\Newtheorem[corollary and definition][Corollary and definition]{corunddef}[\mytheoremcounter]{Corollary and Definition}
\Newtheorem[construction]{construction}[\mytheoremcounter]{Construction}
\Newtheorem[notation]{notation}[\mytheoremcounter]{Notation}
\theoremstyle{remark}
\Newtheorem[remark]{remark}[\mytheoremcounter]{Remark}
\Newtheorem[note]{note}[\mytheoremcounter]{Note}

\NewDocumentEnvironment{Proof}{d<>o}{%
 \IfValueTF{#1}{\proof[\IfValueTF{#2}{#2}{Proof}~of~\refth{#1}]}
 {\IfValueTF{#2}{\proof[#2]}\proof}%
 \def\paragraph##1{\stepcounter{paragraph}\par\medskip\par\noindent%
  {\normalfont\emph{\IfValueTF{#1}{\IfValueTF{#2}{#2}{Proof}~of~\refth{#1},\ }\relax Step~\arabic{paragraph}~\hbox{$(\!$##1$)$\normalfont:}\ }}}%
 \setcounter{paragraph}0}
{\endproof}
\makeatletter
\NewDocumentCommand\indexsmall{m}{\@ifnextchar\egroup{#1}{\indexsmallsnd{#1}}}
\makeatother
\NewDocumentCommand\indexsmallsnd{mg}{\IfValueTF{#2}{#1{#2}}{\indexsmallthd{#1}}}
\NewDocumentCommand\indexsmallthd{mm}{\expandafter\indexsmallfth\expandafter#1#2}
\NewDocumentCommand\indexsmallfth{mg}{\IfValueTF{#2}{#1{#2}}{\indexsmallfith{#1}}}
\NewDocumentCommand\indexsmallfith{mm}{#1\ifx\indexsmall#2\hspace{-.075em}\fi#2}
\gdef\ii{\indexsmall{I}}\gdef\ij{\indexsmall{J}}
\gdef\oi{i}
\gdef\ui{\alpha}

\NewDocumentCommand\outsymbolsmall{om}{\bar{\IfValueTF{#1}{#1{#2}}{#2}}}
\NewDocumentCommand\outsymbol{om}{\overline{\IfValueTF{#1}{#1{#2}}{#2}}}
\NewDocumentCommand\unisymbol{om}{\widehat{\IfValueTF{#1}{#1{#2}}{#2}}}

\NewDocumentCommand\outtensor{d()}{\IndexSymbol(\outsymbol[#1])}
\makeatletter
\newdimen\middle@width
\newcommand*\phantomas[3][c]{\ifmmode\makebox[\widthof{$#2$}][#1]{$#3$}\else\makebox[\widthof{#2}][#1]{#3}\fi}
\NewDocumentCommand\tracefree{m}%
{\setlength\middle@width{\widthof{\ensuremath{#1}}}%
 #1\hskip-\middle@width%
 \phantomas[r]{#1}{\vphantom{\ensuremath{#1}}^{\,\!^\circ}}}%
\NewDocumentCommand\mean{om}%
{\IfValueTF{#1}%
 {\setlength\middle@width{\widthof{\ensuremath{#1#2}}}%
  \phantomas{#1#2}{{#1\hspace{-.1em}\textbf{\small$\boldsymbol\backslash$}\!}}%
  \hskip-\middle@width#1#2}%
 {\mathchoice{\mean[\displaystyle]{#2}}{\mean[\textstyle]{#2}}{\mean[\scriptstyle]{#2}}{\mean[\scriptscriptstyle]{#2}}}}
\makeatother

\NewTensor*\mass{\MakeSymbol[\outsymbol m]<{\outsymbol m}>{\vvecmass{\outsymbol m}}}
\NewTensor*\masstd{\MakeSymbol[m]<m>{\vvec m}}
\NewTensor[\MakeSymbol<\outsymbol>\vec]*\outcenterz[\!]Z

\let\oldPhi\Phi\RenewTensor\Phi<\hspace{-.05em}>\oldPhi
\NewDocumentCommand\hPhi{o}{\Phi[{\IfValueTF{#1}{#1,}\hyperbolich}]}
\NewDocumentCommand\roPhi{o}{\Phi[\hspace{.08em}{\IfValueTF{#1}{#1,}\relax\roundr\hspace{-.08em}}]}
\NewTensor[\outsymbol]\outPhi\oldPhi
\let\oldvarphi\varphi\RenewTensor*\varphi\oldvarphi
\let\oldPsi\Psi\RenewTensor*\Psi<\!>\oldPsi
\let\oldpsi\psi\RenewTensor*\psi<\!>\oldpsi

\NewTensor[\outsymbol]*\outPsi\oldPsi
\NewTensor[\outsymbol]*\outx x
\NewTensor[\outsymbol]*\outy y
\NewTensor[\outsymbol]*\outp p
\NewTensor[\outsymbol]*\outP P
\gdef\outHradius{\outsymbol\Hradius}

\NewTensor\intervalI{\normalfont I}

\gdef\hyperbolich{{\mathcal h}}
\gdef\roundr{{r}}

\gdef\houtg{\outg[\hyperbolich]}

\gdef\houtsc{\outsc[\hyperbolich]}

\NewDocumentCommand\hg{o}
{\g[\IfValueTF{#1}{#1,}\relax\hyperbolich]}
\NewDocumentCommand\rog{o}
{\g[\IfValueTF{#1}{#1,}\relax\roundr]}
\NewDocumentCommand\hsc{o}
{\sc[\IfValueTF{#1}{#1,}\relax\hyperbolich]}
\gdef\houttr{\outtr[\hyperbolich\,]}
\gdef\htr{\tr[\hyperbolich]}
\NewDocumentCommand\hH{o}
{\H[\IfValueTF{#1}{#1,}\relax\hyperbolich]}
\NewDocumentCommand\roH{o}
{\H[\IfValueTF{#1}{#1,}\relax\roundr]}
\NewDocumentCommand\hzFund{o}
{\zFund[\IfValueTF{#1}{#1,}\relax\hyperbolich]}
\NewDocumentCommand\rozFund{o}
{\zFund[\IfValueTF{#1}{#1,}\relax\roundr]}
\NewDocumentCommand\rolevi{o}
{\levi[\IfValueTF{#1}{#1,}\relax\roundr]}
\NewDocumentCommand\hzFundtrf{o}
{\zFundtrf[\IfValueTF{#1}{#1,}\relax\hyperbolich\hspace{.1em}]}
\NewDocumentCommand\rozFundtrf{o}
{\zFundtrf[\IfValueTF{#1}{#1,}\relax\roundr\hspace{.1em}]}
\gdef\houtric{\outric[\hyperbolich]}

\gdef\houtlevi{\outlevi[\hyperbolich]}

\gdef\houtlaplace{\outlaplace[\hyperbolich]}

\NewDocumentCommand\hdiv{o}
{\div[\IfValueTF{#1}{#1,}\relax\hyperbolich]}
\NewDocumentCommand\hlevi{o}
{\levi[\IfValueTF{#1}{#1,}\relax\hyperbolich]}
\NewDocumentCommand\hlaplace{o}
{\laplace[\IfValueTF{#1}{#1,}\relax\hyperbolich]}
\NewDocumentCommand\rolaplace{o}
{\laplace[\IfValueTF{#1}{#1,}\relax\roundr]}

\gdef\sinh{\sinhsnd{sinh}}\gdef\cosh{\sinhsnd{cosh}}

\NewDocumentCommand\sinhsnd{mg}{\IfValueTF{#2}{\operatorname{\text{\normalfont #1}(#2)}}{\sinhthd{#1}}}
\NewDocumentCommand\sinhthd{md()}{\IfValueTF{#2}{\sinhsnd{#1}{#2}}{\operatorname{\text{\normalfont #1}}}}

\gdef\hmug{\mug[\hyperbolich]}

\NewDocumentCommand\ronu{o}{\nu[{\IfValueTF{#1}{#1,}\relax\roundr}]}
\NewDocumentCommand\hnu{o}{\nu[{\IfValueTF{#1}{#1,}\relax\hyperbolich}]}
\NewDocumentCommand\hrnu{o}{\rnu[{\IfValueTF{#1}{#1,}\relax\hyperbolich}]}

\gdef\sphtr{\tr[\sphg*\:]}
\gdef\sphmug{\mug[\,\sphg*\hspace*{.05em}]}
\gdef\euclideane{e}

\gdef\eukoutg{\outg[\euclideane]}

\gdef\eukH{\H[\euclideane]}

\gdef\flatf{f}
\gdef\fnu{\normal[\flatf]}
\gdef\frnu{\rnu[\flatf]}
\gdef\foutg{\outg[\flatf]}

\gdef\foutric{\outric[\flatf]}
\gdef\foutsc{\outsc[\flatf]}

\gdef\fg{\g[\flatf]}

\gdef\fH{\H[\flatf]}
\gdef\fzFund{\zFund[\flatf]}
\gdef\fzFundtrf{\zFundtrf[\flatf\hspace{.08em}]}
\NewDocumentCommand\rad{d<>od<>t@}{%
 \IfBooleanTF{#4}%
 {\IfValueTF{#1}{\rad[#2]<#1>@}%
	{\IfValueTF{#3}%
	 {\IfValueTF{#2}{\vphantom|_{#3\!}^{#2\!}}{\vphantom|_{#3\!}}}%
	 {\IfValueTF{#2}{\vphantom|^{#2\!}}\relax}}%
	 |\csname rad@switch\endcsname|}%
 {\IfValueTF{#1}{\rad[#2]<#1>\relax}{|\outx[#2]<#3>|}}}
\expandafter\gdef\csname rad@switch\endcsname#1#2{#2#1}

\NewDocumentCommand\volume{od<>om}
{\IfValueTF{#3}{\volume[#3]<#2>{#4}}%
 {\IfValueTF{#1}%
  {\vphantom{\left|#1\right|}^{#1}\IfValueTF{#2}{_{#2}}\relax\hspace{-.25em}}%
	{\IfValueTF{#2}{\vphantom{\left|#1\right|}_{#2}\hspace{-.25em}}\relax}\left|#4\right|}}%
\gdef\hvolume{\volume[\hyperbolich]}

\gdef\lieD#1#2{\mathfrak L_{#1}#2}

\undef\c\NewDocumentCommand\c{sG{c}}{\IfBooleanTF{#1}{#2}{\IndexSymbol*{#2}}}

\NewDocumentCommand\Cof{G{C}d<>d()}
{\IfValueTF{#3}{\Cof{#1}<\IfValueTF{#2}{#2,}\relax{#3}>}{\CofSnd{#1\IfValueTF{#2}{[#2]}\relax}}}
\NewDocumentCommand\CofSnd{G{C}d<>o}
{\IfValueTF{#3}{\CofSnd{#1}<\IfValueTF{#2}{#2,}\relax{#3}>}{\mathop{#1\IfValueTF{#2}{(#2)}\relax}}}

\NewTensor\M[\hspace{-.05em}]<\hspace{-.05em}>{{\normalfont\Sigma}}
\NewTensor[\outsymbol]\outM[\hspace{-.025em}]{{\normalfont\textrm M}}

\NewTensor[\unisymbol]\uniM[\hspace{-.025em}]{{\textrm M}}

\NewTensor\embedding[\hspace{-.05em}]{\mathcal i}
\NewTensor[\outsymbol]\outexp{\exp*}\gdef\houtexp{\outexp[\hyperbolich]}
\global\let\graphfsymbol\xi
\global\def\graphfPDEsymbol{\underline\graphfsymbol}
\NewTensor*\graphf[\hspace{-.05em}]\graphfsymbol
\NewTensor*\graphfPDE[\hspace{-.05em}]\graphfPDEsymbol
\gdef\graphfPDE{\graphf[\laplace*]}
\let\graphft s
\NewTensor*\solutionf f

\NewTensor\graphF[\!]S
\let\graphFt S
\NewTensor\graphpushforward[\!]{\pushforward{S}\hspace{-.05em}}
\NewTensor\conformalf v

\NewTensor\sphg{\Omega\vphantom{\Omega}}
\newcommand\sphlaplace{\laplace[\sphg*]}

\NewTensor*\rnu[\!]u
\NewTensor*\lapse[\!]u
\NewTensor\rbeta[\!]\beta
\NewTensor[\MakeSymbol<>\vec]*\centerz[\!]<\hspace{-.05em}>z
\undef\metric\NewTensor[\newmathcal]\metric[\hspace{-.15em}]<\,>g
\let\g\metric
\NewTensor[\newmathcal]\metrictwo[\hspace{-.05em}]<\hspace{-.05em}>h

\NewTensor[\outsymbol[\newmathcal]]\outmetric<\:>g
\let\outg\outmetric
\NewTensor[\outsymbol[\newmathcal]]\outfg<\:>{e\hspace{-.05em}r}
\NewTensor[\outsymbol[\newmathcal]]\outh<\:>h
\NewTensor[\outsymbol[\newmathcal]]\outgk<\:>k
\NewTensor[\mathfrak]\massaspect m

\NewTensor[\unisymbol[\newmathcal]]\unimetric[\!\!]<\:>g

\NewTensor[\newmathcal]\riem R

\NewTensor[\outsymbol[\newmathcal]]\outriem R
\NewTensor[\unisymbol[\newmathcal]]\uniriem R

\NewTensor[\normalfont\textrm]\ricci{Ric}
\let\ric\ricci
\let\Ric\ricci
\NewTensor[\outsymbol[\normalfont]]\outK{\textrm{K}}
\NewTensor[\outsymbol[\normalfont\textrm]]\outricci{Ric}
\let\outric\outricci
\NewTensor[\unisymbol[\normalfont\textrm]]\uniricci{Ric}

\NewTensor[\newmathcal]\scalar[\!]S
\NewTensor[\outsymbol[\newmathcal]]\outscalar S
\let\outsc\outscalar
\NewTensor[\unisymbol[\newmathcal]]\uniscalar S

\NewTensor[\newmathcal]\gauss G
\NewTensor[\textrm]\II k

\let\zFund\k
\NewTensor[\tracefree]\IItrf[\!]{\textrm k}
\NewTensor[\tracefree]\Ktrf[\!]{\textrm K}

\let\zFundtrf\ktrf
\NewTensor[\outsymbol[\textrm]]\outII[\hspace{-.05em}]k

\let\outzFund\outk
\NewDocumentCommand\troutzFund{D<>{}O{}}{\tr<#1>[#2]\:\outzFund[#2]}
\NewDocumentCommand\outzFundnu{D<>{}O{}}{\mathop{{\outzFund[#2]_{\nu<#1>[#2]}}}}
\NewDocumentCommand\outzFundnunu{D<>{}O{}}{\mathop{{\outzFund[#2]_{\nu<#1>[#2]\nu<#1>[#2]}}}}
\NewTensor[\mathcal]\mc[\!]<\hspace{-.05em}> H
\let\H\mc
\NewTensor[\outsymbol[\mathcal]]\outmc[\!]<\hspace{-.05em}>H

\let\oldnu\nu
\NewTensor*\normal[\!]\oldnu
\let\nu\normal
\NewTensor[\outsymbol]\timenormal[\!]\vartheta

\NewTensor*\laplace[\!]\Delta
\NewTensor[\outsymbol]*\outlaplace[\!]\Delta
\NewTensor\Hess[\!]{\text{Hess}}
\NewTensor\Hesstrf[\!]{\text H\tracefree{\text{es}}\text s}
\NewTensor[\outsymbol]\outHess[\!]{\text{Hess}}
\undef\div\NewTensor\div[\hspace{-.05em}]{\text{div}}
\NewTensor[\outsymbol]\outdiv[\!]{\text{div}}
\NewTensor[\unisymbol]\unidiv[\!]{\text{div}}
\NewTensor*\tr[\hspace{0em}]{\text{tr}}
\NewTensor[\outsymbol]*\outtr{\text{tr}}
\NewTensor[\unisymbol]*\unitr[\!]{\text{tr}}
\NewTensor\levi<\MakeSymbol<>{\!}>{\MakeSymbol<\Gamma>\nabla}
\NewTensor[\outsymbol]\outlevi<\MakeSymbol<>{\!}>{\MakeSymbol<\Gamma>\nabla}
\NewTensor[\unisymbol]\unilevi<\MakeSymbol<>{\!}>{\MakeSymbol<\Gamma>\nabla}
\NewTensor*\vol{\text{vol}}

\NewTensor*\mug[\!]\mu
\NewTensor[\outsymbol]*\outmug[\!]\mu
\NewTensor[\outsymbol]\momentum[\hspace{-.05em}]\Pi
\NewTensor[\outsymbol]\einstein{{\normalfont\textrm G}}
\NewTensor[\outsymbol]\outmomden[\!]J
\NewTensor[\outsymbol]\outenden[\hspace{-.05em}]\varrho

\NewTensor[\textrm]*\jacobiext[\hspace{-.05em}]<\hspace{-.05em}> L
\gdef\hjacobiext{\jacobiext[\hyperbolich]}

\NewDocumentCommand\trzd{smm}{{#2}\odot{#3}}
\NewDocumentCommand\htrzd{smm}{{#2}{}\hspace{.15em}^\hyperbolich\hspace{-.32em}\odot{#3}}
\NewDocumentCommand\rotrzd{smm}{{#2}{}\hspace{.15em}^\roundr\hspace{-.32em}\odot{#3}}
\NewDocumentCommand\sphtrzd{smm}{{#2}{}\hspace{.15em}^{\sphg*}\hspace{-.32em}\odot{#3}}
\NewDocumentCommand\trtr{smm}
{\IfBooleanTF{#1}{\tr(\trzd{#2}{#3})}%
 {\ifthenelse{\equal{#2}{#3}}{\left|#2\right|_{\hspace{-.05em}\g*}^2}%
  {\trtr*{#2}{#3}}}}
\NewDocumentCommand\htrtr{smm}
{\IfBooleanTF{#1}{\htr(\htrzd{#2}{#3})}%
 {\ifthenelse{\equal{#2}{#3}}{\left|#2\right|_{\hspace{-.05em}\hg*}^2}%
  {\trtr*{#2}{#3}}}}
\NewDocumentCommand\sphtrtr{smm}
{\sphtr(\sphtrzd{#2}{#3})}

\NewDocumentCommand\outc{G{c}}{\IndexSymbol(\outsymbol)*{#1}}\let\oc\outc

\NewTensor\decay\varpi
\NewTensor\scdecay\upsilon
\NewTensor\outve\ve
\NewDocumentCommand\outck{d<>o}{\outc<#1>[#2]_{\outzFund*}\vphantom{\outc_{\outzFund*}}}
\gdef\outtr{\IndexSymbol(\outsymbol)[\!]{\text{tr}}}
\NewDocumentCommand\outtrzd{D<>{}O{}mm}
{\mathop{{#3\IndexSymbol(\outsymbol)[#2]<#1>\odot{#4}}}}
\NewDocumentCommand\outtrtr{D<>{}O{}mm}
{\mathop{{\ifthenelse{\equal{#3}{#4}}%
 {{\left|#3\right|_{\outg<#1>[#2]}^2}}%
 {\mathop{{\outtr<#1>[#2](\outtrzd<#1>[#2]{#3}{#4})}}}}}}%
\NewDocumentCommand\houttrzd{mm}
{\mathop{#1\mathop{\:^{\hyperbolich}\hspace{-.05em}\outsymbol\odot}#2}}
\NewDocumentCommand\houttrtr{mm}
{\ifthenelse{\equal{#1}{#2}}%
 {\left|#1\right|_{\houtg*}^2}%
 {\mathop{\houttr(\houttrzd{#1}{#2})}}}%
\NewTensor*\eflap[\!]f
\NewTensor*\ewlap[\!]\lambda
\NewTensor*\efjac[\!]{\mathcal f}
\NewTensor*\ewjac[\hspace{-.05em}]<\!>\kappa
\NewTensor*\funcg[\!]g
\NewTensor*\funch[\!]h
\gdef\spheflap{\eflap[\sphg]}
\NewDocumentCommand\hboost{o}{\boost[{\IfValueTF{#1}{#1,}\relax\hyperbolich}]}
\NewDocumentCommand\hdeform{o}{\deform[{\IfValueTF{#1}{#1,}\relax\hyperbolich}]}
\NewDocumentCommand\roboost{o}{\boost[{\IfValueTF{#1}{#1,}\relax\roundr
}]}
\NewDocumentCommand\rodeform{o}{\deform[{\IfValueTF{#1}{#1,}\relax\roundr}]}
\NewDocumentCommand\functionpart{mod<>d()}
{\IfValueTF{#4}{\functionpartsnd{\hspace{-.05em}}{#1}[#2]<#3>{(#4)}}{\functionpartsnd\!{#1}[#2]<#3>}}
\NewDocumentCommand\functionpartsnd{mmod<>m}
{{#5}^{\IfValueTF{#3}{\IfValueTF{#4}\relax{#1}}\relax#2[#3]<#4>}}
\gdef\boost{\functionpart\boostpart}
\gdef\deform{\functionpart\deformpart}
\NewTensor\boostpart[\hspace{-.05em}]<\hspace{-.08em}>{\text{\normalfont b}}
\NewTensor\deformpart[\hspace{-.05em}]<\hspace{-.08em}>{\text{\normalfont d}}
\makeatletter
\gdef\boostLp^#1(#2){\Lp^{#1}(#2)\hspace{-.05em}\boost{\vphantom{#2}}}
\gdef\deformLp^#1(#2){\Lp^{#1}(#2)\hspace{-.05em}\deform{\vphantom{#2}}}
\gdef\hdist{{}^{\hyperbolich\hspace{-.05em}}\text{\normalfont dist}}

\NewDocumentCommand\normof{d<>od<>m}
{\IfValueTF{#1}%
 {\IfValueTF{#2}%
  {\IfValueTF{#3}{\normofleft<#1#3>[#2]{#4}}{\normofleft<#1>[#2]{#4}}}%
  {\IfValueTF{#3}{\normofleft<#1#3>{#4}}{\normofleft<#1>{#4}}}}%
 {\IfValueTF{#2}%
  {\IfValueTF{#3}{\normofleft<#3>[#2]{#4}}{\normofleft[#2]{#4}}}%
  {\IfValueTF{#3}{\normofleft<#3>{#4}}\relax}}%
 \lvert#4\rvert}
\NewDocumentCommand\normofleft{d<>om}%
{\vphantom{\lvert#3\rvert}\IfValueTF{#1}{_{#1}}\relax\IfValueTF{#2}{^{#2}}\relax}

\let\Hradius\sigma
\let\Aradius R

\def\rradius{{\underline r}}

\NewTensor\mHaw{{\normalfont m_{\text{H}}}}
\NewTensor\HmHaw{{\normalfont m_{\text{H}}^{\hspace{-.05em}\text h}}}
\NewTensor\varphimHaw{{\normalfont m_{\varphi}^{\hspace{-.05em}\text h}}}
\NewTensor\HmInt{{\normalfont m_{\text{I}}^{\text h}}}

\NewTensor[\MakeSymbol<\outsymbol>\vec]\impuls P

\undef\d
\NewDocumentCommand\d{s}{\IfBooleanTF{#1}\relax{\mathop{}\!}\mathrm d}

\newcommand\pullback[1]{{#1}^*}
\newcommand\pushforward[1]{{#1}_*}

\DeclareMathOperator\id{id}

\DeclareMathOperator\graph{graph}%

\newcommand\R{\mathds{R}}

\newcommand\X{\mathfrak{X}}
\NewDocumentCommand\Lp{t^}{\IfBooleanTF{#1}\Lpsnd{{\normalfont\textrm L}}}
\gdef\Lpsnd#1{\mathchoice%
 {{\Lp\relax}^{\!#1}}
 {{\Lp\relax}^{\!#1}}
 {{\Lp\relax}^{^{\!#1}}}
 {{\Lp\relax}^{^{\!#1}}}}
\newcommand\Wkp{{\normalfont\textrm W}}
\newcommand\Hk{{\normalfont\textrm H}}
\newcommand\Ck{\mathcal C}

\NewDocumentCommand\sphere{t^}{\mathds S\IfBooleanTF{#1}^{^2}}
\gdef\hsphere{\vphantom{S}^\hyperbolich\sphere}
\gdef\euksphere{\vphantom{S}^{\euclideane\hspace{-.05em}}\sphere}

\NewDocumentCommand\hyperbolicspace{t^}{\IfBooleanTF{#1}{\mathbbm H^}{\hyperbolicspace^3}}%
\newcommand\hyperbolicisometry{\boldsymbol\gamma}
\NewTensor*\hiso<\!>\hyperbolicisometry
\NewDocumentCommand\HISO{t_}{\boldsymbol\Gamma\IfBooleanTF{#1}{\hspace*{-.2em}_}\relax}
\let\ve\varepsilon

\usepackage{urwchancal}
\DeclareMathAlphabet{\mathcal}{OT1}{pzc}{m}{n}
\let\newmathcal\mathcal

\makeatletter
\def\my@overarrow@#1#2#3{\vbox{\ialign{##\crcr#1#2\crcr\noalign{\kern-2.75\p@\nointerlineskip}$\m@th\hfil#2#3\hfil$\crcr}}}
\def\my@overarrowt@#1#2#3{\vbox{\ialign{##\crcr#1#2\crcr\noalign{\kern-1.5\p@\nointerlineskip}$\m@th\hfil#2#3\hfil$\crcr}}}
\gdef\vvecmass{%
 \mathpalette{\my@overarrow@\vectfillb@}}
\gdef\vvec{%
 \mathpalette{\my@overarrowt@\vectfillb@}}
\newcommand{\vecbar}{%
  \scalebox{1}{$\relbar$}}
\def\vectfillb@{\arrowfill@\vecbar\vecbar{\raisebox{-3.65pt}[\p@][\p@]{$\mathord\mathchar"017E$}}}
\makeatother

\makeatletter
\gdef\extractspeciallabel#1{%
 \expandafter\expandafter\expandafter%
 \extractspeciallabel@fst\getrefnumber{#1}}
\NewDocumentCommand\extractspeciallabel@fst{g}%
{\IfValueTF{#1}{\extractspeciallabel@fst#1}\extractspeciallabel@snd}
\gdef\extractspeciallabel@snd#1{%
 \ifx#1\hbox\expandafter\extractspeciallabel@thd\else#1\fi}
\gdef\extractspeciallabel@thd#1{\extractspeciallabel@thdb#1}
\gdef\extractspeciallabel@thdb#1{%
 \ifx#1\normalfont\expandafter\extractspeciallabel@fth\else A\string#1A\fi}
\gdef\extractspeciallabel@fth(#1){#1}

\makeatother

\gdef\emptyarg{{\:}{\cdot}{\:}}
\gdef\masserr{\masserrsnd{m_e}}\gdef\Masserr{\masserrsnd{M_e}}
\NewDocumentCommand\masserrsnd{mssd()}
{\IfBooleanTF{#2}%
 {\IfBooleanTF{#3}%
  {#1}%
  {\IfValueTF{#4}{\masserrsnd{#1}{\Hradius}(#4)}{\masserrsnd{#1}{\Hradius}}}}%
 {\IfValueTF{#4}{\masserrth{#1}{#4}}{\masserrth{#1}}}}
\gdef\masserrth#1#2{#1(#2)}

\NewDocumentCommand\kuerzel{mm}%
{\NewDocumentCommand#1{s}{\IfBooleanTF{##1}{#2}{\hbox{#2}\xspace}}}
\kuerzel\ie{i.e.}
\kuerzel\eg{e.g.}
\kuerzel\etc{etc.}

\let\oldexp\exp\undef\exp
\NewDocumentCommand\exp{s}{\IfBooleanTF{#1}\oldexp\expsnd}
\NewDocumentCommand\expsnd{d()}{\IfNoValueTF{#1}\expthd{e^{#1}}}
\NewDocumentCommand\expthd{t-m}{\IfBooleanTF{#1}{\exp({-}#2)}{\exp(#2)}}

\gdef\Ckah{\CWkah{\Ck{hyperbolic}}}

\gdef\CkaE{\CWkah{\Ck{Euclidean}}}

\gdef\Wkpro{\CWkah{\Wkp{round}}}
\NewDocumentCommand\CWkah{mO{}t^t_}%
{\IfBooleanTF{#3}{\CWkahtw{#1}[#2]^}%
 {\IfBooleanTF{#4}{\CWkahtw{#1}[#2]_}{\CWkahth#1[#2]}}}
\NewDocumentCommand\CWkahtw{momm}{\CWkah{#1}[#2#3{#4}]}
\NewDocumentCommand\CWkahth{mmod()}
{$#1\IfValueTF{#3}{#3}\relax%
 \IfValueTF{#4}{(#4)}\relax$-as\-ymp\-toti\-cal\-ly #2\xspace}
\gdef\CMCfol{CMC-fo\-li\-a\-tion\xspace}

\makeatletter
\tikzset{
  use path for main/.code={%
    \tikz@addmode{%
      \expandafter\pgfsyssoftpath@setcurrentpath\csname tikz@intersect@path@name@#1\endcsname
    }%
  },
  use path for actions/.code={%
    \expandafter\def\expandafter\tikz@preactions\expandafter{\tikz@preactions\expandafter\let\expandafter\tikz@actions@path\csname tikz@intersect@path@name@#1\endcsname}%
  },
  use path/.style={%
    use path for main=#1,
    use path for actions=#1,
  }
}

\gdef\pagerefs{\pagerefs@prepare\relax}
\NewDocumentCommand\pagerefs@prepare{omm}{%
 \edef\pagrefs@current{\getpagerefnumber{#3}}%
 \pagerefs@start\pagrefs@current{#1}\BooleanTrue}
\NewDocumentCommand\pagerefs@start{mmmg}{\IfValueTF{#4}{\pagerefs@parse{#1}{#2}{#3}{#4}}{\pagerefs@output{#1}{#2}{#3}}}
\gdef\pagerefs@parse#1#2#3#4{%
 \let\pagerefs@last#1%
 \edef\pagerefs@current{\getpagerefnumber{#4}}%
 \ifx\pagerefs@last\pagerefs@current\expandafter\@firstoftwo\else\expandafter\@secondoftwo\fi%
	{\pagerefs@start\pagerefs@last{#2}{#3}}
	{\IfValueTF{#2}%
	 {\edef\totalstring{#2, \pagerefs@last}\let\issecond\BooleanFalse}%
	 {\edef\totalstring{\pagerefs@last}\let\issecond\BooleanTrue}%
	 \pagerefs@start\pagerefs@current\totalstring\issecond}}
\gdef\pagerefs@output#1#2#3{%
 \IfValueTF{#1}{\IfValueTF{#2}{\IfBooleanTF{#3}{pages~#2 and~#1}{pages~#2, and~#1}}{page~#1}}\relax\xspace}
\makeatother

\newsavebox\FigureBox
\NewDocumentEnvironment{Figure}{d<>omom}
{\begin{lrbox}{\FigureBox}%
	\begin{minipage}{.975\textwidth}\vspace*{.5em}%
	  \IfValueTF{#2}{\tikzsetnextfilename{image_#2}}\relax%
		\centering#3\medskip\par\noindent%
		\begin{minipage}{.9\textwidth}\footnotesize}
{			\vspace*{.5em}%
		\end{minipage}%
	\end{minipage}%
 \end{lrbox}%
 \begin{figure}[!hbt]\centering\IfValueTF{#1}{\vspace*{#1}}\relax%
  \framebox{\usebox{\FigureBox}}%
	\vspace*{-.7em}\par\noindent%
	\IfValueTF{#4}{\caption[#4]{#5}}{\caption{#5}}%
	\IfValueTF{#2}{\label{#2}}\relax%
	\vspace*{-1.25em}\IfValueTF{#1}{\vspace*{#1}}\relax%
 \end{figure}}

\title[Geom.\ charac.\ of asymp.\ hyperb.\ mflds by ex.\ of a suit.\ CMC-foliation]{Geometric characterizations of asymptotically hyperbolic Riemannian \texorpdfstring{$\boldsymbol 3$}3-manifolds by the existence of a suitable CMC-foliation}
\author[Christopher Nerz]{Christopher Nerz}
\address{Department of mathematics\\Royal institut of Technology KTH\\Stockholm\\Sweden}
\email{ncroman@kth.se}
\date{\today}
\hypersetup{
	pdftitle		= {Geometric characterizations of asymptotically hyperbolic Riemannian 3-manifolds by the existence of a suitable CMC-foliation},
	pdfauthor		= {Christopher Nerz},
	pdfsubject	= {A Riemannian manifold is asymptotically hyperbolic if and only if it possesses a locally unique CMC-foliation},
	pdfkeywords	= {CMC-foliation, constant mean curvature foliation, CMC, CMC-surface, asymptotically hyperbolic,	geometric characterization, construct coordinates}}
	
\begin{document}
\begin{abstract}
In 1996, Huisken-Yau proved that every three-dimensional Riemannian manifold can be uniquely foliated near infinity by stable closed surfaces of constant mean curvature (CMC) if it is asymptotically equal to the (spatial) Schwarzschild solution. Using their method, Rigger proved the same theorem for Riemannian manifolds being asymptotically equal to the (spatial) (Schwarzschild-)Anti-de Sitter solution. This was generalized to asymptotically hyperbolic manifolds by Neves-Tian, Chodosh, and the author at a later stage. In this work, we prove the reverse implication as the author already did in the Euclidean setting, \ie any three-dimensional Riemannian manifold is asymptotically hyperbolic if it (and only if) possesses a CMC-cover satisfying certain geometric curvature estimates, a uniqueness property, and each surface has controlled instability. As toy application of these geometric characterizations of asymptotically Euclidean and hyperbolic manifolds, we present a method for replacing an asymptotically hyperbolic by an asymptotically Euclidean end and apply this method to prove that the Hawking mass of the CMC-surfaces is bounded by their limit being the total mass of the asymptotically hyperbolic manifold, where equality holds only for the $t{=}0$-slice of the (Schwarzschild-)Anti-de Sitter spacetime.
\end{abstract}
\maketitle
\let\sc\scalar%

\section{Introduction}
In~1996, Huisken-Yau proved that manifolds which are asymptotic to the spatial Schwarzschild metric with positive mass possesses a foliation by stable constant mean curvature (CMC) hypersurfaces, \cite{huisken_yau_foliation}. They used this foliation as a definition for the center of mass of the manifold and also gave a coordinated version of this center. Since then, this foliation proved to be a suitable tool for the study of asymptotically Euclidean (\ie asymptotically flat Riemannian) manifolds and several generalizations of Huisken-Yau's result were made, \eg by Metzger, Huang, Eich\-mair-Metzger, and the author, \cite{metzger2007foliations,Huang__Foliations_by_Stable_Spheres_with_Constant_Mean_Curvature,metzger_eichmair_2012_unique,nerz2015CMCfoliation}. In~2004, Rigger used Huisken-Yau's method---the mean curvature flow---to prove the existence and uniqueness of such a foliation for manifolds asymptotic to the $t{=}0$-slice of the (Schwarzschild-)Anti-de Sitter spacetime, \cite{rigger2004foliation}. This result was generalized using other methods to more general asymptotically hyperbolic manifolds by Neves-Tian, Chodosh, and the author, \cite{NevesTianExistenceCMC_I,NevesTianExistenceCMC_II,chodosh2014large,nerz2016HBCMCExistence}.\smallskip

In~\cite{nerz2015GeometricCharac}, the author proved that the existence of a CMC-foliation is not only an implication of asymptotic flatness but a characterization of it, \ie an arbitrary Riemannian $3$-manifold possesses a \lq suitable\rq\ CMC-foliation if and only if it is asymptotically Euclidean. In this article, we prove the equivalent theorem for the hyperbolic setting or more precisely the missing implication: if a Riemannian $3$-manifold possesses a \lq suitable\rq\ CMC-foliation, then it is asymptotically hyperbolic.

As a toy application of these characterizations of asymptotically Euclidean and hyperbolic manifolds, we show that we can replace any asymptotically hyperbolic end by an asymptotically Euclidean one. Using this construction, \cite{bray1997mon,huisken2001inverse} prove that if $\outsc\ge{-}6$, then the (hyperbolic) Hawking mass is monotone along the leaves of the CMC-foliation and bounded from above by the total mass of the surrounding (asymptotically hyperbolic) manifold, where equality holds if and only if the surrounding manifold is a compact perturbation of the $[t{=}0]$-slice of the \SAdS spacetime.\medskip

\subsection{The main results}
\begin{maintheorem}[CMC-characterization~of asymptotically\ hyperbolic manifolds]\labelth{MainTheorem}
Let $\decay\in\interval{\frac52}3$ and $\scdecay\ge\decay+\frac12$ be constants and $(\outM,\outg*)$ be a Riemannian manifold. $(\outM,\outg*)$ is \Ckah^2_{\decay,\scdecay} with\footnote{Note that the mass vector~$\mass$ itself depends on the asymptotically hyperbolic coordinates system, but the total mass $\mass*:={-}\vert\mass\vert_{\R^{3,1}}$ does not, and therefore it is a coordinate independent property whether the mass vector is timelike or not, \cite{chrusciel2003mass}.} timelike mass vector~$\mass$ if and only if there exists a family $\{\M<\Hradius>\}_{\Hradius>\Hradius_0}$ of hypersurfaces of $\outM$ such that
\begin{enumerate}[nosep,label=\hbox{\normalfont(\alph{*})}]
\item $\{\M<\Hradius>\}_{\Hradius>\Hradius_0}$ is a \Wkpro^{2,\infty}_{\decay,\scdecay} CMC-cover;\vspace{-.1em}\label{main_theorem_round}
\item $\{\M<\Hradius>\}_{\Hradius>\Hradius_1}$ covers $\outM$ outside a compact set $\outsymbol K=\outsymbol K(\Hradius_1)$ for every $\Hradius_1>\Hradius_0$;\label{main_theorem_cover}
\item $\{\M<\Hradius>\}_{\Hradius>\Hradius_0}$ is locally unique;\label{main_theorem_uniqueness}
\item $\{\M<\Hradius>\}_{\Hradius>\Hradius_0}$ has uniformly timelike and bounded Ricci-mass.\label{main_theorem_mass}
\end{enumerate}
Furthermore, the (coordinate-independent) hyperbolic Hawking mass of $\M<\Hradius>$ converges to the total mass $\mass*:=\vert\mass\vert_{\R^{3,1}}$ of $(\outM,\outg*)$ as $\Hradius\to\infty$.
\end{maintheorem}
The definitions used here are given as Definitions~\ref{Ck_asymptotically_hb}, \ref{ControlledInstability}, \ref{Round_spheres}, and~\ref{Round_covers} on \hbox{pages~\pageref{Ck_asymptotically_hb}--\pageref{Round_covers}}. The existence of such a round cover for \Ckah^2_{\decay,\scdecay} manifold with $\decay\in\interval{\frac52}3$ and $\scdecay>3$ was proven by the author in \cite{nerz2016HBCMCExistence}. In this article, we prove the reverse implication, \ie that the existence of a suitable \CMCfol implies the existence of a \Ckah^2_{\decay,\scdecay} chart.\smallskip

Furthermore, we prove that the Hawking mass is monotone increasing along the foliation and bounded by the total mass, where equality only holds for Schwarzschild-Anti de Sitter.
\begin{maintheorem}\labelth{HawkingMassSmallerTotalMass}
Let $\decay\in\interval{\frac52}3$ be a constant and $(\outM,\outg*)$ be a $\Ck^2_{\decay}$-asymptotically hyperbolic manifold with $0\le\exp(\rad)(\outsc+6)\in\Lp^1(\outM)$, where $\outx$ is any $\Ck^2_{\decay}$-asymptotically hyperbolic coordinate system\footnote{More precisely $\exp(\rad)(\outsc+6)$ has only to be integrable, where $\rad$ is well-defined.}. The function $\HmHaw:\interval{\Hradius_0}\infty\to\R:\Hradius\mapsto\HmHaw(\M<\Hradius>)$ mapping the mean curvature radius to the hyperbolic Hawking mass of the corresponding CMC-leaf $\M<\Hradius>$ with mean curvature $\H<\Hradius>\equiv{-}2\,\frac{\cosh(\Hradius)}{\sinh(\Hradius)}$ is a monotone non-decreasing function converging to $\big|\vert\mass\vert_{\R^{3,1}}\big|$ as $\Hradius\to\infty$, where the later is the total mass\footnote{with respect to any asymptotically hyperbolic coordinate system} of $(\outM,\outg*)$.

Here, equality holds for some large mean curvature radius if and only if $\outM$ is (outside of the corresponding CMC-leaf) isometric to the standard $[t{=}0]$-timeslice (outside of a ball) of the Schwarzschild-Anti-de Sitter spacetime.
\end{maintheorem}
\begin{remark}
In contrast to the Euclidean setting, \refth{HawkingMassSmallerTotalMass} is not necessarily true if we replace $\{\M<\Hradius>\}_\Hradius$ with some other foliation of $\outM$, \ie there are smooth (arbitrarily round) hypersurfaces $\M$ within $\outM$ having larger Hawking mass than the total mass of $\outM$. This can be seen as a straightforward calculation proves that the mass vector $\mass(\outy)$ for any non-balanced asymptotically hyperbolic coordinate system $\outy$ of $\outM$ satisfies
\[ \vert\vphantom{\big|}\HmHaw(\vphantom{\big|}\lbrace\rad@\outy=\rradius\rbrace)\vert^2 \xrightarrow{\rradius\to\infty} \vert\vphantom{\big|}\mass^0(\outy)\vert^2 = {-}\vert\vphantom{\big|}\mass(\outy)\vert_{\R^{3,1}}^2 + \vert(\mass^\oi(\outy))_{\oi=1}^3\vert_{\R^3}^2 > {-}\vert\vphantom{\big|}\mass(\outy)\vert_{\R^{3,1}}^2. \]
Here, we used Neves-Tian's definition of a balanced coordinate system, see \cite{NevesTianExistenceCMC_II}: a coordinate system $\outx$ of an asymptotically hyperbolic manifold is called \emph{balanced} if and only if $\mass(\outx)=(\mass*(\outx),0,0,0)$.
\end{remark}
\refth{HawkingMassSmallerTotalMass} is actually a direct corollary of the monotony of the Hawking mass along the CMC-foliation, \cite{bray1997mon}, and under the inverse mean curvature flow, \cite{huisken2001inverse}, combined with the following corollary of \refth{MainTheorem} and \cite{nerz2015GeometricCharac}: \smallskip

\begin{corollary}[Replacing the asymptotic end]\labelth{ReplacingTheEnd}
Let $\decay\in\interval{\frac52}3$ and $\scdecay\in\interval*{\decay+\frac12}*{2\decay}$ be arbitrary constants.\footnote{$\scdecay$ may be even be in $\interval*\decay*{2\decay}$ if $\outsc\ge{-}6$.}\ For every \Ckah^2_{\decay,\scdecay} manifold $(\outM,\outg)$, there exists a metric $\foutg$ on $\outM$ with the following properties
\begin{itemize}[nosep]
\item $\foutg$ is \CkaE^2_{\decay-2,\scdecay};
\item the CMC-foliations with respect to $\outg$ and $\foutg$ are identical;
\item $\outg|_{\textrm T\M<\Hradius>^2}=\foutg|_{\textrm T\M<\Hradius>^2}$ for every $\Hradius>\Hradius_0$ and $\outg=\foutg$ in $\outM\setminus\bigcup_\Hradius\M<\Hradius>$;
\item $\fnu<\Hradius>=\cosh(\Hradius)\nu[\outg\:]<\Hradius>$ for every $\Hradius>\Hradius_1$ and the outer unit normals $\fnu<\Hradius>$ and $\nu<\Hradius>$ of $\M<\Hradius>$ with respect to $\foutg$ and $\outg$, respectively;
\item $\HmHaw(\M<\Hradius>)=\mHaw(\M<\Hradius>)$ for every $\Hradius>\Hradius_1$;
\item if $\outsc[\outg]\ge{-}6$, then $\foutsc|_{\M<\Hradius>}\ge0$ for every $\Hradius>\Hradius_1$.
\end{itemize}
Here, $\Hradius_0$ denotes the infimum of the mean curvature radius of the canonical CMC-foliation and $\Hradius_1>\Hradius_0$ is arbitrary.
\end{corollary}
Note that the behavior of $\foutsc$ is also in $\bigcup_{\Hradius>\Hradius_0}^{\Hradius_1}\M<\Hradius>$ well-controlled---\ie in the region between the one where $[\outg\,{=}\,\foutg]$ and the one where $[\foutsc\,{\ge}\,0]$---, see the construction~\eqref{ReplacingTheEnd_construction} on page~\pageref{ReplacingTheEnd_construction}.

\vfill\par\textbf{Acknowledgment.}
The author thanks Stephen McCormick and Katharina Radermacher for helpful discussions on applications of \refth{HawkingMassSmallerTotalMass} and the \emph{Alexander von Humboldt Foundation} for ongoing financial support via the \emph{Feodor Lynen scholarship}.

\section{Structure of the paper}
In Section~\ref{Assumptions_and_notation}, we give the basic definitions and explain the notations used in this article. In particular, we define what \Wkpro^{2,p}_{\decay,\scdecay} spheres and covers are. We prove in Section~\ref{Regularity_spheres} that \Wkpro^{2,p}_{\decay,\scdecay} spheres satisfy strict estimates on their extrinsic curvature and other regularity properties of these objects. Then, we use them in Section~\ref{Regularity_covers} to conclude strict estimates on \Wkpro^{2,p}_{\decay,\scdecay} covers and to show that such a cover always has a well-defined mass. In Section~\ref{ProofOfMainTheorem}, we then explain and present the proof of \refth{MainTheorem}. Finally, we prove \refth{HawkingMassSmallerTotalMass} and \refth{ReplacingTheEnd} in the last Section~\ref{ReplacingAsymptoticEnd}.

\section{Assumptions and notation}\label{Assumptions_and_notation}
\begin{notation}[Notations for the most important tensors]
In order to study foliations (near infinity) of three-dimensional Riemannian manifolds by two-dimensional spheres, we have to deal with different manifolds (of different or the same dimension) and different metrics on these manifolds, simultaneously. To distinguish between them, all three-dimensional quantities like the surrounding manifold $(\outM,\outg*)$, its Ricci and scalar curvature $\outric$ and $\outsc$ and all other derived quantities carry a bar, while all two-dimensional quantities like the CMC leaf $(\M,\g*)$, its second fundamental form $\zFund*$, the trace-free part of its second fundamental form $\zFundtrf:=\zFund*-\frac12\,(\tr\zFund*)\g*$, its Ricci, scalar, and mean curvature $\ric$, $\sc$, and $\H:=\tr\zFund$, its outer unit normal $\nu$, and all other derived quantities do not.\end{notation}

As explained, we interpret the second fundamental form and the normal vector of a hypersurface as quantities of the surface (and thus as two-dimensional). For example, if $\M<\Hradius>$ is a hypersurface in $\outM$, then $\nu<\Hradius>$ denotes its normal (and \emph{not} $\IndexSymbol{\outsymbol\nu}<\Hradius>$). The same is true for the \lq lapse function\rq\ and the \lq shift vector\rq\ of a hypersurfaces arising as a leaf of a given deformation or foliation. Furthermore, we stress that the sign convention used for the second fundamental form, \ie $\zFund(X,Y)=\outg(\outlevi*_{\!X}Y,\nu)$ for $X,Y\in\X(\M)$, results in the \emph{negative} mean curvature $\eukH(\sphere_\rradius)\equiv{-}\frac2\rradius$ for the two-dimensional Euclidean sphere of radius $\rradius$.

\begin{notation}[Left indexes and accents of tensors]
If different two-dimensional manifolds or metrics are involved, then the lower left index denotes the mean curvature index $\Hradius$ of the current leaf $\M<\Hradius>$, \ie the leaf with mean curvature $\H<\Hradius>\equiv{-}2\,\frac{\cosh(\Hradius)}{\sinh(\Hradius)}$, or the radius $\rradius$ of a coordinate sphere $\sphere_\rradius(0)$. Quantities carry the upper left index $\hyperbolich$, $\euclideane$, and $\sphg*$ if they are calculated with respect to the hyperbolic metric $\houtg*$, the Euclidean metric $\eukoutg$, and the standard metric $\sphg<\Hradius>$ of the Euclidean sphere $\sphere_\Hradius(0)$, correspondingly. Furthermore, we use the upper left index $\roundr$ for quantities calculated with respect to the hyperbolic metric $\houtg$ along a specific (\lq round\rq) embedding of the CMC-leafs to the hyperbolic space, see Section~\ref{ProofOfMainTheorem}. We abuse notation and suppress the left indexes, whenever it is clear from the context which manifold and metric we refer to.
\end{notation}

\begin{notation}[Indexes]
We use upper case latin indices $\ii$ and $\ij$ for the two-dimensional range $\lbrace2,3\rbrace$, the lower case latin index $\oi$ for the three-dimensional range $\lbrace 1,2,3\rbrace$, and the greek index $\ui$ for the four-dimensional range $\{0,1,2,3\}$. The Einstein summation convention is used accordingly.\pagebreak[3]\smallskip
\end{notation}

As there are different definitions of \lq asymptotically hyperbolic\rq\ in the literature, we now give the one used in this paper. 
\begin{definition}[\texorpdfstring{$\Ck^2_{\decay,\scdecay}$}{C2}-asymp\-to\-tic\-ally hyperbolic Riemannian manifolds]\label{Ck_asymptotically_hb}
Let $\decay,\scdecay>0$ be constants. A triple $(\outM,\outg,\outx)$ is called (three-dimensional) \emph{$\Ck^2_{\decay,\scdecay}$-asymp\-to\-tic\-ally hyperbolic} Riemannian manifold if $(\outM,\outg)$ is a three-dimensional smooth Riemannian manifold and $\outx:\outM\setminus\overline L\to\R^3$ is a smooth chart of $\outM$ outside a compact set $\overline L\subseteq\outM$ such that there exists a constant $\oc\ge0$ with
\begin{equation*} 
 \vert\outg-\houtg\vert_{\houtg*} \hspace{-.05em}+ \vert\houtlevi*(\outg*-\houtg*)\vert_{\houtg*} \hspace{-.05em}+ \vert\outric-\houtric \vert_{\houtg*} \le \oc\,\exp(-\decay\,\rad), \quad
 \vert \outsc-\houtsc\vert \le \oc\,\exp({-}\scdecay\,\rad),\hspace{-.1em}
 \labeleq{Decay_assumptions_g}\end{equation*}
where $\houtg*=\d r^2+\sinh(\rad)^2\sphg*$ and $\sphg*$ denote the hyperbolic metric and the standard metric of the Euclidean unit sphere $\sphere$, respectively. Here, these quantities are identified with their push-forward along $\outx$. Finally, $(\outM,\outg*,\outx)$ is called $\Ck^2_{\decay}$-asymptotically hyperbolic if it is $\Ck^2_{\decay,\decay}$-asymptotically hyperbolic.
\end{definition}
We often abuse notation and suppress the chart $\outx$.

\begin{remark}[Boundedness of the scalar curvature]\labelth{BoundednessScalarCurvature_coordinates}
For everything, we do in this article the assumption on the scalar curvature can also be reduced to 
\[ {-}\oc\,\exp({-}\scdecay\rad) \le\outsc-\houtsc, \qquad
	\exp\rad\,(\outsc+\houtsc)\in\Lp^1(\outM), \]
see also Remarks~\ref{BoundednessScalarCurvature_spheres}, \ref{BoundednessScalarCurvature_stability}, \ref{BoundednessScalarCurvature_rnu}, and \ref{BoundednessScalarCurvature_mass}. However, we then have to assume that the mass of $(\outM,\outg*)$ is future pointing timelike instead of only assuming that it is timelike, see \refth{MainTheorem}.
\end{remark}

\begin{definition}[Controlled instability]\labelth{ControlledInstability}
Let $\M\hookrightarrow(\outM,\outg*)$ be a hypersurface within a three-dimensional Riemannian manifold and let $\alpha\in\R$ be a constant. If $\M$ has constant mean curvature, then it is called of \emph{$\alpha$-controlled instability} if the smallest eigenvalue of the (negative) stability operator ${-}\jacobiext$ is greater than (or equal to) $\alpha$, \ie
\[ \int\vert\levi f\vert^2_{\g*}\d\mug \ge \int(\trtr\zFund\zFund + \outric(\nu,\nu) + \alpha)(f-\mean f)^2\d\mug\qquad\forall\,f\in\Hk^2(\M), \]
where $\mean f:=\fint f\d\mug:=\volume{\M}^{{-}1}\int f\d\mug$ denotes the mean value of any function~$f\in\Hk^2(\M)$. The surface $\M$ is called \emph{stable} and \emph{strictly stable} if it has $\alpha$-controlled instability for $\alpha=0$ and $\alpha>0$, respectively.
\end{definition}

\begin{definition}[Hyperbolic hawking mass{, \eg \cite{wang2001mass}}]
If $(\M,\g*)$ is a hypersurface within a three-dimensional Riemannian manifold, then
\begin{equation*}
 \HmHaw(\M) := (\frac{\volume{\M}}{16\pi})^{\frac12}(1-\frac1{16\pi}\int_{\M}(\vphantom{\big|}\H^2-4)\d\mug)
\end{equation*}
is called \emph{hyperbolic Hawking mass of $\M$}.
\end{definition}

\begin{definition}[Round spheres]\labelth{Round_spheres}
Let $(\outM,\outg*)$ be a three-dimensional Riemannian manifold and let $\decay>0$, $\eta\in\interval0*4$, $p\in\interval*1*\infty$, and $\c\ge0$ be arbitrary constants.

A hypersurface $\M\hookrightarrow(\outM,\outg*)$ with constant mean curvature is called \emph{\Wkpro^{2,p}_{\decay,\scdecay}(\c,\eta) sphere} \emph{of mean curvature radius $\Hradius$} if
\begin{enumerate}[nosep,label={\hbox{\normalfont(RS-\arabic{*})}}]
 \item $\M$ is diffeomorphic to the Euclidean sphere; \label{Round_spheres_sphere}
 \item $\M$ has constant mean curvature with mean curvature radius~$\Hradius$, \ie $\displaystyle\H\equiv{-}\textstyle2\,\frac{\cosh(\Hradius)}{\sinh(\Hradius)}$; \label{Round_spheres_CMC}
 \item $\displaystyle\left\Vert\lvert\outric+2\outg\rvert_{\outg}\right\Vert_{\Lp^p(\M)}\le\c\,\exp({-}\decay\Hradius)\volume{\M}^{\frac1p}$ and $\displaystyle\lVert\outsc+6\rVert_{\Lp^1(\M)}\le\c\:\exp({-}\scdecay\Hradius)\volume{\M}$; \label{Round_spheres_outric}
 \item $\M$ has $\displaystyle{-}(4-\eta)\sinh(\Hradius)^{{-}2}$ controlled instability; \label{Round_spheres_instability}
 \item one of the following assumptions is true
  \begin{enumerate}[label={\hbox{\normalfont(RS-\arabic{enumi}\alph{*})}}]
		\item $\M$ satisfies $\exp({-}2\Hradius)\,\volume{\M<\Hradius>}\in\interval{\c^{-1}}\c$;\label{Round_sphere_volume}
		\item $\M$ has ${-}\frac12\displaystyle(4-\eta)\sinh(\Hradius)^{{-}2}$ controlled instability\label{Round_sphere_instability_strong};
		\item $\M$ has $\c$-bounded Hawking mass, \ie $\displaystyle\vert\HmHaw(\M)\vert\in\interval{\c^{-1}}\c$. \label{Round_sphere_bounded_mass}
	\end{enumerate}\label{Round_sphere_volume_general}
\end{enumerate}
The \emph{Ricci-mass} $\masstd(\M)=(\masstd^\ui(\M))_\ui\in\R^{3,1}$ of such a surface is defined by 
\[
 \masstd^0(\M) := {-}\frac{\volume{\M}}{16\pi^{\frac32}}{{}_\Hradius\outsymbol G}^0, \qquad\quad
	\masstd^\oi(\M) := \frac{\volume{\M}}{16\sqrt 3\pi^{\frac32}}{{}_\Hradius\outsymbol G}^\oi\quad\forall\,\oi\in\{1,2,3\},
	\]
where ${}_\Hradius{\outsymbol G}^n\,\eflap_n$ denotes the Fourier series of ${}_\Hradius\outsymbol G:=(\outric(\nu,\nu)-\frac12\outsc-1)|_{\M<\Hradius>}$ and $\nu$ is a unit normal of $\M$, \ie ${}_\Hradius{\outsymbol G}^n$ denotes the $n^{\text{th}}$-coefficient of ${}_\Hradius\outsymbol G$ with respect to the complete $\Lp^2(\M<\Hradius>)$-orthogonal system $\{\eflap_n\}_{n=0}^\infty$ of eigenfunctions of the (negative) Laplace operator with corresponding eigenvalues $\ewlap_n$ satisfying $\ewlap_{n+1}\ge\ewlap_n\ge0$.
\end{definition}

\begin{remark}[The definition of Ricci-mass]\labelth{RemDefRicciMass}
Note that coordinate spheres $\M=\sphere^2_\rradius(0):=\{\rad=\rradius\}$ in an asymptotically hyperbolic manifold with mass vector $\mass=(\mass^0,\dots,\mass^3)$ satisfies
\begin{alignat*}4
 \masstd^0(\sphere^2_\rradius(0))
	={}& \frac{{-}1}{8\pi}\int_{\sphere^2_\rradius(0)}{}_\rradius\outsymbol G\,\sinh(\rradius)\d\mug + \mathcal O(\exp({-}\outve\Hradius)) \\
	={}& \frac{{-}1}{8\pi}\int(\outric-\frac12\outsc\outg-\outg)(\nu,\outsymbol X^0)\d\mug + \mathcal O(\exp({-}\outve\Hradius))
	&{}={}& \mass^0 + \mathcal O(\exp({-}\outve\Hradius)), \displaybreak[1]\\
 \masstd^i(\sphere^2_\rradius(0))
	={}& \frac1{8\pi}\int_{\sphere^2_\rradius(0)}{}_\rradius\outsymbol G\,\frac{\outx^\oi\sinh(\rradius)}{\rad}\d\mug + \mathcal O(\exp({-}\outve\Hradius)) \\
	={}& \frac1{8\pi}\int(\outric-\frac12\outsc\outg-\outg)(\nu,\outsymbol X^i)\d\mug + \mathcal O(\exp({-}\outve\Hradius))
	&{}={}& \mass^i + \mathcal O(\exp({-}\outve\Hradius)),
\end{alignat*}
where $\outsymbol X^0$ and $\outsymbol X^1,\dots\outsymbol X^3$ are the radial vector field and the composition of the translation (in the Euclidean standard directions) and the inversion map, \ie the basic conformal vector fields of the hyperbolic space, see \cite{herzlich2015computing} for more information. This motivates our \emph{coordinate independent} definition.
\end{remark}

\begin{remark}[On the mass assumptions of the spheres]
Note that \emph{a~posteriori} any round sphere has even $O(\exp{{-}3\Hradius})$ controlled instability and satisfies $\exp({-}2\Hradius)\volume{\M<\Hradius>}\in\interval{C^{-1}}C$, \ie a~posteriori it satisfies at least~\ref{Round_sphere_volume} and~\ref{Round_sphere_instability_strong}. Thus, \ref{Round_sphere_bounded_mass} is the strongest of the assumptions in~\ref{Round_sphere_volume_general}. We will see that if the round sphere has sufficiently large mean curvature radius $\Hradius$ and is an element of a round cover (see below) with bounded and uniformly timelike Ricci-mass, then all assumptions in~\ref{Round_sphere_volume_general} are equivalent.
\end{remark}

\begin{definition}[Round covers]\label{Round_covers}
Let $(\outM,\outg*)$ be a three-dimensional Riemannian manifold and let $\decay>0$, $\eta\in\interval0*4$, $p\in\interval*1*\infty$, and $\c\ge0$ be arbitrary constants.

A family $\mathcal M:=\{\M<\Hradius>\}_{\Hradius>\Hradius_0}$ of hypersurfaces of $(\outM,\outg*)$ is called \emph{\Wkpro^{2,p}(\c,\eta) CMC-cover} if
\begin{enumerate}[label={\hbox{\normalfont(RC-\arabic*)}}]
 \item each surface $\M<\Hradius>$ is a \Wkpro^{2,p}_{\decay,\scdecay}(\c,\eta) sphere with mean curvature $\H\equiv{-}2\,\frac{\cosh(\Hradius)}{\sinh(\Hradius)}$,\label{Round_foliation_mc}
 \item $\bigcup_{\Hradius>\Hradius_1}\M<\Hradius>$ covers $\outM$ outside of a compact set $\outsymbol K(\Hradius_1)\subseteq\outM$ for every $\Hradius_1\in\interval\Hradius_0\infty$.\label{Round_foliation_cover}
\end{enumerate}
A \Wkpro^{2,p}(\c,\eta) CMC-cover $\mathcal M$ is called \emph{locally unique} if
\begin{enumerate}[resume*]
 \item for every $\M\in\mathcal M$ there exist $q=\Cof{q}[\mathcal M][\M]\in\interval2p$ and $\delta=\Cof{\delta}[\mathcal M][\M]>0$ such that the following holds for every function~$f\in\Wkp^{2,q}(\M<\Hradius>)$\label{Round_foliation_unique}
\begin{itemizeequation}
 \Vert f\Vert_{\Wkp^{2,q}(\M<\Hradius>)}<\delta, \ \H(\graph f)\equiv\text{const}\quad\Longrightarrow\quad\graph f\in\mathcal M.
\end{itemizeequation}%
\end{enumerate}
It has \emph{uniformly timelike Ricci-mass} if
\begin{enumerate}[resume*]
\item $\vert\masstd(\M<\Hradius>)\vert_{\R^{3,1}}<{-}\oc^{{-}1}$ for every $\Hradius>\Hradius_0$
\end{enumerate}
and \emph{bounded Ricci-mass} if 
\begin{enumerate}[resume*]
\item $\vert\masstd(\M<\Hradius>)\vert_{\R^{3,1}}\in\interval-\oc\oc$ for every $\Hradius>\Hradius_0$.
\end{enumerate}
Finally, the \emph{Ricci-mass} $\mass$ of a \Wkpro^{2,p} cover is defined by $\mass:=\mass(\mathcal M):=\lim_\Hradius\mass(\M<\Hradius>)\in\R^{3,1}$ if this limit exist.
\end{definition}
In the following, we abbreviate $\Wkp^{2,p}_{\decay,\decay}$ by $\Wkp^{2,p}_{\decay}$.
\begin{remark}[The assumptions on the mass]
Recalling \refth{RemDefRicciMass} and \cite{nerz2016HBCMCExistence}, we see that the CMC-foliation of any \Ckah^2_{\decay,\scdecay} Riemannian three-manifold with timelike mass vector $\mass$ is \Wkpro^{2,p}(\c,\eta) for some constant $\c$ and every $p\in\interval1\infty$ and $\eta\ge0$. Furthermore, the Ricci mass of this foliation is $({\pm}\vert\mass\vert_{\R^{3,1}},0,0,0)$ for some sign $\pm\in\{{-}1,1\}$.

Note that the assumption that the Ricci-masses of the leaves of a \Wkpro^{2,p} cover are bounded and uniformly timelike implies that the (absolute value of the) Hawking mass is bounded from below, but it does \emph{a~priori} neither imply that the Hawking masses are bounded from above nor that the Ricci-masses converge, \ie that the Ricci-mass of the cover is well-defined. However, \emph{a~posteriori} both is true and even $\mass=(\mass^0,0,0,0)$, see \refth{Ricci_mass_well-defined}.\footnote{The convergence of the Ricci-mass is implied by \refth{MainTheorem}, as it implies that Ricci-masses converge and then \refth{Ricci_mass_well-defined} proves this claim.}
\end{remark}
\begin{remark}[Locally unique covers are foliations]
Note that a~priori we do not assume that the cover is a foliation, \ie that the surfaces are disjoint. However, we will later see that the elements of a locally unique cover with bounded and uniformly timelike Ricci-mass are in fact pairwise disjoint and therefore a~posteriori the cover is a foliation.
\end{remark}
\begin{remark}[Boundedness of the scalar curvature]\labelth{BoundednessScalarCurvature_spheres}
We can reduce the assumption on the scalar curvature by only assuming integrability and one sided boundedness, \ie
\[ \Vert(\outsc+6)^-\Vert_{\Lp^1(\M<\Hradius>)} \le \c\,\exp({-}\scdecay\Hradius)\volume{\M<\Hradius>}, \qquad
	\exp{\Hradius}\Vert\outsc+6\Vert_{\Lp^1(\M<\Hradius>)} \le \tilde{\oc}(\Hradius) \in\Lp^1(\interval\Hradius_0\infty), \]
where $(\emptyarg)^-:=\min\{0,\emptyarg\}$, see also Remarks~\ref{BoundednessScalarCurvature_coordinates}, \ref{BoundednessScalarCurvature_stability}, \ref{BoundednessScalarCurvature_rnu}, and~\ref{BoundednessScalarCurvature_mass}. However, we then have to also assume that the Hawking mass (or equivalent the $0^{\text{th}}$-component of the Ricci-mass) of every $\M<\Hradius>$ is non-negative.
\end{remark}

Finally, we use the following partition of $\Lp^2(\M)$ (for any asymptotically round sphere $\M$) which was introduced and motivated in~\cite[Sect.~4]{nerz2016HBCMCExistence}.
\begin{definition}[Canonical partition of \texorpdfstring{$\Lp^2$}{L-2}]\label{Canonical_partition}
Let $\M$ be a \Wkpro^{2,p}_{\decay,\scdecay}(\c,\eta) sphere of mean curvature radius $\Hradius$. Let $\boost g$ be the $\Lp^2(\M)$-orthogonal projection of a function $g\in\Lp^2(\M)$ on the linear span of eigenfunctions of the (negative) Laplacian with eigenvalue $\lambda$ satisfying $\vert\lambda-2\sinh(\Hradius)^{{-2}}\vert\le \frac32\,\sinh(\Hradius)^{-2}$, \ie
\[ \boost g := \sum\lbrace \eflap_i\,\int_{\M} g\,\eflap_i \d\mug \ \middle|\ \frac12 \le \sinh(\Hradius)^2\,\ewlap_i \le \frac72\rbrace\qquad\forall\,g\in\Lp^2(\M), \]
where $\{f_i\}_{i=0}^\infty$ denotes a complete orthonormal system of $\Lp^2(\M)$ by eigenfunctions $\eflap_i$ of the (negative) Laplace operator with corresponding eigenvalue $\ewlap_i$ satisfying $0\le\ewlap_i\le\ewlap_{i+1}$. Finally, $\deform g:=g-\boost g$ denotes the rest of such a function $g\in\Lp^2(\M)$. Elements of $\boostLp^2(\M):=\lbrace\boost f:f\in\Lp^2(\M)\rbrace$ are called \emph{linearized boosts} and those of $\deformLp^2(\M):=\lbrace\deform f:f\in\Lp^2(\M)\rbrace$ are called \emph{deformations}.
\end{definition}

\section{Regularity of \texorpdfstring{\Wkpro^{2,p}_{\decay}}{asymptotically round} spheres}\label{Regularity_spheres}
\begin{lemma}[Boundedness of the Hawking mass implies boundedness of the area]\label{BoundednessHawkingmassAreaBoundedness}
For each constant $\c>0$, there exist two constants $\Hradius_0=\Cof{\Hradius_0}[\c]$ and $C=\Cof[\c]$ with the following property:

If $\M$ is a closed, oriented hypersurface of a three-dimensional Riemannian manifold satisfying~\ref{Round_spheres_CMC} and~\ref{Round_sphere_bounded_mass} for some $\Hradius>\Hradius_0$, then
\[ \vert\Aradius-\Hradius\vert\le C\exp({-}\Hradius), \quad\text{\ie}\quad
	 \vert\volume{\M} - 4\pi\sinh(\Hradius)^2\vert \le C\,\exp(\Hradius), \]
where $\Aradius$ denotes the hyperbolic area radius, \ie $\volume{\M}=4\pi\sinh(\Aradius)^2$. In particular, $\M$ satisfies~\ref{Round_sphere_volume}.
\end{lemma}
\begin{proof}
As $\H^2-4\equiv4\,\sinh(\Hradius)^{{-}2}$, we know
\begin{equation*}
 \frac12\sinh(\Aradius)\:\vert 1 - \frac{\sinh(\Aradius)^2}{\sinh(\Hradius)^2}\vert = \vert\HmHaw\vert \in \interval{\c^{-1}}\c. \labeleq{Hawking_mass_1}
\end{equation*}
This implies
\[ \left.\begin{aligned}
		\Hradius\le\Aradius\ \ {}& \Longrightarrow&  \exp(2(\Aradius-\Hradius)) \le{}& \frac{\sinh(\Aradius)^2}{\sinh(\Hradius)^2} \le 1 + \frac C{\sinh(\Aradius)} \\
		\Hradius\ge\Aradius\ \ {}& \Longrightarrow& \exp(2(\Aradius-\Hradius)) \ge{}& \frac{\sinh(\Aradius)^2}{\sinh(\Hradius)^2} \ge 1 - \frac C{\sinh(\Hradius)}
	\end{aligned}\ \,\right\}\;\ \Longrightarrow\ \;\vert\Aradius-\Hradius\vert\le C\,\exp({-}\Hradius). \qedhere
 \]
\end{proof}
\begin{lemma}[Strictly controlled instability implies boundedness of the area]\label{StrictControlledInstabilityAreaBoundedness}
For each $\c>0$, $\decay>\frac52$, and $\eta\in\interval0*4$, there exist two constants $\Hradius_0=\Cof{\Hradius_0}[\c][\decay][\eta]$ and $C=\Cof[\c][\decay][\eta]$ with the following property:

If $\M$ is a \Wkpro^{2,p}_{\decay} sphere satisfying~\ref{Round_sphere_instability_strong} for some $\Hradius>\Hradius_0$, then $\M$ satisfies~\ref{Round_sphere_volume} for $C$ instead of $\c$, too.
\end{lemma}
\begin{proof}
This proof is equivalent to the begin of \cite[Proof~of~Thm~3.1]{nerz2016HBCMCExistence}. We recall it nevertheless for the readers convenience. We start as in \cite[Lemma~4.1]{NevesTianExistenceCMC_I} and use the test functions $\varphi_\oi:=\outx^i\circ\psi^{-1}$, where $\psi:\sphere\to\M$ is a conformal parametrization of $\M$ with $\int\varphi_i\d\mug=0$. These were already used by Huisken-Yau in \cite[Prop.~5.3]{huisken_yau_foliation} and were based on an idea by Christodoulou-Yau, \cite{christodoulou71some}. By the controlled instability assumption, this implies
\begin{equation*}
 {-}\frac{8\pi}3
 =	\int_{\sphere}\outx_\oi\,\sphlaplace\outx_\oi\d\sphmug
 =	\int_{\M}\varphi_\oi\:\laplace\varphi_\oi\d\mug 
 \le \int(\frac{4-\eta}{2\sinh(\Hradius)^2}-\trtr\zFund\zFund-\outric(\nu,\nu))\varphi_\oi^2\d\mug
\end{equation*}
for every $\oi\in\{1,2,3\}$, where we have used the conformal invariance of $\laplace f\d\mug$. Now, we recall that $(\sum_i\varphi_i^2)\circ\psi=\sum_i\outx_i^2\equiv1$ to get
\begin{align*}
 8\pi
 \ge{}& \int\trtr\zFundtrf\zFundtrf + \frac{\H^2-4}2
					+ 2+\outric(\nu,\nu)
					- \frac{4-\eta}{2\sinh(\Hradius)^2} \d\mug \\
 \ge{}& \Vert\zFundtrf\Vert_{\Lp^2(\M)}^2
				+ \int\frac{\eta+C\exp((2-\decay)\Hradius)}{2\sinh(\Hradius)^2}\d\mug, \labeleq{Triv_on_L2_Zfundtrf}
\end{align*}
\ie
\[ \eta\sinh(\Aradius)^2 \le 16\pi\sinh(\Hradius)^2 + C\,\exp((2-\decay)\Hradius)\,\sinh(\Aradius)^2 \]
implying $\Aradius\le\Hradius+C$, \ie $\volume{\M}\le C\exp(2\Hradius)$. On the other hand, the Gau\ss-Bonnet theorem and the Gau\ss\ equation combined with the assumptions on $\outric$ give
\begin{align*}
 8\pi
  ={}& \int\sc\d\mug
	= \int(\outsc-2\outric(\nu,\nu)-\trtr\zFundtrf\zFundtrf+\frac{\H^2}2\d\mug) \\
	\le{}& C\,\int \exp({-}\decay\Hradius) \d\mug
					+ \int\frac2{\sinh(\Hradius)^2}\d\mug
	\le (8\pi+C\,\exp((2-\decay)\Hradius))\,\frac{\sinh(\Aradius)^2}{\sinh(\Hradius)^2}
\end{align*}
implying $\Aradius\ge\Hradius-C$, \ie $C^{-1}\exp(2\Hradius)\le\volume{\M}$.
\end{proof}
\begin{corollary}[Round spheres have bounded area]\labelth{RoundSpheresBoundedArea}
For each $\c>0$, $\decay>\frac52$, and $\eta\in\interval0*4$, there exist two constants $\Hradius_0=\Cof{\Hradius_0}[\c][\decay][\eta]$ and $C=\Cof[\c][\decay][\eta]$ with the following property:

If $\M$ is a \Wkpro^{2,p}_{\decay} sphere for some mean curvature radius $\Hradius>\Hradius_0$, then $\M$ satisfies~\ref{Round_sphere_volume} for $C$ instead of $\c$. 
\end{corollary}

Now, let us cite two major regularity results---in the notation we introduced above. We can apply these results due to the result in \refth{RoundSpheresBoundedArea}.
\begin{lemma}[{\cite[Prop.~3.5]{nerz2016HBCMCExistence}}]\labelth{Regularity_Surfaces}
For all constants $\decay\in\interval2*3$, $\eta\in\interval0*4$, $\c>0$, and $p\in\interval2\infty$, there exist two constants $\Hradius_0=\Cof{\Hradius_0}[\decay][\eta][\c][p]$ and $C=\Cof[\decay][\eta][\c][p]$ with the following property:

If $(\M,\g*)$ is a \Wkpro^{2,p}_{\decay,\decay}(\c,\eta) sphere with $\Hradius>\Hradius_0$, then there exists a conformal parametrization $\varphi:\sphere\to\M$ with corresponding conformal factor $\conformalf\in\Hk^2(\sphere)$, \ie $\varphi^*\g*=\exp(2\conformalf)\,\sinh(\Hradius)^2\,\sphg*$, such that
\begin{equation*}
 \Vert\conformalf\Vert_{\Wkp^{2,p}(\sphere,\sphg*)} \le C\,\exp((2-\decay)\,\Hradius), \qquad
 \Vert\zFundtrf\Vert_{\Wkp^{1,p}(\M)} \le C\,\exp((1+\frac2p-\decay)\,\Hradius), \labeleq{Regularity_of_the_spheres__k}
\end{equation*}
where $\sphg*$ denotes the standard metric of the Euclidean unit sphere. \pagebreak[2]
\end{lemma}

\begin{remark}
By the Gau\ss\ equation and the Gau\ss-Bonnet theorem, \refth{Regularity_Surfaces} implies that every \Wkpro^{2,1}_{2+\outve} sphere satisfies $\vert\HmHaw(\M)-\masstd^0(\M)\vert \le C\,\exp({-}\outve\Hradius)$.
\end{remark}

\begin{proposition}[{\cite[Prop.~4.3]{nerz2016HBCMCExistence}}]\labelth{Stability}
For all constants $\decay:=\frac52+\outve\in\interval{\frac52}*3$, $\scdecay\ge3+\outve$, $\eta\in\interval0*4$, $\c>0$, $p>2$, $q\in\interval1*p$ with $q<\infty$, there exist two constants $\Hradius_0=\Cof{\Hradius_0}[\outve][\eta][\c][p]$ and $C=\Cof[\outve][\eta][\c][p][q]$ with the following property:

If $(\M,\g*)$ is a \Wkpro^{2,p}_{\decay,\scdecay}(\c,\eta) sphere with $\Hradius>\Hradius_0$, then~\ref{Round_sphere_instability_strong} holds for $\eta$ arbitrary close to $4$ (depending on $\Hradius_0$). More precisely $\jacobiext$ is invertible and if $\efjac\in\Hk^2(\M)$ is an eigenfunction of ${-}\jacobiext*$ with corresponding eigenvalue $\ewjac$, then either $\vert\hspace*{.05em}\ewjac\vert \ge \frac32\sinh(\Hradius)^{{-}2}$ or
\begin{equation*}
	\Vert\,\deform{\efjac}\Vert_{\Hk^2(\M)} \le C\,\exp({-}(\frac12+\outve)\Hradius)\Vert\,\efjac\Vert_{\Hk^2(\M)},\qquad \vert\,\ewjac-\frac{6\,\mass^0(\M)}{\sinh(\Hradius)^3}\vert \le C\,\exp({-}(3+\outve)\,\Hradius) \labeleq{Stability__ewjac} \end{equation*}
and for all functions $g,h\in\Hk^2(\M)$ the inequality
\begin{equation*}
 \vert \int_{\M} (\jacobiext*\boost g)\,\boost h\d\mug + \frac{6\,\masstd^0(\M)}{\sinh(\Hradius)^3}\,\int_{\M}\boost g\,\boost h\d\mug \vert \le C\,\exp({-}(3+\outve)\Hradius)\,\Vert\boost g\Vert_{\Lp^2(\M)}\,\Vert\boost h\Vert_{\Lp^2(\M)}, \labeleq{Stability__t}
\end{equation*}
holds. Furthermore, the corresponding $\Wkp^{2,p}$-inequalities
\begin{align*}
 \Vert\boost g\Vert_{\Wkp^{2,q}(\M)}	\le{}& (\frac{\sinh(\Hradius)^3}{6\,\vert\mass^0(\M)\vert}+C\,\exp((3-\outve)\Hradius))\,\Vert\jacobiext*g\Vert_{\Lp^q(\M)}, \\
 \Vert\deform g\Vert_{\Wkp^{2,q}(\M)} \le{}& C\,\exp(2\Hradius)\,\Vert\jacobiext*g\Vert_{\Lp^q(\M)}, \\
 \Vert\Hesstrf\,g\Vert_{\Lp^q(\M)}		\le{}& C\,\exp((\frac12-\outve)\Hradius)\,\Vert\jacobiext*g\Vert_{\Lp^q(\M)}
\end{align*}
hold for every function $g\in\Wkp^{2,1}(\M)$.
\end{proposition}
\begin{remark}[Assuming only one-sided boundedness of the scalar curvature]\labelth{BoundednessScalarCurvature_stability}
If we only assuming one sided boundedness of $\outsc$ as it is explained in \refth{BoundednessScalarCurvature_spheres}, then~\eqref{Stability__ewjac} and~\eqref{Stability__t} have to be weakened to
\begin{equation*}
	\Vert\,\deform{\efjac}\Vert_{\Hk^2(\M)} \le C\,\exp({-}(\frac12+\outve)\Hradius)\Vert\,\efjac\Vert_{\Hk^2(\M)},\qquad \ewjac\ge\frac{6\,\mass^0(\M)}{\sinh(\Hradius)^3} + C\,\exp({-}(3+\outve)\,\Hradius),  \tag{\ref{Stability__ewjac}'}
\end{equation*}
and
\begin{equation*}
 \vert\int_{\M} (\jacobiext*\boost g)\,\boost h\d\mug\vert \le \frac{6\,\masstd^0(\M)}{\sinh(\Hradius)^3}\,\int_{\M}\boost g\,\boost h\d\mug + C\,\exp({-}(3+\outve)\Hradius)\,\Vert\boost g\Vert_{\Lp^2(\M)}\,\Vert\boost h\Vert_{\Lp^2(\M)} \tag{\ref{Stability__t}'}
\end{equation*}
respectively, see \cite[Remark~4.4]{nerz2016HBCMCExistence}. As we then also assume that the Hawking mass (or equivalent the $0^{\text{th}}$-component of the Ricci-mass) is positive, this still implies that $\jacobiext*$ is invertible.
\end{remark}
\begin{lemma}[Local estimates of the lapse function]\labelth{Local_Estimates_rnu}
For all constants $\decay:=\frac52+\outve\in\interval{\frac52}*3$, $\scdecay\ge3+\outve$, $\eta\in\interval0*4$, $\c>0$, $p>2$, $q\in\interval1*p$ with $q<\infty$, there exist two constants $\Hradius_0=\Cof{\Hradius_0}[\outve][\eta][\c][p]$ and $C=\Cof[\outve][\eta][\c][p][q]$ with the following property:

If $(\M,\g*)$ is a \Wkpro^{2,p}_{\decay,\scdecay}(\c,\eta) sphere with $\vert\HmHaw(\M)\vert>\c$ and $\Hradius>\Hradius_0$, then there exist a constant $\delta_0>0$ and a $\Ck^1$-map $\Phi:\interval-{\delta_0}{\delta_0}\times\sphere^2$ with $\M=\Phi(0,\sphere^2)$ and such that $\M<\delta>:=\Phi(\delta,\sphere^2)$ is a $\Wkp^{2,p}$-hypersurface with constant mean curvature $\H<\delta>\equiv{-}2\frac{\cosh(\Hradius+\delta)}{\sinh(\Hradius+\delta)}$ for every $\delta\in\interval-{\delta_0}{\delta_0}$. In this setting, the lapse function $\rnu:=\outg(\partial*_\Hradius\Phi,\nu)|_{\M}$ satisfies
\begin{equation*}\labeleq{rnu_1_ineq}
 \Vert\deform{\rnu}-1\Vert_{\Wkp^{2,q}(\M)} \le C\,\exp((\frac2q-1-\outve)\Hradius), \qquad
 \Vert\boost\rnu - \sqrt{\frac{\volume{\M}}3}\frac{\masstd^\oi}{\masstd^0}\eflap_\oi\Vert_{\Wkp^{3,q}(\M)}
	\le C\,\exp((\frac2q-\outve)\Hradius),
\end{equation*}
where $(\mass^\ui)_{\ui=0}^3:=\mass:=\mass(\M)$ and where $\eflap^\oi$ are $\Lp^2$-orthonormal eigenfunctions of the Laplace operator with eigenvalue $\ewlap^\oi\in\interval{\sinh(\Hradius)^{{-}2}}{3\sinh(\Hradius)^{{-}2}}$. In particular, the corresponding $\Wkp^{3,q}(\M)$-estimates of $\boost{\rnu}$ hold. Furthermore, $\rnu$ is strictly positive if $\masstd(\M)$ is controlled timelike, \ie $\vert\masstd(\M)\vert_{\R^{3,1}}<{-}\c^{{-}1}$, and $\rnu$ changes sign if $\masstd(\M)$ is controlled spacelike, \ie $\vert\masstd(\M)\vert_{\R^{3,1}}>\c^{{-}1}$.
\end{lemma}
\begin{proof}
Without loss of generality $q=p\in\interval2\infty$. We can assume that $\Hradius$ is so large that we can apply \refth{Regularity_Surfaces} and \refth{Stability}. We know that the stability operator 
\[ \jacobiext*:\Wkp^{2,q}(\M)\to\Lp^q(\M):f\mapsto\laplace f + (\trtr\zFund\zFund+\outric*(\nu,\nu))f \]
is the Fr\'echet derivative of the mean curvature map
\[ \textsf H : \Wkp^{2,q}(\M) \to \Lp^q(\M) : f \mapsto \H(\graph f) \]
at $f=0$ (for every $q>2$), where $\H(\graph f)$ denotes the mean curvature of the graph of $f$ which we interpret as function on $\M$.\pagebreak[1] By \refth{Stability}, the stability operator is invertible if $\vert\HmHaw\vert>\c$. Thus, the inverse function theorem implies that $\textsf H$ is bijective from a $\Wkp^{2,q}(\M)$-neigh\-bor\-hood of $0\in\Wkp^{2,q}(\M)$ to a $\Lp^q(\M)$-neighborhood of $\H\in\Lp^q(\M)$. As the Hawking mass depends continuously on $\Wkp^{2,q}$-deformation of $\M$, this proves the first claim.

Per Definition of $\Phi$ and $\M<\delta>$, we know
\[ \jacobiext*\rnu = \partial[\delta]@{\H<\delta>}<\delta=0> \equiv \frac2{\sinh(\Hradius)^2}. \]
Thus, \refth{Regularity_Surfaces} implies
\[ \vert\hspace{.05em}\jacobiext*(\rnu-1) + 2 + \outric*(\nu,\nu)\vert \le C\,\exp({-}(3+\outve)\,\Hradius). \]
Hence using \refth{Stability}, we get
\[ \Vert\deform{(\rnu-1)}\Vert_{\Wkp^{2,p}(\M)} \le C\,\exp({-}(1+\outve-\frac2p)\Hradius) \]
---which is the first inequality in~\eqref{rnu_1_ineq} as $\boost 1\equiv0$. Choosing $\rnu^\oi\in\R$ with $\boost\rnu=\rnu^\oi\,\eflap_\oi$ for the $\Lp^2$-orthogonal eigenfunctions $\eflap_i$ of the Laplace operator having an eigenvalue $\ewlap_\oi\in\interval{\sinh(\Hradius)^{{-}2}}{3\sinh(\Hradius)^{{-}2}}$, we see
\begin{align*}\hspace{3em}&\hspace{-3em}
 \vert \rnu^\oi - \sqrt{\frac{4\pi}3}\frac{\sinh(\Hradius)\masstd^\oi}{\masstd<\Hradius>^0}\vert \\
	\le{}& \vert\int(\rnu-1)\eflap^\oi\d\mug - \frac{\sinh(\Hradius)^3}{6\masstd^0}\int\jacobiext(\rnu-1)\eflap^\oi\d\mug \vert + C\,\exp((1-\outve)\Hradius)
	\le C\,\exp((1-\outve)\Hradius),
\end{align*}
where $\masstd:=\mass(\M)$. This proves~\eqref{rnu_1_ineq} and
\begin{equation*}\labeleq{mass_by_rnu}
	\vert\frac{\left|\masstd(\M)\right|_{\R^{3,1}}^2}{\masstd^0(\M)} - \frac{\sinh(\Hradius)}{8\pi}\int\outsymbol G(\nu,\nu)\,\rnu\d\mug \vert \le C\,\exp({-}\outve\Hradius).
\end{equation*}
In the setting of the Euclidean sphere, we have
\[ \Vert\spheflap^\oi\Vert_{\Lp^2(\euksphere)}^2 = \frac{4\pi}{3} \Vert\spheflap^\oi\Vert_{\Lp^\infty(\euksphere)}^2 \qquad\forall\,\spheflap^\oi\in\Lp^\infty(\sphere^2):\ \sphlaplace\spheflap^\oi={-}2\hspace{.05em}\spheflap^\oi \]
and therefore \refth{Regularity_Surfaces} implies that the same holds with respect to $\g$ up to a lower order error term. Thus, we get
\begin{equation*}\labeleq{rnu_mass_ineq}
 \left|\Vert\boost\rnu\Vert_{\Lp^\infty(\M)}^2 - \vert\masstd<\Hradius>^0\vert^{{-}2}\,\vert(\masstd<\Hradius>^\oi)_{\oi=1}^3\vert_{\R^3}^2\right|\le C\,\exp({-}\outve\Hradius).
\end{equation*}
By again comparing with the Euclidean setting, we see that $\sup(\boost\rnu)^+$ is equal to $\sup(\boost\rnu)^-$ (up to term vanashing with $\Hradius\to\infty$), where $(\,{\cdot}
\,)^+:=\max\{{\,\cdot\,},0\}$ and $(\,{\cdot}\,)^-:={-}\min\{{\,\cdot\,},0\}$ denote the positive and negative part of a function. In particular, $\rnu$ changes sign if $\Vert\boost\rnu\Vert_{\Lp^\infty(\M)}>1+\delta$ and does not change sign if $\Vert\boost\rnu\Vert_{\Lp^\infty(\M)}<1-\delta$, where $\delta>0$ has to be independent of $\Hradius$.
\end{proof}
\begin{remark}[Assuming only one-sided boundedness of the scalar curvature]\labelth{BoundednessScalarCurvature_rnu}
If we only assuming one sided boundedness of $\outsc$ as it is explained in \refth{BoundednessScalarCurvature_spheres}, then \eqref{rnu_1_ineq} and \eqref{mass_by_rnu} have to be weakened to
\begin{equation*}\tag{\ref{rnu_1_ineq}'}
 \Vert\deform{\rnu}-1\Vert_{\Wkp^{2,q}(\M)} \le C\,\exp((\frac2q-1-\outve)\Hradius), \qquad
 \vert\boost\rnu\vert \le \vert\sqrt{\frac{\volume{\M}}3}\frac{\masstd^\oi}{\masstd^0}\eflap_\oi\vert + C\,\exp((\frac2q-\outve)\Hradius),
\end{equation*}
and
\begin{equation*}\edef\oldmassref{\getrefnumber{mass_by_rnu}'}
	\frac{\sinh(\Hradius)}{8\pi}\int\outsymbol G(\nu,\nu)\,\rnu\d\mug \le \frac{\left|\masstd(\M)\right|_{\R^{3,1}}^2}{\masstd^0(\M)} + C\,\exp({-}\outve\Hradius),  \labeleq[\oldmassref]{mass_by_rnu'}
\end{equation*}
respectively, see \refth{BoundednessScalarCurvature_stability}.
\end{remark}

\section{Regularity of \texorpdfstring{\Wkpro^{2,p}_{\decay}}{asymptotically round} covers}\label{Regularity_covers}
\begin{proposition}[The Ricci mass is well-defined]\labelth{Ricci_mass_well-defined}
For all constants $\decay\ge\frac52+\outve\in\interval{\frac52}*3$, $\scdecay\ge3+\outve$, $\eta<4$, $p>2$, and $\c\ge0$, there exist two constants $\Hradius_0=\Cof{\Hradius_0}[\outve][p][\eta][\c]$ and $C=\Cof[\outve][p][\eta][\c]$ with the following property:

Let $\mathcal M=\{\M<\Hradius>\}_{\Hradius>\Hradius_1}$ be a locally unique \Wkpro^{2,p}_{\decay,\scdecay}(\c,\eta) foliation with uniformly timelike Ricci-mass. If $\Hradius_1\ge\Hradius_0$, then $\vert\HmHaw<\Hradius>\vert-C\exp({-}\outve\Hradius)$ is monotone increasing in $\Hradius$ and
\begin{equation*}\labeleq{derivative_Hawking_mass}
 \vert\partial[\Hradius]@{(\HmHaw<\Hradius>)^2} - 2\left|(\masstd<\Hradius>^\oi)_{\oi=1}^3\right|_{\R^3}^2\vert 
		 \le C\,\exp({-}\outve\Hradius).
\end{equation*}
If $\mathcal M$ has furthermore bounded Ricci-mass, then $\mass^0(\mathcal M)$ is well-defined and non-vanishing and the function $\Hradius\to\vert(\masstd<\Hradius>^i)_{i=1}^3\vert_{\R^3}^2$ is integrable. In particular, the total part $\vert\masstd(\M<\Hradius>)\vert_{\R^{3,1}}$ of the Ricci-mass of the leaves converge if and only if the Ricci-mass of the leaves converge, \ie if $\mass(\mathcal M)$ is well-defined, and in this setting $\mass(\mathcal M)=(\mass^0(\mathcal M),0,0,0)$.
\end{proposition}
\begin{proof}
Per assumption, we know
\[ \vert\HmHaw(\M<\Hradius>)\vert
	\ge \vert\masstd^0(\M<\Hradius>)\vert - C\exp({-}\outve\Hradius)
	\ge {-}\vert\masstd(\M<\Hradius>)\vert_{\R^{3,1}} - C\exp({-}\outve\Hradius)
	\ge \frac1{2\oc}\qquad\forall\,\Hradius>\Hradius_1 \]
if $\Hradius_1$ is sufficiently large. Therefore, we can without loss of generality assume that $\Hradius_1$ is so large that we can apply all the results so far on each leaf $\M<\Hradius>$.

Fix any $\Hradius'>\Hradius_1$ and suppress the corresponding index $\Hradius'$. By \refth{Local_Estimates_rnu}, we know that there exists a map $\Phi:\interval{\Hradius'-\delta}{\Hradius'+\delta}\times\sphere^2$ such that $\Phi(\Hradius,\sphere^2)$ has constant mean curvature $\H<\Hradius>\equiv{-}2\frac{\cosh(\Hradius)}{\sinh(\Hradius)}$, $\M=\Phi(\Hradius,\sphere^2)$, and that the lapse function satisfies~\eqref{rnu_1_ineq}. By the uniqueness assumption this implies $\M<\Hradius>=\Phi(\Hradius,\sphere^2)$ for every $\Hradius\in\interval{\Hradius'-\delta}{\Hradius'+\delta}$. As $\Hradius$ was arbitrary, we can repeat this step and---by applying diffeomorphisms to $\sphere^2$---we can glue those $\Phi$ to one smooth map $\Phi:\interval{\Hradius_1}\infty\times\sphere^2\to\outM$ with $\M<\Hradius>=\Phi(\Hradius,\sphere^2)$.

Now, we note that
\begin{equation*}\labeleq{conv_mass__1}
 \partial[\Hradius]@{\volume{\M<\Hradius>}} = {-}\int\H\rnu\d\mug, \qquad
 \jacobiext\rnu=\partial[\Hradius]@{\H<\Hradius>} \equiv 2\sinh(\Hradius)^{{-}2}, \end{equation*}
where $\rnu:=\outg(\partial*_\Hradius\Phi,\nu)$ again denotes the lapse function. In particular, \eqref{mass_by_rnu} implies
\begin{align*}
 \vert\partial[\Hradius]\frac{\volume{\M<\Hradius>}}{\sinh(\Hradius)^2} + \frac{8\pi|\masstd|_{\R^{3,1}}^2}{\sinh(\Hradius)\masstd^0} \vert
 \le{}& \vert\int\frac{{-}\H\rnu}{\sinh(\Hradius)^2} + \outsymbol G(\nu,\nu)\,\rnu\d\mug - \frac{2\volume{\M}}{\sinh(\Hradius)^2} \vert + C\,\exp({-}(1+\outve)) \\
 \le{}& \vert\int(\frac{\H^2}2+\outric(\nu,\nu))\rnu - \jacobiext\rnu\d\mug \vert + C\,\exp({-}(1+\outve)\Hradius).
\end{align*}
As $\int\laplace\rnu\d\mug=0$ and $\vert\zFundtrf\vert\le C\,\exp(2(1-\decay))$, this gives us
\begin{equation*}\labeleq{conv_mass__2}
	\vert\partial[\Hradius]\frac{\volume{\M<\Hradius>}}{\sinh(\Hradius)^2} + \frac{8\pi|\masstd|_{\R^{3,1}}^2}{\sinh(\Hradius)\masstd^0} \vert \le C\,\exp({-}(1+\outve)).
\end{equation*}
By the definition of the Hawking mass $\HmHaw<\Hradius>:=\HmHaw(\M<\Hradius>)$, we get
\[ \vert \left.\frac{\partial*^k}{\partial*\Hradius^k}\right|_{\Hradius=\Hradius'}(\frac{\volume{\M<\Hradius>}}{4\pi\sinh(\Hradius)^2} - (1 - 2\frac{\HmHaw<\Hradius>}{\sinh(\Hradius)}))\vert \le C\,\exp((1-\outve)\Hradius) \qquad\forall\,k\in\{0,1\}. \]
Thus, we have
\[ \vert\partial[\Hradius]@{\HmHaw<\Hradius>}
		 - (\HmHaw<\Hradius> - \sinh(\Hradius)\partial[\Hradius]\frac{\volume{\M<\Hradius>}}{8\pi\sinh(\Hradius)^2})\vert \le C\,\exp({-}\outve\Hradius). \]
and \eqref{conv_mass__2} gives
\[ \vert\partial[\Hradius]@{\HmHaw<\Hradius>} - (\masstd<\Hradius>^0)^{{-}1}\,\left|(\masstd<\Hradius>^\oi)_{\oi=1}^3\right|_{\R^{3}}^2\vert 
		= \vert\partial[\Hradius]@{\HmHaw<\Hradius>} - (\masstd<\Hradius>^0+\frac{|\masstd<\Hradius>|_{\R^{3,1}}^2}{\masstd<\Hradius>^0})\vert  \le C\,\exp({-}\outve\Hradius). \]
This implies that $\Hradius\to\vert\HmHaw<\Hradius>\vert-C\,\exp({-}\outve\Hradius)$ is monotone increasing and the combination with $\vert\HmHaw<\Hradius>-\masstd<\Hradius>^0\vert\le C\,\exp({-}\outve\Hradius)$ proves~\eqref{derivative_Hawking_mass}.

Now, let us assume that $\mathcal M$ has bounded Ricci-mass, but $\vert\HmHaw<\Hradius>\vert$ (and therefore $\vert\masstd<\Hradius>^0\vert$) is unbounded. By the above monotonicity, $\vert\HmHaw<\Hradius>\vert$ has to converge to infinity. By the boundedness of the Ricci-mass, we therefore know
\[ \frac2{\sqrt5}\vert\masstd<\Hradius>^0\vert
		= \sqrt{\frac45\vert\masstd<\Hradius>^0\vert^2}
		\le \sqrt{\vert\masstd<\Hradius>^0\vert^2 - C}
		\le \sqrt{\vert\masstd<\Hradius>^0\vert^2 + \vert\masstd<\Hradius>\vert_{\R^{3,1}}^2}
		\le \vert(\masstd<\Hradius>^\oi)_\oi\vert_{\R^3} \]
for sufficiently large $\Hradius$. Thus, \eqref{derivative_Hawking_mass} implies
\[ \partial[\Hradius]@{\:\vert\HmHaw<\Hradius>\vert}
		\ge \frac{\vert(\masstd<\Hradius>^\oi)_\oi\vert_{\R^3}^2}{\vert\masstd<\Hradius>^0\vert} - C
		\ge \frac45\vert\masstd<\Hradius>^0\vert-C
		\ge \frac34\vert\HmHaw<\Hradius>\vert. \]
This proves $\vert\HmHaw<\Hradius>\vert\ge \exp(\frac23\Hradius)$ for sufficiently large $\Hradius$ which contradicts $\vert\HmHaw<\Hradius>\vert\le \exp(\frac12\Hradius)$ being an implication of $\vert\g<\Hradius>-\sinh(\Hradius)^2\sphg\vert\le C\,\exp((\frac32-\outve)\Hradius)$, see \refth{RoundSpheresBoundedArea}. Thus, $\HmHaw<\Hradius>$ is bounded and by the above monotonicity it therefore converges. Thus, $\mass^0:=\lim_\Hradius\masstd<\Hradius>^0$ is well-defined, non-vanishing, and finite.
\end{proof}
\begin{remark}[Assuming only one-sided boundedness of the scalar curvature]\labelth{BoundednessScalarCurvature_mass}
If we only assuming one sided boundedness of $\outsc$ as it is explained in \refth{BoundednessScalarCurvature_spheres}, then we can not apply \eqref{mass_by_rnu}, but only \eqref{mass_by_rnu'} and therefore \eqref{derivative_Hawking_mass} has to be weakened to
\begin{equation*}\tag{\ref{derivative_Hawking_mass}'}
 \partial[\Hradius]@{(\HmHaw<\Hradius>)^2} \ge 2\left|(\masstd<\Hradius>^\oi)_{\oi=1}^3\right|_{\R^3}^2 - C\,\exp({-}\outve\Hradius).
\end{equation*}
However, the rest of the claims remain true.
\end{remark}

The next step is to conclude a even better decay estimate on the second fundamental form as we established in \refth{Regularity_Surfaces}.
\begin{lemma}[Better decay rates of the second fundamental form]\labelth{second_fundamental_form}
For all constants $\decay\ge\frac52+\outve\in\interval{\frac52}*3$, $\scdecay\ge3+\outve$, $\eta<4$, $p>2$, $\c\ge0$, and $q\in\interval2*p$ with $q<\infty$, there exist two constants $\Hradius_0=\Cof{\Hradius_0}[\outve][p][\eta][\c]$ and $C=\Cof[\outve][p][\eta][\c][q]$ with the following property:

Let $\mathcal M=\{\M<\Hradius>\}_{\Hradius>\Hradius_1}$ be a locally unique \Wkpro^{2,p}_{\decay,\scdecay}(\c,\eta) foliation with uniformly timelike Ricci-mass, where $\Hradius_1\ge\Hradius_0$. Then
\begin{equation*}\labeleq{Stronger_estimates_on_zFundtrf}
 \Vert\zFundtrf\Vert_{\Lp^q(\M<\Hradius>)}
		+ \Vert\levi\zFundtrf\Vert_{\Lp^q(\M<\Hradius>)}
		\le C\,\exp((\frac2q-\decay)\Hradius).
\end{equation*}
\end{lemma}
\begin{remark}
On the first glimpse, it seems to be an unnatural estimate as it implies that the trace free part of the second fundamental form and its ($\M<\Hradius>$-tangential) derivative decay with the same decay rate. However, the author explained in \cite[Rem.~3.7]{nerz2016HBCMCExistence} why this is actually a natural property in the hyperbolic space. Furthermore, the author proved the equivalent estimate for CMC-leaves in an a~priori asymptotically hyperbolic manifold in \cite[Thm~3.1]{nerz2016HBCMCExistence} using a coordinate depending ansatz.
\end{remark}
\begin{proof}
Without loss of generality, $\decay=\frac52+\outve$, $\scdecay=3+\outve$, $\dot\bigcup_{\Hradius>\Hradius_0}\M<\Hradius>=\outM$, $p=q<\infty$, and $\Hradius_0$ is so large, that we can apply \refth{Regularity_Surfaces} and \refth{Stability} as well as \refth{Local_Estimates_rnu} on every $\M<\Hradius>$ and \refth{Ricci_mass_well-defined} on $\mathcal M$. By \refth{Local_Estimates_rnu} and \refth{Ricci_mass_well-defined} there exists a smooth map $\Phi:\interval{\Hradius_0}\infty\times\sphere\to\outM$ with $\Phi(\Hradius,\sphere^2)=\M<\Hradius>$ for every $\Hradius$ and for each such diffeomorphism the lapse function $\rnu<\Hradius>$ is everywhere positive.\pagebreak[1]

A direct calculation proves
\begin{align*}\labeleq{derivative-zFundtrf}
 \partial[\Hradius]@{\zFundtrf}
  ={}& \partial[\Hradius]@{\zFund} - \frac12(\partial[\Hradius]@\H)\g - \frac12\H\partial[\Hradius]@\g \\
	={}& \rnu(\outric-\frac12\sc\g-2\trzd\zFundtrf\zFund) + \Hess\,\rnu \\
		 & - \frac12(\laplace\rnu+(\trtr\zFund\zFund+\outric(\nu,\nu))\rnu)\g + \rnu\H\zFund \\
	={}& \rnu(\outric+\frac12\outric(\nu,\nu)\g-\frac12\outsc\g
		  - 2\trzd\zFundtrf\zFundtrf)x
			+ \Hesstrf\,\rnu,
\end{align*}
where \pagebreak[1]%
we identified tensors with their pullback along $\Phi$ and used the Gau\ss\ equation. In particular, we have
\begin{align*}
 \frac1q\partial[\Hradius]@{\vert\zFundtrf\vert^q}
   ={}& \frac 12\vert\zFundtrf\vert^{q-2}\partial[\Hradius]@{\vert\zFundtrf\vert^2} \\
	={}& \vert\zFundtrf\vert^{q-2}(\vphantom{\Big|}\rnu\,\trtr\zFundtrf{\lieD{\rnu\nu}\zFundtrf} + \rnu\,\trtr{\trzd\zFund\zFundtrf}\zFundtrf) \\
	={}& \vert\zFundtrf\vert^{q-2}\trtr*{\vphantom{\bigg|}\zFundtrf}{\Big(\rnu(\outric+2\outg)+\Hesstrf\,\rnu\Big)} + \frac12\H\rnu\vert\zFundtrf\vert^q
\end{align*}
for every $q\ge2$, where we have used $\trtr\zFundtrf\outg=\tr\zFundtrf\equiv0$ and $\trtr{\trzd\zFundtrf\zFundtrf}\zFundtrf\equiv0$. \pagebreak[1]%
The latter is true as the dimension is two.\footnote{More precisely, $\zFundtrf$ is only non-vanishing on $\textrm T\M<\Hradius>$ and this space is two-dimensional. And in two dimensions every odd power of a tracefree operator is tracefree.}\pagebreak[1] By comparing $\deform{(\,{\cdot}\,)}$ with its analog on $\euksphere^2$, we know $\Vert\Hesstrf\,\boost\rnu\Vert_{\Lp^p(\M)}\le C\,\exp({-}(\frac52-\frac2p+\outve)\Hradius)$, see \refth{Stability}. Therefore the H\"older inequality and $\partial[\Hradius]\d\mug={-}\H\d\mug$ imply
\begin{align*}
 \vert\partial[\Hradius]\Vert\zFundtrf\Vert_{\Lp^p(\M<\Hradius>)}^p\vert 
	\le{}& 2\int(\vert\rnu\vert\,\vert\outric+2\outg\vert+\vert\Hesstrf\,\rnu\vert)\vert\zFundtrf\vert^{p-1}\d\mug
					+ (\frac p2-1)\int\H\rnu\vert\zFundtrf\vert^p\d\mug \\
	\le{}& C\,\exp((\frac2p-\decay)\Hradius)\Vert\zFundtrf\Vert_{\Lp^p(\M)}^{p-1}
\end{align*}
where we have used $p>2$, $\rnu\ge 0$, and $\H\relax<0$. By the chain rule, we know $\partial*_\Hradius\Vert\zFundtrf\Vert_{\Lp^p(\M)}^p=p\Vert\zFundtrf\Vert_{\Lp^p(\M)}^{p-1}\partial*_\Hradius\Vert\zFundtrf\Vert_{\Lp^p(\M)}$ and integration therefore proves
\[ \Vert\zFundtrf\Vert_{\Lp^p(\M)} \le C\,\exp((\frac2p-\decay)\Hradius). \]
This implies~\eqref{Stronger_estimates_on_zFundtrf} as the estimate on $\hlevi\zFundtrf$ is already known by \refth{Regularity_Surfaces}.
\end{proof}

Now, we use our estimates on the second fundamental form to prove that the metrics of the CMC-leaves approach the round metric.
\begin{lemma}[Better decay estimates on the metric]\labelth{Better_estimates_metric}
For all constants $\decay\ge\frac52+\outve\in\interval{\frac52}*3$, $\scdecay\ge3+\outve$, $\eta<4$, $p>2$, $\c\ge0$, and $q\in\interval2*p$ with $q<\infty$, there exist two constants $\Hradius_0=\Cof{\Hradius_0}[\decay][\scdecay][\eta][\c]$ and $C=\Cof[\decay][\scdecay][\eta][\c][q]$ with the following property:

Let $\mathcal M=\{\M<\Hradius>\}_{\Hradius>\Hradius_1}$ be a locally unique \Wkpro^{2,p}_{\decay,\scdecay}(\c,\eta) foliation with uniformly timelike Ricci-mass, where $\Hradius_1\ge\Hradius_0$. There exists a diffeomorphism $\Phi:\interval{\Hradius_1}\infty\times\sphere^2\to\outM\setminus\outsymbol K$ as in \refth{Local_Estimates_rnu} with $\partial*_\Hradius\Phi=\rnu\,\nu$ ans
\[ \g<\Hradius>'\to\sphg \text{ in }\Wkp^{2,p}(\sphere^2), \qquad
	\Vert\partial[\Hradius]@{(\g<\Hradius>')} - \H(1-\rnu)\g<\Hradius> '\Vert_{\Wkp^{1,q}(\sphere,\sphg)} \le C\,\exp({-}\decay\Hradius), \]
where $\g<\Hradius>':=\sinh(\Hradius)^{{-}2}\pullback{\Phi<\Hradius>}\g<\Hradius>$ and $\Phi<\Hradius>:=\Phi(\Hradius,\,{\cdot}\,)$.
\end{lemma}
\begin{proof}
Again, we assume without loss of generality $\decay=\frac52+\outve$, $\scdecay=3+\outve$, $\dot\bigcup_{\Hradius>\Hradius_0}\M<\Hradius>=\outM$, $p=q$, and $\Hradius_0$ is so large, that we can apply all the results proven so far on every $\M<\Hradius>$ and $\mathcal M$. In particular, \refth{Local_Estimates_rnu} and \refth{Ricci_mass_well-defined} imply that the surfaces are pairwise disjoint. Thus, we can define a global vector field $X$ such that $X|_{\M<\Hradius>}=\rnu\nu$, where $\rnu$ is as in \refth{Local_Estimates_rnu} for any parametrization $\Phi$. Fix some arbitrary $\varsigma>\Hradius_0$ and some arbitrary parametrization $\psi<\varsigma>:\sphere^2\to\M<\varsigma>$ and denote by $\gamma_p:\interval{\Hradius_0}\infty\to\outM$ the integral curve to $X$ starting in $\psi<\varsigma>(p)=\gamma_p(\varsigma)$. Per definition $\gamma_p(\Hradius)\in\M<\Hradius>$. Thus, $\Phi:\interval{\Hradius_0}\infty\times\sphere^2\to\outM:(\Hradius,p)\mapsto\gamma_p(\Hradius)$ is a diffeomorphism as in \refth{Local_Estimates_rnu} and satisfies $\partial*_\Hradius\Phi=\rnu\nu$.

We note that
\[ \lieD{\rnu\nu}(\sinh(\Hradius)^{{-}2}\g) = \frac{1-\rnu}{\sinh(\Hradius)^2}\H\g + 2\sinh(\Hradius)^{{-}2}\zFundtrf \]
and therefore
\begin{equation*}\labeleq{derivative_metric}
 \Vert\partial[\Hradius]@{(\g<\Hradius>')} - \H(1-\rnu)\g<\Hradius>'\Vert_{\Wkp^{1,p}(\M,\g<\Hradius>')}
		\le C\exp({-}\decay\Hradius),
\end{equation*}
where $\g<\Hradius>':=\sinh(\Hradius)^{{-}2}\g<\Hradius>$ denotes the rescaled metric. On the left hand side, the norm with respect to $\g<\Hradius>'\approx\sphg$ is used and not the one with respect to $\g<\Hradius>\approx\sinh(\Hradius)^2\sphg*$ which gives the additional factor $\sinh(\Hradius)^2$. As $\Vert\rnu-1\Vert_{\Wkp^{2,p}(\M<\Hradius>)}$ is integrable, this proves that $\g<\Hradius>'$ converges in $\Wkp^{1,p}(\sphere^2)$ to a metric $\g<\infty>$ on $\sphere^2$. Using
\begin{equation*}\labeleq{weak-sc}
 \sc<\Hradius>' = {-}\laplace[\sphg]\,(\tr[\sphg]\,\g<\Hradius>') + \div[\sphg]\,\div[\sphg]\,\g<\Hradius>' + P_{\g<\Hradius>*'}(\levi[\sphg]*\g<\Hradius>',\levi[\sphg]*\g<\Hradius>'),
\end{equation*}
we see that the scalar curvature $\sc<\infty>'$ of $\g<\infty>'$ is well-defined in $\Wkp^{{-}1,\frac p{p-1}}(\sphere^2)$ and the $\Wkp^{{-}1,\frac p{p-1}}(\sphere^2)$-limit of $\sc<\Hradius>'$. As
\[ \Vert\sc<\Hradius>'-2\Vert_{\Lp^p(\sphere^2)} \le C\,\exp((2-\decay)\Hradius) \xrightarrow{\Hradius\to\infty} 0, \]
this implies that $\sc<\infty>'\equiv2$ in $\Lp^p(\sphere^2)$ and therefore $\g<\infty>$ is---up to reparametrization---the round metric, see for example \cite[Thm~A.1]{nerz2015GeometricCharac}. By changing the initial data $(s,\varphi_s)$ (see construction of $\Phi$) to $(s,\varphi_s\circ\psi)$, we can assume $\psi=\id|_{\sphere^2}$. By \cite[Thm~A.1]{nerz2015GeometricCharac}, this implies furthermore $\g<\Hradius>'\xrightarrow{\Hradius\to\infty}\sphg$ in $\Wkp^{2,p}(\sphere^2)$.
\end{proof}
\begin{remark}\labelth{diffeomorphism_ids_at_infinity}
The above proves that the diffeomorphism $\Phi:\interval{\Hradius_0}\infty\times\sphere^2\to\outM$ with $\partial*_\Hradius\Phi=\rnu\nu$ is uniquely determined by its limit value $\lim_\Hradius\Phi(\Hradius,\,{\cdot}\,)$ which can (in a well-defined sense) be interpreted as diffeomorphism of $\sphere^2$. In particular if we think of $\M<\Hradius>$ as almost-round spheres in $\hyperbolicspace$, it does not seem to be surprising that the \lq initial\rq\ value problem
\[ \partial[\Hradius]@\Phi = \rnu\,\nu, \qquad
		\Phi(\Hradius,\,{\cdot}\,)\xrightarrow{\Hradius\to\infty}\varphi \]
has a unique solution---where $\varphi:\sphere^2\to\sphere^2$ is an arbitrary diffeomorphism. However, this is a crucial step in the proof of the main theorem. Furthermore, it is a \lq CMC-version\rq\ of the coordinate result that the change from one asymptotically hyperbolic chart to another is (asymptotically) an isometry of the hyperbolic space which is characterized by an conformal map of the sphere at infinity, see~\cite{chrusciel2003mass}.
\end{remark}

\section{Proof of \texorpdfstring{\refth{MainTheorem}}{of the first main theorem}}\label{ProofOfMainTheorem}

\def\paragraph#1{\smallskip\par\noindent\textbf{#1}:\ }
We first explain the idea of the proof of \refth{MainTheorem}. Then, we explain the idea of each single step in more detail before finally proving \refth{MainTheorem}.
\paragraph{Overview of the proof of \refth{MainTheorem}}
In the first step, we fix a large radius~$\varsigma$ and construct a \Ckah^0_{\decay-2,\scdecay} chart of the interior of $\M<\varsigma>$ (outside of a compact set) by mapping each CMC-leaf to a geodesic sphere in the hyperbolic space, see Figure~\ref{Image__round} \vpageref[below]{Image__round}. In the second step, we repeat the construction of the first step but modify the construction a little bit to get a $\Ck^0_{\decay,\scdecay}$- and \Ckah^1_{\decay-1,\scdecay} chart of the interior of $\M<\varsigma>$ (outside of a compact set), see Figures~\ref{Image__non-round} and~\ref{Image__with-graph} on pages~\pageref{Image__non-round} and~\pageref{Image__with-graph}. We then prove several sharp estimates on the difference between the image of every CMC-leaf along the first and second chart. In the third step, we prove that we can take the limit as $\varsigma\to\infty$ to get two global charts of $\outM$ (outside of a compact set) and that all proven inequalities remain true.  Finally in the fourth step, we use the regularity of the Ricci curvature to conclude that the constructed chart is in fact \Ckah^2_{\decay,\scdecay}.\smallskip

To simplify notation, we suppress the dependency on $\varsigma$ in the first two steps and identify $\M<\Hradius>$ with $\{\Hradius\}\times\sphere^2$ using \refth{Better_estimates_metric}.

\paragraph{Overview of the first two steps}\label{OverviewFirstSteps}
To construct the mentioned charts in the first two steps, we fix a nice parametrization $\varphi<\varsigma>:\hsphere^2_\varsigma(\centerz<\varsigma>)\to\M<\varsigma>$ of the (large) CMC-leaf $\M<\varsigma>$, see \refth{Better_estimates_metric}. We then construct the chart of the interior of $\M<\varsigma>$ by \lq integrating\rq\ a chosen function along $\Hradius$, \ie for the first and second step choose a solution $\Phi:\interval\Hradius_0\varsigma\times\sphere^2\to\hyperbolicspace$ to $\partial*_\Hradius\Phi=f\,\hnu$ and $\Phi(\varsigma,\emptyarg)=\varphi<\varsigma>^{{-}1}$, where the function $f$ depends on the step and is defined using the real lapse function $\rnu$.
\begin{Figure}<-.2em>[Image__round]%
 {  \begin{tikzpicture}[scale=.45]
		\def\colorlar{black!0!red}			\expandafter\let\csname Rcolor1\endcsname\colorlar
\def\colormid{black!0!blue}			\expandafter\let\csname Rcolor2\endcsname\colormid
\def\colorsma{black!15!green}		\expandafter\let\csname Rcolor3\endcsname\colorsma
\def\coloralar{black!0!red}			\expandafter\let\csname Gcolor1\endcsname\coloralar
\def\coloramid{black!30!blue}		\expandafter\let\csname Gcolor2\endcsname\coloramid
\def\colorasma{black!35!green}	\expandafter\let\csname Gcolor3\endcsname\colorasma
\def\stylelar{}									\expandafter\let\csname Rstyle1\endcsname\stylelar	
\def\stylemid{densely dashed}		\expandafter\let\csname Rstyle2\endcsname\stylemid
\def\stylesma{}									\expandafter\let\csname Rstyle3\endcsname\stylesma
\def\stylealar{}								\expandafter\let\csname Gstyle1\endcsname\stylealar	
\def\styleamid{}								\expandafter\let\csname Gstyle2\endcsname\styleamid
\def\styleasma{dashed}					\expandafter\let\csname Gstyle3\endcsname\styleasma
\def\colorcen{purple}
\def\ptssize{.15}
\def\pointerth{thick}
%
\def\coloramid{red!35!blue}
\def\colorasma{black!50!green}
\def\colorrnu{black}							\def\stylernu{}
\def\colorarnu{orange}						\def\stylearnu{densely dashed}
\def\colorgraph{green}						\def\stylegraph{}
\def\coloragraph{blue}						\def\stylegraph{}
\def\ptssize{.15}
\def\factormid{.9}
\def\factorsma{.75}

\footnotesize
\coordinate (center1) at (0,0);
\coordinate (center2) at (-.5,1.5);
\coordinate (center3) at (-.25,2);
\pgfmathsetseed{3}
\coordinate (G2-P1) at ($(center2)+\factormid*(rnd,rnd)+(-9.6,0)$);	\coordinate (G2-P2) at ($(center2)+\factormid*(rnd,rnd)+(-6.79,6.79)$);
\coordinate (G2-P3) at ($(center2)+\factormid*(rnd,rnd)+(0,9.6)$);	\coordinate (G2-P4) at ($(center2)+\factormid*(rnd,rnd)+(6.79,6.79)$);
\coordinate (G2-P5) at ($(center2)+\factormid*(rnd,rnd)+(9.6,0)$);	\coordinate (G2-P6) at ($(center2)+\factormid*(rnd,rnd)+(6.79,-6.79)$);
\coordinate (G2-P7) at ($(center2)+\factormid*(rnd,rnd)+(0,-9.6)$);	\coordinate (G2-P8) at ($(center2)+\factormid*(rnd,rnd)+(-6.79,-6.79)$);
\coordinate (G3-P1) at ($(center3)+\factorsma*(rnd,rnd)+(-7.2,0)$);	\coordinate (G3-P2) at ($(center3)+\factorsma*(rnd,rnd)+(-5.09,5.09)$);
\coordinate (G3-P3) at ($(center3)+\factorsma*(rnd,rnd)+(0,7.2)$);	\coordinate (G3-P4) at ($(center3)+\factorsma*(rnd,rnd)+(5.09,5.09)$);
\coordinate (G3-P5) at ($(center3)+\factorsma*(rnd,rnd)+(7.2,0)$);	\coordinate (G3-P6) at ($(center3)+\factorsma*(rnd,rnd)+(5.09,-5.09)$);
\coordinate (G3-P7) at ($(center3)+\factorsma*(rnd,rnd)+(0,-7.2)$);	\coordinate (G3-P8) at ($(center3)+\factorsma*(rnd,rnd)+(-5.09,-5.09)$);
%
\begin{pgfinterruptboundingbox}
	\begin{scope}
		\clip (-100,-4) -- (-100,100) -- (100,100) -- (100,-4);
		\path [name path global=G1] (center1) circle (12);
		\path [name path global=G2] plot [smooth cycle] coordinates
			{ (G2-P1) (G2-P2) (G2-P3) (G2-P4) (G2-P5) (G2-P6) (G2-P7) (G2-P8) };
		\path [name path global=G3] plot [smooth cycle] coordinates
			{ (G3-P1) (G3-P2) (G3-P3) (G3-P4) (G3-P5) (G3-P6) (G3-P7) (G3-P8) };
		\path [name path global=R1] (center1) circle (12);
		\path [name path global=R2] (center2) circle (9.6);
		\path [name path global=R3] (center3) circle (7.2);
		\path [name path global=l1	]	(center1) -- ($(center1)+(6,12)$);
		\path [name path global=l2	]	(center1) -- ($(center1)+(-10,14)$);
		\path [name path global=l3	]	(center1) -- ($(center1)+(-12,0)$);
		\path [name path global=l4	]	(center1) -- ($(center1)+(-14,-4)$);
		\path [name path global=l5	]	(center1) -- ($(center1)+(12,0)$);
		\path [name path global=l6	]	(center1) -- ($(center1)+(-12,3.5)$);
		\path [name path global=l7	]	(center1) -- ($(center1)+(12,2)$);
		\path [name path global=l1-	]	(center1) -- ($(center1)+(6,11.9)$);
		\path [name path global=l2-	]	(center1) -- ($(center1)+(-10,14.1)$);
		\path [name path global=l3-	]	(center1) -- ($(center1)+(-12,.1)$);
		\path [name path global=l4-	]	(center1) -- ($(center1)+(-14,-3.9)$);
		\path [name path global=l5-	]	(center1) -- ($(center1)+(12,-.1)$);
		\path [name path global=l6-	]	(center1) -- ($(center1)+(-12,3.6)$);
		\path [name path global=l7-	]	(center1) -- ($(center1)+(12,1.9)$);
		\path [name path global=l1+	]	(center1) -- ($(center1)+(6,12.1)$);
		\path [name path global=l2+	]	(center1) -- ($(center1)+(-10,13.9)$);
		\path [name path global=l3+	]	(center1) -- ($(center1)+(-12,-.1)$);
		\path [name path global=l4+	]	(center1) -- ($(center1)+(-14,-4.1)$);
		\path [name path global=l5+	]	(center1) -- ($(center1)+(12,.1)$);
		\path [name path global=l6+	]	(center1) -- ($(center1)+(-12,3.4)$);
		\path [name path global=l7+	]	(center1) -- ($(center1)+(12,2.1)$);
		\coordinate (R1-E) at ($(center1)+(12,0)$);
		\coordinate (R1-W) at ($(center1)-(12,0)$);
		\coordinate (R2-E) at ($(center2)+(9.6,0)$);
		\coordinate (R2-W) at ($(center2)-(9.6,0)$);
		\coordinate (R3-E) at ($(center3)+(7.2,0)$);
		\coordinate (R3-W) at ($(center3)-(7.2,0)$);
		\foreach \ln in {1,...,7} {
			\path [name intersections={of= G1 and l\ln}] coordinate (G1-\ln) at (intersection-1);
			\path [name intersections={of= G2 and l\ln}] coordinate (G2-\ln) at (intersection-1);
			\path [name intersections={of= G2 and l\ln-}] coordinate (G2-\ln-) at (intersection-1);
			\path [name intersections={of= G2 and l\ln+}] coordinate (G2-\ln+) at (intersection-1);
			\path [name intersections={of= R1 and l\ln}] coordinate (R1-\ln) at (intersection-1);
			\path [name intersections={of= R2 and l\ln}] coordinate (R2-\ln) at (intersection-1);
			\path [name path=G2orth] (G2-\ln) -- ($(G2-\ln-)!10cm!90:(G2-\ln+)$);
			\path [name path=R2orth] (R2-\ln) -- (center2);
			\path [name intersections={of= G3 and G2orth}] coordinate (G3-\ln) at (intersection-1);
			\path [name intersections={of= R3 and R2orth}] coordinate (R3-\ln) at (intersection-1);
			\path [name path=R3orth] ($(center3)+2*($(R3-\ln)-(center3)$)$) -- (center3);
			\path [name intersections={of= G2 and R2orth}] coordinate (R2-G2-\ln) at (intersection-1);
			\path [name intersections={of= G3 and R3orth}] coordinate (R3-G3-\ln) at (intersection-1);
		}
	\end{scope}
	\path [name path global=clip1] (-100,-4) -- (100,-4);
	\path [name path global=clip2] (-100,0) -- (100,0);
	\path [name path global=clip3] (0,0) -- (0,100);
	\path [name intersections={of=R1 and clip1}] coordinate (SW) at (intersection-1) coordinate (SE) at (intersection-2);
	\path [name intersections={of=R1 and clip2}] coordinate  (W) at (intersection-1) coordinate  (E) at (intersection-2);
	\path [name intersections={of=R1 and clip3}] coordinate (N)  at (intersection-1);
\end{pgfinterruptboundingbox}
\makeatletter
\NewDocumentCommand\drawasphere{smm}{%
	\IfBooleanTF{#1}{%
		\draw [thick,draw opacity=0] (SW) -- (SE) -- ($(E)+(0,.1)$) -- ($(E)-(0,.1)$)%
			-- ($(N)-(.1,0)$) -- ($(N)+(.1,0)$) -- ($(W)-(0,.1)$) -- ($(W)+(0,.1)$);%
	}\relax%
	\begin{pgfinterruptboundingbox}
		\begin{scope}
			\clip (-100,-3.25) -- (-100,100) -- (100,100) -- (100,-3.25);
			\draw [\csname #2color#3\endcsname,use path=#2#3,thick,\csname #2style#3\endcsname] (0,0);
			\ifx #2R\drawasphere@addtoround#3\else\ifx 1#3\drawasphere@addtoround1\fi\fi%
		\end{scope}
		\begin{scope}
			\clip (-100,-4) -- (-100,-3.23) -- (100,-3.23) -- (100,-4);
			\draw [\csname #2color#3\endcsname,use path=#2#3,thick,densely dotted] (0,0);
		\end{scope}
	\end{pgfinterruptboundingbox}
}
\gdef\drawasphere@addtoround#1{%
	\ifx#11\drawasphere@addtoround@ 1\varsigma{12}{-30}\else\ifx#12\drawasphere@addtoround@ 2\Hradius{9.6}{-30}\else\ifx#13 \drawasphere@addtoround@ 3s{7.2}{-30}\fi\fi\fi%
}
\gdef\drawasphere@addtoround@#1#2#3#4{%
	\draw [\csname Rcolor#1\endcsname,fill] (center#1) circle (\ptssize);%
	\coordinate (center-r) at ($(center#1)!#3!#4:($(center#1)-(1,0)$)$);%
	\draw [|<->|,black!75!white] (center#1) -- (center-r);%
	\draw [black!75!white] ($(center#1)!.4!(center-r)$) node [fill=white] () {\textcolor{black!75!white}{$#2$}};%
}
\makeatother
\NewDocumentCommand\drawspheres{s}{\IfBooleanTF{#1}{\drawasphere R1}{\drawasphere*R1} \drawasphere R2 \drawasphere R3}
\NewDocumentCommand\drawgraphs{s}{ \IfBooleanTF{#1}{\drawasphere G1}{\drawasphere*G1} \drawasphere G2 \drawasphere G3}
		\draw (R1-4.south east) node [right] () {\textcolor{\colorlar}{$\varphi<\varsigma>^{{-}1}(p)=\roPhi(\varsigma,p)$}};
		\draw (R1-1) node [right]	() {\textcolor{\colorlar}{$\varphi<\varsigma>^{{-}1}(p')=\roPhi(\varsigma,p')$}};
		\draw (R2-4) node [above]	() {\textcolor{\colormid}{$\roPhi(\Hradius,p)$}};
		\draw (R2-1) node [right]	() {\textcolor{\colormid}{$\roPhi(\Hradius,p')$}};
		\draw (R3-4) node [right]	() {\textcolor{\colorsma}{$\roPhi(s,p)$}};
		\draw (R3-1) node [left]	() {\textcolor{\colorsma}{$\roPhi(s,p')$}};
		\draw (R1-W) node [left]	() {\textcolor{\colorlar}{${}_\varsigma\!W$}};
		\draw (R1-E) node [right]	() {\textcolor{\colorlar}{${}_\varsigma\!E$}};
		\draw [\colorsma,fill] (R3-W) circle (\ptssize);
		\draw [\colorsma,fill] (R3-E) circle (\ptssize);
		\draw (R3-W) node [left]	() {\textcolor{\colorsma}{${}_s\!W$}};
		\draw (R3-E) node [right]	() {\textcolor{\colorsma}{${}_s\!E$}};
		\draw (R3-3) node [right] () {\textcolor{\colorsma}{$\Phi(s,\varphi({}_\varsigma\!W))$}};%
		\draw (R3-5) node [left] () {\textcolor{\colorsma}{$\Phi(s,\varphi({}_\varsigma\!E))$}};%
		\begin{scope}
			\drawspheres
			\foreach \ln in {1,...,5} {
				\draw [\colorlar,fill] (R1-\ln) circle (\ptssize);
				\draw [\colormid,fill] (R2-\ln) circle (\ptssize);
				\draw [\colorsma,fill] (R3-\ln) circle (\ptssize);
				\draw [|->,\colorrnu,\stylernu,\pointerth] (R1-\ln) -- (R2-\ln);
				\draw [|->,\colorrnu,\stylernu,\pointerth] (R2-\ln) -- (R3-\ln);
			}
			\draw [->,thick,\colorcen] (center1) -- (center2);
			\draw [->,thick,\colorcen] (center2) -- (center3);
		\end{scope}
	\end{tikzpicture}
\par\noindent%
 \hbox{%
	\begin{tikzpicture}[scale=.75]\footnotesize\def\colorlar{black!0!red}			\expandafter\let\csname Rcolor1\endcsname\colorlar
\def\colormid{black!0!blue}			\expandafter\let\csname Rcolor2\endcsname\colormid
\def\colorsma{black!15!green}		\expandafter\let\csname Rcolor3\endcsname\colorsma
\def\coloralar{black!0!red}			\expandafter\let\csname Gcolor1\endcsname\coloralar
\def\coloramid{black!30!blue}		\expandafter\let\csname Gcolor2\endcsname\coloramid
\def\colorasma{black!35!green}	\expandafter\let\csname Gcolor3\endcsname\colorasma
\def\stylelar{}									\expandafter\let\csname Rstyle1\endcsname\stylelar	
\def\stylemid{densely dashed}		\expandafter\let\csname Rstyle2\endcsname\stylemid
\def\stylesma{}									\expandafter\let\csname Rstyle3\endcsname\stylesma
\def\stylealar{}								\expandafter\let\csname Gstyle1\endcsname\stylealar	
\def\styleamid{}								\expandafter\let\csname Gstyle2\endcsname\styleamid
\def\styleasma{dashed}					\expandafter\let\csname Gstyle3\endcsname\styleasma
\def\colorcen{purple}
\def\ptssize{.15}
\def\pointerth{thick}
%
\def\coloramid{red!35!blue}
\def\colorasma{black!50!green}
\def\colorrnu{black}							\def\stylernu{}
\def\colorarnu{orange}						\def\stylearnu{densely dashed}
\def\colorgraph{green}						\def\stylegraph{}
\def\coloragraph{blue}						\def\stylegraph{}
\def\ptssize{.15}
\def\factormid{.9}
\def\factorsma{.75}

	 \begin{scope}
	  \clip (-.27,-.04) -- (-.27,-.18) -- (.27,-.18) -- (.27,-.04);
	  \draw [\colorlar,thick,densely dotted] (0,0) circle (.25);
	 \end{scope}
	  \clip (-.27,.27) -- (-.27,-.05) -- (.27,-.05) -- (.27,.27);
	  \draw [\colorlar,\stylelar,thick] (0,0) circle (.25);
	\end{tikzpicture}}\ :\
 $\roPhi(\M<\varsigma>)=\hsphere^2_\varsigma(\centerz<\varsigma>)$\hspace*{1.8em}
 \hbox{%
	\begin{tikzpicture}[scale=.75]\footnotesize\def\colorlar{black!0!red}			\expandafter\let\csname Rcolor1\endcsname\colorlar
\def\colormid{black!0!blue}			\expandafter\let\csname Rcolor2\endcsname\colormid
\def\colorsma{black!15!green}		\expandafter\let\csname Rcolor3\endcsname\colorsma
\def\coloralar{black!0!red}			\expandafter\let\csname Gcolor1\endcsname\coloralar
\def\coloramid{black!30!blue}		\expandafter\let\csname Gcolor2\endcsname\coloramid
\def\colorasma{black!35!green}	\expandafter\let\csname Gcolor3\endcsname\colorasma
\def\stylelar{}									\expandafter\let\csname Rstyle1\endcsname\stylelar	
\def\stylemid{densely dashed}		\expandafter\let\csname Rstyle2\endcsname\stylemid
\def\stylesma{}									\expandafter\let\csname Rstyle3\endcsname\stylesma
\def\stylealar{}								\expandafter\let\csname Gstyle1\endcsname\stylealar	
\def\styleamid{}								\expandafter\let\csname Gstyle2\endcsname\styleamid
\def\styleasma{dashed}					\expandafter\let\csname Gstyle3\endcsname\styleasma
\def\colorcen{purple}
\def\ptssize{.15}
\def\pointerth{thick}
%
\def\coloramid{red!35!blue}
\def\colorasma{black!50!green}
\def\colorrnu{black}							\def\stylernu{}
\def\colorarnu{orange}						\def\stylearnu{densely dashed}
\def\colorgraph{green}						\def\stylegraph{}
\def\coloragraph{blue}						\def\stylegraph{}
\def\ptssize{.15}
\def\factormid{.9}
\def\factorsma{.75}

	 \begin{scope}
	  \clip (-.27,-.04) -- (-.27,-.18) -- (.27,-.18) -- (.27,-.04);
	  \draw [\colormid,thick,densely dotted] (0,0) circle (.25);
	 \end{scope}
	  \clip (-.27,.27) -- (-.27,-.05) -- (.27,-.05) -- (.27,.27);
	  \draw [\colormid,\stylemid,thick] (0,0) circle (.25);
	\end{tikzpicture}}\ :\
 $\roPhi(\M<\Hradius>)=\hsphere^2_\varsigma(\centerz<\Hradius>)$\hspace*{1.8em}
 \hbox{%
	\begin{tikzpicture}[scale=.75]\footnotesize\def\colorlar{black!0!red}			\expandafter\let\csname Rcolor1\endcsname\colorlar
\def\colormid{black!0!blue}			\expandafter\let\csname Rcolor2\endcsname\colormid
\def\colorsma{black!15!green}		\expandafter\let\csname Rcolor3\endcsname\colorsma
\def\coloralar{black!0!red}			\expandafter\let\csname Gcolor1\endcsname\coloralar
\def\coloramid{black!30!blue}		\expandafter\let\csname Gcolor2\endcsname\coloramid
\def\colorasma{black!35!green}	\expandafter\let\csname Gcolor3\endcsname\colorasma
\def\stylelar{}									\expandafter\let\csname Rstyle1\endcsname\stylelar	
\def\stylemid{densely dashed}		\expandafter\let\csname Rstyle2\endcsname\stylemid
\def\stylesma{}									\expandafter\let\csname Rstyle3\endcsname\stylesma
\def\stylealar{}								\expandafter\let\csname Gstyle1\endcsname\stylealar	
\def\styleamid{}								\expandafter\let\csname Gstyle2\endcsname\styleamid
\def\styleasma{dashed}					\expandafter\let\csname Gstyle3\endcsname\styleasma
\def\colorcen{purple}
\def\ptssize{.15}
\def\pointerth{thick}
%
\def\coloramid{red!35!blue}
\def\colorasma{black!50!green}
\def\colorrnu{black}							\def\stylernu{}
\def\colorarnu{orange}						\def\stylearnu{densely dashed}
\def\colorgraph{green}						\def\stylegraph{}
\def\coloragraph{blue}						\def\stylegraph{}
\def\ptssize{.15}
\def\factormid{.9}
\def\factorsma{.75}

	 \begin{scope}
	  \clip (-.27,-.04) -- (-.27,-.18) -- (.27,-.18) -- (.27,-.04);
	  \draw [\colorsma,thick,densely dotted] (0,0) circle (.25);
	 \end{scope}
	  \clip (-.27,.27) -- (-.27,-.05) -- (.27,-.05) -- (.27,.27);
	  \draw [\colorsma,\stylesma,thick] (0,0) circle (.25);
	\end{tikzpicture}}\ :\
 $\roPhi(\M<s>)=\hsphere^2_s(\centerz<s>)$%
 \smallskip\par\noindent%
 \raisebox{.25em}{\hbox{%
	\begin{tikzpicture}[scale=.75]\footnotesize\def\colorlar{black!0!red}			\expandafter\let\csname Rcolor1\endcsname\colorlar
\def\colormid{black!0!blue}			\expandafter\let\csname Rcolor2\endcsname\colormid
\def\colorsma{black!15!green}		\expandafter\let\csname Rcolor3\endcsname\colorsma
\def\coloralar{black!0!red}			\expandafter\let\csname Gcolor1\endcsname\coloralar
\def\coloramid{black!30!blue}		\expandafter\let\csname Gcolor2\endcsname\coloramid
\def\colorasma{black!35!green}	\expandafter\let\csname Gcolor3\endcsname\colorasma
\def\stylelar{}									\expandafter\let\csname Rstyle1\endcsname\stylelar	
\def\stylemid{densely dashed}		\expandafter\let\csname Rstyle2\endcsname\stylemid
\def\stylesma{}									\expandafter\let\csname Rstyle3\endcsname\stylesma
\def\stylealar{}								\expandafter\let\csname Gstyle1\endcsname\stylealar	
\def\styleamid{}								\expandafter\let\csname Gstyle2\endcsname\styleamid
\def\styleasma{dashed}					\expandafter\let\csname Gstyle3\endcsname\styleasma
\def\colorcen{purple}
\def\ptssize{.15}
\def\pointerth{thick}
%
\def\coloramid{red!35!blue}
\def\colorasma{black!50!green}
\def\colorrnu{black}							\def\stylernu{}
\def\colorarnu{orange}						\def\stylearnu{densely dashed}
\def\colorgraph{green}						\def\stylegraph{}
\def\coloragraph{blue}						\def\stylegraph{}
\def\ptssize{.15}
\def\factormid{.9}
\def\factorsma{.75}

	  \draw [|->,\colorrnu,\stylernu,thick] (0,0) -- (1,0);
	\end{tikzpicture}\vspace*{.5em}}}\ :\
	$\partial*_\Hradius\roPhi=(1{\,+\,}\roboost{\rnu})\,\ronu$\hspace*{1.8em}
 \raisebox{.25em}{\hbox{%
	\begin{tikzpicture}[scale=.75]\footnotesize\def\colorlar{black!0!red}			\expandafter\let\csname Rcolor1\endcsname\colorlar
\def\colormid{black!0!blue}			\expandafter\let\csname Rcolor2\endcsname\colormid
\def\colorsma{black!15!green}		\expandafter\let\csname Rcolor3\endcsname\colorsma
\def\coloralar{black!0!red}			\expandafter\let\csname Gcolor1\endcsname\coloralar
\def\coloramid{black!30!blue}		\expandafter\let\csname Gcolor2\endcsname\coloramid
\def\colorasma{black!35!green}	\expandafter\let\csname Gcolor3\endcsname\colorasma
\def\stylelar{}									\expandafter\let\csname Rstyle1\endcsname\stylelar	
\def\stylemid{densely dashed}		\expandafter\let\csname Rstyle2\endcsname\stylemid
\def\stylesma{}									\expandafter\let\csname Rstyle3\endcsname\stylesma
\def\stylealar{}								\expandafter\let\csname Gstyle1\endcsname\stylealar	
\def\styleamid{}								\expandafter\let\csname Gstyle2\endcsname\styleamid
\def\styleasma{dashed}					\expandafter\let\csname Gstyle3\endcsname\styleasma
\def\colorcen{purple}
\def\ptssize{.15}
\def\pointerth{thick}
%
\def\coloramid{red!35!blue}
\def\colorasma{black!50!green}
\def\colorrnu{black}							\def\stylernu{}
\def\colorarnu{orange}						\def\stylearnu{densely dashed}
\def\colorgraph{green}						\def\stylegraph{}
\def\coloragraph{blue}						\def\stylegraph{}
\def\ptssize{.15}
\def\factormid{.9}
\def\factorsma{.75}

	  \draw [|->,\colorcen,thick] (0,0) -- (1,0);
	\end{tikzpicture}\vspace*{.5em}}}\ :\
	$\partial*_\Hradius\centerz<\Hradius>$}%
 [Round images of the CMC-leaves and the lapse function]%
 {$\roPhi=\roPhi[\varsigma]$-images of the CMC-leaves and the lapse function}
 \emph{Construction:} Start with a non-isometric embedding $\varphi<\varsigma>^{{-}1}$ of $\M<\varsigma>$ to a hyperbolic sphere $\hsphere^2_\varsigma(\centerz<\varsigma>)$, \ie the outer (red) sphere denotes $\varphi^{{-}1}(\M<\varsigma>)=\roPhi(\M<\varsigma>)$. Now, choose $\roPhi$ with $\roPhi(\varsigma,\emptyarg)=\varphi<\varsigma>^{{-}1}$ and $\partial*_\Hradius\roPhi=\roboost{\rnu}\ronu$, where $\roboost{\rnu}$ and $\ronu$ denote the boosting part of the ($\roPhi$-pushforward of the) lapse function $\rnu$  with respect to the metric induced on $\roPhi(\emptyarg,\sphere^2)$ by $\houtg*$ and the outer unit normal with respect to $\houtg*$. This leads to $\roPhi(\M<\varsigma>)$ (the middle, dashed, blue sphere) and from there we get $\Phi(\M<s>)$ (the green one).\smallskip\par\noindent
 \emph{Note:} Even if $\M<\varsigma>$ and $\M<s>$ are exactly round, \ie have constant Gau\ss\ curvature, and $\varphi<\varsigma>$ is an isometry, the map $\roPhi(s,\varphi(\emptyarg))$ of $\M<s>=\{s\}\times\sphere^2$ to $\hsphere^2_\Hradius(\centerz<s>)\hookrightarrow\hyperbolicspace^3$ can still be non-trivial as we \lq boosted\rq\ its coordinates. This can easily be seen in the picture by looking at $\roPhi(\M<s>)$ (the inner, green sphere): We see that its \lq northern hemisphere\rq\ (the upper part of $\roPhi(\M<s>)$ bounded by ${}_s\!W$ and ${}_s\!E$) is by far smaller than the $\roPhi$-image (bounded by $\roPhi(s,\varphi({}_\varsigma\!W))$ and $\roPhi(s,\varphi({}_\varsigma\!E))$) of the northern hemisphere of $\roPhi(\M<r>)$. More precisely, $\roPhi(s,\emptyarg):\sphere^2\to\hsphere^2_s(\centerz<s>)$ is a conformal map but (in most cases) non-trivial.\smallskip\par\noindent%
 \emph{Note:} As we only want to describe the idea, we simplify the images by drawing an Euclidean setting, \eg we use coordinate spheres instead of the coordination of hyperbolic spheres. Furthermore, the arrows actually do not picture $\partial*\Hradius\Phi=\roboost\,\rnu$, but its non-infinitesimal version, \ie $\houtexp^{{-}1}_{\Phi(\varsigma,\emptyarg)}(\Phi(\Hradius,\emptyarg))$ etc. The same is true for the following pictures.
\end{Figure}
\paragraph{The first step}
In the first step, \pagebreak[1]we only look at the boosting part of $\rnu$ and call the above solution $\roPhi$, \pagebreak[1]\ie $f$ is here the boosting part $\roboost<\Hradius>{\rnu<\Hradius>}$ of the lapse function $\rnu$ with respect to the metric $\rog<\Hradius>$ being the pullback along $\roPhi(\Hradius,\emptyarg)$ of the one induced on $\roPhi(\Hradius,\sphere^2)$ by $\houtg$, see Figure~\ref{Image__round} \vpageref[above]{Image__round}.\pagebreak[1]

As the initial data $\roPhi(\varsigma,\emptyarg)=\varphi<\varsigma>^{{-}1}$ is a chosen round (hyperbolic) sphere $\sphere^2_\varsigma(\centerz<\varsigma>)$ and the \emph{boosting functions} on round spheres correspond to boosts of the sphere, see \refth{Boosting_deformations}, this implies that $\roPhi$ maps every $\M<\Hradius>$ to a round sphere $\hsphere^2_\Hradius(\centerz<\Hradius>)$ in the hyperbolic space. This explains the upper-index $\roundr$ being short for \lq round\rq. Now, we can prove that the metric $\g<\Hradius>$ of $\M<\Hradius>$ is quite close to $\rog<\Hradius>$. This finishes the first step.

Note that the error between $\outg(\partial*_\Hradius,\partial*_\Hradius)=\outg(\rnu\nu,\rnu,\nu)=\rnu^2$ and its counterpart $\houtg(\partial*_\Hradius\roPhi,\partial*_\Hradius\roPhi)=(\roboost<\Hradius>{\rnu<\Hradius>})^2$ is of the same order as $\deform{\rnu<\Hradius>}$, \ie decays only as $\exp((2-\decay)\Hradius)$ and therefore insufficiently fast. Furthermore, we do not have any control on the derivative of this error. Thus, we only constructed a $\Ck^0_{\decay-2,\scdecay}$-asymptotically hyperbolic chart.

\paragraph{The second step}
In the second step, we choose the full lapse function $f=\rnu$ as derivative of $\Phi$, see the construction explained in the \lq overview of the first two steps\rq \vpageref[above]{OverviewFirstSteps} and Figure~\ref{Image__non-round} \vpageref[below]{Image__non-round}.
\begin{Figure}[Image__non-round]%
 {  \begin{tikzpicture}[scale=.45]
		
		\draw (G1-4) node [below]	() {\textcolor{\coloralar}{$\varphi<\varsigma>^{{-}1}(p)=\Phi(\varsigma,p)$}};
		\draw (G1-1) node [right]	() {\textcolor{\coloralar}{$\varphi<\varsigma>^{{-}1}(p')=\Phi(\varsigma,p')$}};
		\draw (G2-4) node [above]	() {\textcolor{\coloramid}{$\Phi(\Hradius,p)$}};
		\draw (G2-1) node [right]	() {\textcolor{\coloramid}{$\Phi(\Hradius,p')$}};
		\draw (G3-4) node [right]	() {\textcolor{\colorasma}{$\Phi(s,p)$}};
		\draw (G3-1) node [left]	() {\textcolor{\colorasma}{$\Phi(s,p')$}};
		\path ($(0,0)-(0,0 -| G1-4)$) node [below,draw opacity=0]	() {\textcolor{white}{$\varphi<\varsigma>^{{-}1}(p)=\Phi(\varsigma,p)$}};
		\begin{scope}
			\drawgraphs
			\foreach \ln in {1,...,5} {
				\draw [\coloralar,fill] (G1-\ln) circle (\ptssize);
				\draw [\coloramid,fill] (G2-\ln) circle (\ptssize);
				\draw [\colorasma,fill] (G3-\ln) circle (\ptssize);
				\draw [|->,\colorarnu,\stylearnu,\pointerth] (G1-\ln) -- (G2-\ln);
				\draw [|->,\colorarnu,\stylearnu,\pointerth] (G2-\ln) -- (G3-\ln);
			}
		\end{scope}
	\end{tikzpicture}
\par\noindent%
 \hbox{%
	\begin{tikzpicture}[scale=.75]\footnotesize\def\colorlar{black!0!red}			\expandafter\let\csname Rcolor1\endcsname\colorlar
\def\colormid{black!0!blue}			\expandafter\let\csname Rcolor2\endcsname\colormid
\def\colorsma{black!15!green}		\expandafter\let\csname Rcolor3\endcsname\colorsma
\def\coloralar{black!0!red}			\expandafter\let\csname Gcolor1\endcsname\coloralar
\def\coloramid{black!30!blue}		\expandafter\let\csname Gcolor2\endcsname\coloramid
\def\colorasma{black!35!green}	\expandafter\let\csname Gcolor3\endcsname\colorasma
\def\stylelar{}									\expandafter\let\csname Rstyle1\endcsname\stylelar	
\def\stylemid{densely dashed}		\expandafter\let\csname Rstyle2\endcsname\stylemid
\def\stylesma{}									\expandafter\let\csname Rstyle3\endcsname\stylesma
\def\stylealar{}								\expandafter\let\csname Gstyle1\endcsname\stylealar	
\def\styleamid{}								\expandafter\let\csname Gstyle2\endcsname\styleamid
\def\styleasma{dashed}					\expandafter\let\csname Gstyle3\endcsname\styleasma
\def\colorcen{purple}
\def\ptssize{.15}
\def\pointerth{thick}
%
\def\coloramid{red!35!blue}
\def\colorasma{black!50!green}
\def\colorrnu{black}							\def\stylernu{}
\def\colorarnu{orange}						\def\stylearnu{densely dashed}
\def\colorgraph{green}						\def\stylegraph{}
\def\coloragraph{blue}						\def\stylegraph{}
\def\ptssize{.15}
\def\factormid{.9}
\def\factorsma{.75}

	 \begin{scope}
	  \clip (-.27,-.04) -- (-.27,-.18) -- (.27,-.18) -- (.27,-.04);
	  \draw [\coloralar,thick,densely dotted] (0,0) circle (.25);
	 \end{scope}
	  \clip (-.27,.27) -- (-.27,-.05) -- (.27,-.05) -- (.27,.27);
	  \draw [\coloralar,\stylealar,thick] (0,0) circle (.25);
	\end{tikzpicture}}\ :\
 $\Phi(\M<\varsigma>)\subseteq\hyperbolicspace^3$\hspace*{1.8em}
 \hbox{%
	\begin{tikzpicture}[scale=.75]\footnotesize\def\colorlar{black!0!red}			\expandafter\let\csname Rcolor1\endcsname\colorlar
\def\colormid{black!0!blue}			\expandafter\let\csname Rcolor2\endcsname\colormid
\def\colorsma{black!15!green}		\expandafter\let\csname Rcolor3\endcsname\colorsma
\def\coloralar{black!0!red}			\expandafter\let\csname Gcolor1\endcsname\coloralar
\def\coloramid{black!30!blue}		\expandafter\let\csname Gcolor2\endcsname\coloramid
\def\colorasma{black!35!green}	\expandafter\let\csname Gcolor3\endcsname\colorasma
\def\stylelar{}									\expandafter\let\csname Rstyle1\endcsname\stylelar	
\def\stylemid{densely dashed}		\expandafter\let\csname Rstyle2\endcsname\stylemid
\def\stylesma{}									\expandafter\let\csname Rstyle3\endcsname\stylesma
\def\stylealar{}								\expandafter\let\csname Gstyle1\endcsname\stylealar	
\def\styleamid{}								\expandafter\let\csname Gstyle2\endcsname\styleamid
\def\styleasma{dashed}					\expandafter\let\csname Gstyle3\endcsname\styleasma
\def\colorcen{purple}
\def\ptssize{.15}
\def\pointerth{thick}
%
\def\coloramid{red!35!blue}
\def\colorasma{black!50!green}
\def\colorrnu{black}							\def\stylernu{}
\def\colorarnu{orange}						\def\stylearnu{densely dashed}
\def\colorgraph{green}						\def\stylegraph{}
\def\coloragraph{blue}						\def\stylegraph{}
\def\ptssize{.15}
\def\factormid{.9}
\def\factorsma{.75}

	 \begin{scope}
	  \clip (-.27,-.04) -- (-.27,-.18) -- (.27,-.18) -- (.27,-.04);
	  \draw [\coloramid,thick,densely dotted] (0,0) circle (.25);
	 \end{scope}
	  \clip (-.27,.27) -- (-.27,-.05) -- (.27,-.05) -- (.27,.27);
	  \draw [\coloramid,\styleamid,thick] (0,0) circle (.25);
	\end{tikzpicture}}\ :\
 $\Phi(\M<\Hradius>)\subseteq\hyperbolicspace^3$\hspace*{1.8em}%
 \hbox{%
	\begin{tikzpicture}[scale=.75]\footnotesize\def\colorlar{black!0!red}			\expandafter\let\csname Rcolor1\endcsname\colorlar
\def\colormid{black!0!blue}			\expandafter\let\csname Rcolor2\endcsname\colormid
\def\colorsma{black!15!green}		\expandafter\let\csname Rcolor3\endcsname\colorsma
\def\coloralar{black!0!red}			\expandafter\let\csname Gcolor1\endcsname\coloralar
\def\coloramid{black!30!blue}		\expandafter\let\csname Gcolor2\endcsname\coloramid
\def\colorasma{black!35!green}	\expandafter\let\csname Gcolor3\endcsname\colorasma
\def\stylelar{}									\expandafter\let\csname Rstyle1\endcsname\stylelar	
\def\stylemid{densely dashed}		\expandafter\let\csname Rstyle2\endcsname\stylemid
\def\stylesma{}									\expandafter\let\csname Rstyle3\endcsname\stylesma
\def\stylealar{}								\expandafter\let\csname Gstyle1\endcsname\stylealar	
\def\styleamid{}								\expandafter\let\csname Gstyle2\endcsname\styleamid
\def\styleasma{dashed}					\expandafter\let\csname Gstyle3\endcsname\styleasma
\def\colorcen{purple}
\def\ptssize{.15}
\def\pointerth{thick}
%
\def\coloramid{red!35!blue}
\def\colorasma{black!50!green}
\def\colorrnu{black}							\def\stylernu{}
\def\colorarnu{orange}						\def\stylearnu{densely dashed}
\def\colorgraph{green}						\def\stylegraph{}
\def\coloragraph{blue}						\def\stylegraph{}
\def\ptssize{.15}
\def\factormid{.9}
\def\factorsma{.75}

	 \begin{scope}
	  \clip (-.27,-.04) -- (-.27,-.18) -- (.27,-.18) -- (.27,-.04);
	  \draw [\colorasma,thick,densely dotted] (0,0) circle (.25);
	 \end{scope}
	  \clip (-.27,.27) -- (-.27,-.05) -- (.27,-.05) -- (.27,.27);
	  \draw [\colorasma,\styleasma,thick] (0,0) circle (.25);
	\end{tikzpicture}}\ :\
 $\Phi(\M<s>)\subseteq\hyperbolicspace^3$%
\smallskip\par\noindent%
 \raisebox{.25em}{\hbox{%
	\begin{tikzpicture}[scale=.75]\footnotesize\def\colorlar{black!0!red}			\expandafter\let\csname Rcolor1\endcsname\colorlar
\def\colormid{black!0!blue}			\expandafter\let\csname Rcolor2\endcsname\colormid
\def\colorsma{black!15!green}		\expandafter\let\csname Rcolor3\endcsname\colorsma
\def\coloralar{black!0!red}			\expandafter\let\csname Gcolor1\endcsname\coloralar
\def\coloramid{black!30!blue}		\expandafter\let\csname Gcolor2\endcsname\coloramid
\def\colorasma{black!35!green}	\expandafter\let\csname Gcolor3\endcsname\colorasma
\def\stylelar{}									\expandafter\let\csname Rstyle1\endcsname\stylelar	
\def\stylemid{densely dashed}		\expandafter\let\csname Rstyle2\endcsname\stylemid
\def\stylesma{}									\expandafter\let\csname Rstyle3\endcsname\stylesma
\def\stylealar{}								\expandafter\let\csname Gstyle1\endcsname\stylealar	
\def\styleamid{}								\expandafter\let\csname Gstyle2\endcsname\styleamid
\def\styleasma{dashed}					\expandafter\let\csname Gstyle3\endcsname\styleasma
\def\colorcen{purple}
\def\ptssize{.15}
\def\pointerth{thick}
%
\def\coloramid{red!35!blue}
\def\colorasma{black!50!green}
\def\colorrnu{black}							\def\stylernu{}
\def\colorarnu{orange}						\def\stylearnu{densely dashed}
\def\colorgraph{green}						\def\stylegraph{}
\def\coloragraph{blue}						\def\stylegraph{}
\def\ptssize{.15}
\def\factormid{.9}
\def\factorsma{.75}

	  \draw [|->,\colorarnu,\stylearnu,thick] (0,0) -- (1,0);
	\end{tikzpicture}\vspace*{.5em}}}\ :\
	$\partial*_\Hradius\Phi=\rnu\,\hnu$}%
 [Non-round images of the CMC-leaves and the lapse function]%
 {$\Phi$-images of the CMC-leaves and the lap6se function}
 \emph{Construction:} Start with a non-isometric embedding $\varphi<\varsigma>^{{-}1}$ of $\M<\varsigma>$ to a hyperbolic sphere $\hsphere^2_\varsigma(\centerz<\varsigma>)$, \ie the outer (red) sphere denotes $\varphi^{{-}1}(\M<\varsigma>)=\Phi(\M<\varsigma>)$. Now, choose $\Phi$ with $\Phi(\varsigma,\emptyarg)=\varphi<\varsigma>^{{-}1}$ and $\partial*_\Hradius\Phi=\rnu\hnu$, where $\hnu$ denotes the outer unit normal with respect to $\houtg*$ and where $\rnu$ is identified with its pushforward along $\Phi$. This leads to $\Phi(\M<\varsigma>)$ (the middle, blue sphere) and from there we get $\Phi(\M<s>)$ (the small, dashed, green one).%
\end{Figure}%
In particular, this implies that the part of the metric $\outg$ orthogonal to $\M<\Hradius>$ and the one of the pullback of $\houtg$ along $\Phi$ are identical which solves the mentioned main problem of the construction in step one.

Now, we compare $\Phi$ and $\roPhi$ to conclude that $\Phi(\M<\Hradius>)$ is a graph above $\roPhi(\M<\Hradius>)=\hsphere^2_\Hradius(\centerz<\Hradius>)$, \ie there exists a function $\graphf<\Hradius>\in\Wkp^{2,p}(\roPhi( \M<\Hradius>))$ such that $\graphF<\Hradius>:=\text{exp}(\graphf\,\ronu)$ is a bijection from $\roPhi(\M<\Hradius>)$ to $\Phi(\M<\Hradius>)$, where $\ronu$ is the outer unit normal of $\roPhi(\M<\Hradius>)$ with respect to $\houtg$, see Figure~\ref{Image__with-graph} \vpageref[above]{Image__with-graph}. Furthermore, we can show that $\Phi(\M<\Hradius>)$ has (with respect to $\houtg$) almost constant mean curvature, where the error decays as $\Hradius\to\infty$ surprisingly fast---at least it is on the first glimpse surprising.
\begin{Figure}<-.2em>%
 [Image__with-graph]%
 {	\begin{tikzpicture}[scale=.45]\footnotesize
		\def\colorlar{black!0!red}			\expandafter\let\csname Rcolor1\endcsname\colorlar
\def\colormid{black!0!blue}			\expandafter\let\csname Rcolor2\endcsname\colormid
\def\colorsma{black!15!green}		\expandafter\let\csname Rcolor3\endcsname\colorsma
\def\coloralar{black!0!red}			\expandafter\let\csname Gcolor1\endcsname\coloralar
\def\coloramid{black!30!blue}		\expandafter\let\csname Gcolor2\endcsname\coloramid
\def\colorasma{black!35!green}	\expandafter\let\csname Gcolor3\endcsname\colorasma
\def\stylelar{}									\expandafter\let\csname Rstyle1\endcsname\stylelar	
\def\stylemid{densely dashed}		\expandafter\let\csname Rstyle2\endcsname\stylemid
\def\stylesma{}									\expandafter\let\csname Rstyle3\endcsname\stylesma
\def\stylealar{}								\expandafter\let\csname Gstyle1\endcsname\stylealar	
\def\styleamid{}								\expandafter\let\csname Gstyle2\endcsname\styleamid
\def\styleasma{dashed}					\expandafter\let\csname Gstyle3\endcsname\styleasma
\def\colorcen{purple}
\def\ptssize{.15}
\def\pointerth{thick}
%
\def\coloramid{red!35!blue}
\def\colorasma{black!50!green}
\def\colorrnu{black}							\def\stylernu{}
\def\colorarnu{orange}						\def\stylearnu{densely dashed}
\def\colorgraph{green}						\def\stylegraph{}
\def\coloragraph{blue}						\def\stylegraph{}
\def\ptssize{.15}
\def\factormid{.9}
\def\factorsma{.75}

		\def\thepicture#1#2{%
			\coordinate (FIG#1-CUT-SW) at ($(SW -| R1-4)+(-.5,.5)$);
			\coordinate (FIG#1-CUT-SE) at ($(SW -| G3-4)+(.5,.5)$);
			\coordinate (FIG#1-CUT-NE) at ($(G3-4)+(.5,.5)$);
			\coordinate (FIG#1-CUT-NW) at ($(G3-4 -| R1-4)+(-.5,.5)$);
			\draw [densely dotted,gray,thick] (FIG#1-CUT-SW) -- (FIG#1-CUT-SE) -- (FIG#1-CUT-NE) -- (FIG#1-CUT-NW) -- (FIG#1-CUT-SW);
			#2
			\drawspheres
			\drawgraphs*
			\foreach \ln in {1,...,5} {
				\draw [\colorlar,fill] (R1-\ln) circle (\ptssize);
				\draw [\coloralar,fill] (G1-\ln) circle (\ptssize);
				\draw [\colormid,fill] (R2-\ln) circle (\ptssize);
				\draw [\coloramid,fill] (G2-\ln) circle (\ptssize);
				\draw [\coloragraph,fill] (R2-G2-\ln) circle (\ptssize);
				\draw [\colorsma,fill] (R3-\ln) circle (\ptssize);
				\draw [\colorasma,fill] (G3-\ln) circle (\ptssize);
				\draw [\coloragraph,fill] (R3-G3-\ln) circle (\ptssize);
				\draw [|->,\colorrnu,\stylernu,\pointerth] (R1-\ln) -- (R2-\ln);
				\draw [|->,\colorrnu,\stylernu,\pointerth] (R2-\ln) -- (R3-\ln);
				\draw [|->,\colorarnu,\stylearnu,\pointerth] (G1-\ln) -- (G2-\ln);
				\draw [|->,\colorarnu,\stylearnu,\pointerth] (G2-\ln) -- (G3-\ln);
				\draw [|->,\colorgraph,\stylegraph,\pointerth] (R3-\ln) -- (R3-G3-\ln);
			}%
			\draw [dotted,gray,thick] (FIG#1-CUT-SW) -- (FIG#1-CUT-SE) -- (FIG#1-CUT-NE) -- (FIG#1-CUT-NW) -- (FIG#1-CUT-SW);
		}
		\begin{scope}
			\def\colorlar{black!0!red}			\expandafter\let\csname Rcolor1\endcsname\colorlar
\def\colormid{black!0!blue}			\expandafter\let\csname Rcolor2\endcsname\colormid
\def\colorsma{black!15!green}		\expandafter\let\csname Rcolor3\endcsname\colorsma
\def\coloralar{black!0!red}			\expandafter\let\csname Gcolor1\endcsname\coloralar
\def\coloramid{black!30!blue}		\expandafter\let\csname Gcolor2\endcsname\coloramid
\def\colorasma{black!35!green}	\expandafter\let\csname Gcolor3\endcsname\colorasma
\def\stylelar{}									\expandafter\let\csname Rstyle1\endcsname\stylelar	
\def\stylemid{densely dashed}		\expandafter\let\csname Rstyle2\endcsname\stylemid
\def\stylesma{}									\expandafter\let\csname Rstyle3\endcsname\stylesma
\def\stylealar{}								\expandafter\let\csname Gstyle1\endcsname\stylealar	
\def\styleamid{}								\expandafter\let\csname Gstyle2\endcsname\styleamid
\def\styleasma{dashed}					\expandafter\let\csname Gstyle3\endcsname\styleasma
\def\colorcen{purple}
\def\ptssize{.15}
\def\pointerth{thick}
%
\def\coloramid{red!35!blue}
\def\colorasma{black!50!green}
\def\colorrnu{black}							\def\stylernu{}
\def\colorarnu{orange}						\def\stylearnu{densely dashed}
\def\colorgraph{green}						\def\stylegraph{}
\def\coloragraph{blue}						\def\stylegraph{}
\def\ptssize{.15}
\def\factormid{.9}
\def\factorsma{.75}

\footnotesize
\coordinate (center1) at (0,0);
\coordinate (center2) at (-.5,1.5);
\coordinate (center3) at (-.25,2);
\pgfmathsetseed{3}
\coordinate (G2-P1) at ($(center2)+\factormid*(rnd,rnd)+(-9.6,0)$);	\coordinate (G2-P2) at ($(center2)+\factormid*(rnd,rnd)+(-6.79,6.79)$);
\coordinate (G2-P3) at ($(center2)+\factormid*(rnd,rnd)+(0,9.6)$);	\coordinate (G2-P4) at ($(center2)+\factormid*(rnd,rnd)+(6.79,6.79)$);
\coordinate (G2-P5) at ($(center2)+\factormid*(rnd,rnd)+(9.6,0)$);	\coordinate (G2-P6) at ($(center2)+\factormid*(rnd,rnd)+(6.79,-6.79)$);
\coordinate (G2-P7) at ($(center2)+\factormid*(rnd,rnd)+(0,-9.6)$);	\coordinate (G2-P8) at ($(center2)+\factormid*(rnd,rnd)+(-6.79,-6.79)$);
\coordinate (G3-P1) at ($(center3)+\factorsma*(rnd,rnd)+(-7.2,0)$);	\coordinate (G3-P2) at ($(center3)+\factorsma*(rnd,rnd)+(-5.09,5.09)$);
\coordinate (G3-P3) at ($(center3)+\factorsma*(rnd,rnd)+(0,7.2)$);	\coordinate (G3-P4) at ($(center3)+\factorsma*(rnd,rnd)+(5.09,5.09)$);
\coordinate (G3-P5) at ($(center3)+\factorsma*(rnd,rnd)+(7.2,0)$);	\coordinate (G3-P6) at ($(center3)+\factorsma*(rnd,rnd)+(5.09,-5.09)$);
\coordinate (G3-P7) at ($(center3)+\factorsma*(rnd,rnd)+(0,-7.2)$);	\coordinate (G3-P8) at ($(center3)+\factorsma*(rnd,rnd)+(-5.09,-5.09)$);
%
\begin{pgfinterruptboundingbox}
	\begin{scope}
		\clip (-100,-4) -- (-100,100) -- (100,100) -- (100,-4);
		\path [name path global=G1] (center1) circle (12);
		\path [name path global=G2] plot [smooth cycle] coordinates
			{ (G2-P1) (G2-P2) (G2-P3) (G2-P4) (G2-P5) (G2-P6) (G2-P7) (G2-P8) };
		\path [name path global=G3] plot [smooth cycle] coordinates
			{ (G3-P1) (G3-P2) (G3-P3) (G3-P4) (G3-P5) (G3-P6) (G3-P7) (G3-P8) };
		\path [name path global=R1] (center1) circle (12);
		\path [name path global=R2] (center2) circle (9.6);
		\path [name path global=R3] (center3) circle (7.2);
		\path [name path global=l1	]	(center1) -- ($(center1)+(6,12)$);
		\path [name path global=l2	]	(center1) -- ($(center1)+(-10,14)$);
		\path [name path global=l3	]	(center1) -- ($(center1)+(-12,0)$);
		\path [name path global=l4	]	(center1) -- ($(center1)+(-14,-4)$);
		\path [name path global=l5	]	(center1) -- ($(center1)+(12,0)$);
		\path [name path global=l6	]	(center1) -- ($(center1)+(-12,3.5)$);
		\path [name path global=l7	]	(center1) -- ($(center1)+(12,2)$);
		\path [name path global=l1-	]	(center1) -- ($(center1)+(6,11.9)$);
		\path [name path global=l2-	]	(center1) -- ($(center1)+(-10,14.1)$);
		\path [name path global=l3-	]	(center1) -- ($(center1)+(-12,.1)$);
		\path [name path global=l4-	]	(center1) -- ($(center1)+(-14,-3.9)$);
		\path [name path global=l5-	]	(center1) -- ($(center1)+(12,-.1)$);
		\path [name path global=l6-	]	(center1) -- ($(center1)+(-12,3.6)$);
		\path [name path global=l7-	]	(center1) -- ($(center1)+(12,1.9)$);
		\path [name path global=l1+	]	(center1) -- ($(center1)+(6,12.1)$);
		\path [name path global=l2+	]	(center1) -- ($(center1)+(-10,13.9)$);
		\path [name path global=l3+	]	(center1) -- ($(center1)+(-12,-.1)$);
		\path [name path global=l4+	]	(center1) -- ($(center1)+(-14,-4.1)$);
		\path [name path global=l5+	]	(center1) -- ($(center1)+(12,.1)$);
		\path [name path global=l6+	]	(center1) -- ($(center1)+(-12,3.4)$);
		\path [name path global=l7+	]	(center1) -- ($(center1)+(12,2.1)$);
		\coordinate (R1-E) at ($(center1)+(12,0)$);
		\coordinate (R1-W) at ($(center1)-(12,0)$);
		\coordinate (R2-E) at ($(center2)+(9.6,0)$);
		\coordinate (R2-W) at ($(center2)-(9.6,0)$);
		\coordinate (R3-E) at ($(center3)+(7.2,0)$);
		\coordinate (R3-W) at ($(center3)-(7.2,0)$);
		\foreach \ln in {1,...,7} {
			\path [name intersections={of= G1 and l\ln}] coordinate (G1-\ln) at (intersection-1);
			\path [name intersections={of= G2 and l\ln}] coordinate (G2-\ln) at (intersection-1);
			\path [name intersections={of= G2 and l\ln-}] coordinate (G2-\ln-) at (intersection-1);
			\path [name intersections={of= G2 and l\ln+}] coordinate (G2-\ln+) at (intersection-1);
			\path [name intersections={of= R1 and l\ln}] coordinate (R1-\ln) at (intersection-1);
			\path [name intersections={of= R2 and l\ln}] coordinate (R2-\ln) at (intersection-1);
			\path [name path=G2orth] (G2-\ln) -- ($(G2-\ln-)!10cm!90:(G2-\ln+)$);
			\path [name path=R2orth] (R2-\ln) -- (center2);
			\path [name intersections={of= G3 and G2orth}] coordinate (G3-\ln) at (intersection-1);
			\path [name intersections={of= R3 and R2orth}] coordinate (R3-\ln) at (intersection-1);
			\path [name path=R3orth] ($(center3)+2*($(R3-\ln)-(center3)$)$) -- (center3);
			\path [name intersections={of= G2 and R2orth}] coordinate (R2-G2-\ln) at (intersection-1);
			\path [name intersections={of= G3 and R3orth}] coordinate (R3-G3-\ln) at (intersection-1);
		}
	\end{scope}
	\path [name path global=clip1] (-100,-4) -- (100,-4);
	\path [name path global=clip2] (-100,0) -- (100,0);
	\path [name path global=clip3] (0,0) -- (0,100);
	\path [name intersections={of=R1 and clip1}] coordinate (SW) at (intersection-1) coordinate (SE) at (intersection-2);
	\path [name intersections={of=R1 and clip2}] coordinate  (W) at (intersection-1) coordinate  (E) at (intersection-2);
	\path [name intersections={of=R1 and clip3}] coordinate (N)  at (intersection-1);
\end{pgfinterruptboundingbox}
\makeatletter
\NewDocumentCommand\drawasphere{smm}{%
	\IfBooleanTF{#1}{%
		\draw [thick,draw opacity=0] (SW) -- (SE) -- ($(E)+(0,.1)$) -- ($(E)-(0,.1)$)%
			-- ($(N)-(.1,0)$) -- ($(N)+(.1,0)$) -- ($(W)-(0,.1)$) -- ($(W)+(0,.1)$);%
	}\relax%
	\begin{pgfinterruptboundingbox}
		\begin{scope}
			\clip (-100,-3.25) -- (-100,100) -- (100,100) -- (100,-3.25);
			\draw [\csname #2color#3\endcsname,use path=#2#3,thick,\csname #2style#3\endcsname] (0,0);
			\ifx #2R\drawasphere@addtoround#3\else\ifx 1#3\drawasphere@addtoround1\fi\fi%
		\end{scope}
		\begin{scope}
			\clip (-100,-4) -- (-100,-3.23) -- (100,-3.23) -- (100,-4);
			\draw [\csname #2color#3\endcsname,use path=#2#3,thick,densely dotted] (0,0);
		\end{scope}
	\end{pgfinterruptboundingbox}
}
\gdef\drawasphere@addtoround#1{%
	\ifx#11\drawasphere@addtoround@ 1\varsigma{12}{-30}\else\ifx#12\drawasphere@addtoround@ 2\Hradius{9.6}{-30}\else\ifx#13 \drawasphere@addtoround@ 3s{7.2}{-30}\fi\fi\fi%
}
\gdef\drawasphere@addtoround@#1#2#3#4{%
	\draw [\csname Rcolor#1\endcsname,fill] (center#1) circle (\ptssize);%
	\coordinate (center-r) at ($(center#1)!#3!#4:($(center#1)-(1,0)$)$);%
	\draw [|<->|,black!75!white] (center#1) -- (center-r);%
	\draw [black!75!white] ($(center#1)!.4!(center-r)$) node [fill=white] () {\textcolor{black!75!white}{$#2$}};%
}
\makeatother
\NewDocumentCommand\drawspheres{s}{\IfBooleanTF{#1}{\drawasphere R1}{\drawasphere*R1} \drawasphere R2 \drawasphere R3}
\NewDocumentCommand\drawgraphs{s}{ \IfBooleanTF{#1}{\drawasphere G1}{\drawasphere*G1} \drawasphere G2 \drawasphere G3}
			\coordinate (FIG1-CUT-SW) at ($(SW -| R1-4)+(-.5,.5)$);
			\coordinate (FIG1-CUT-SE) at ($(SW -| G3-4)+(.5,.5)$);
			\coordinate (FIG1-CUT-NE) at ($(G3-4)+(.5,.5)$);
			\coordinate (FIG1-CUT-NW) at ($(G3-4 -| R1-4)+(-.5,.5)$);
			\draw [densely dotted,gray,thick] (FIG1-CUT-SW) -- (FIG1-CUT-SE) -- (FIG1-CUT-NE) -- (FIG1-CUT-NW) -- (FIG1-CUT-SW);
			r
			\drawspheres
			\drawgraphs*
			\foreach \ln in {1,...,5} {
				\draw [\colorlar,fill] (R1-\ln) circle (\ptssize);
				\draw [\coloralar,fill] (G1-\ln) circle (\ptssize);
				\draw [\colormid,fill] (R2-\ln) circle (\ptssize);
				\draw [\coloramid,fill] (G2-\ln) circle (\ptssize);
				\draw [\coloragraph,fill] (R2-G2-\ln) circle (\ptssize);
				\draw [\colorsma,fill] (R3-\ln) circle (\ptssize);
				\draw [\colorasma,fill] (G3-\ln) circle (\ptssize);
				\draw [\coloragraph,fill] (R3-G3-\ln) circle (\ptssize);
				\draw [|->,\colorrnu,\stylernu,\pointerth] (R1-\ln) -- (R2-\ln);
				\draw [|->,\colorrnu,\stylernu,\pointerth] (R2-\ln) -- (R3-\ln);
				\draw [|->,\colorarnu,\stylearnu,\pointerth] (G1-\ln) -- (G2-\ln);
				\draw [|->,\colorarnu,\stylearnu,\pointerth] (G2-\ln) -- (G3-\ln);
				\draw [|->,\colorgraph,\stylegraph,\pointerth] (R3-\ln) -- (R3-G3-\ln);
			}%
			\draw [dotted,gray,thick] (FIG1-CUT-SW) -- (FIG1-CUT-SE) -- (FIG1-CUT-NE) -- (FIG1-CUT-NW) -- (FIG1-CUT-SW);
		\relax
		\end{scope}
		\def\magn{2.5}%
		\begin{scope}[scale=\magn,shift={($($(current bounding box.south)-(0,.25)$)-\magn*.5*($(FIG1-CUT-NW) + (FIG1-CUT-NE)$)$)}]
			\coordinate (FIG2-CUT-SW) at ($(SW -| R1-4)+(-.5,.5)$);
			\coordinate (FIG2-CUT-SE) at ($(SW -| G3-4)+(.5,.5)$);
			\coordinate (FIG2-CUT-NE) at ($(G3-4)+(.5,.5)$);
			\coordinate (FIG2-CUT-NW) at ($(G3-4 -| R1-4)+(-.5,.5)$);
			\draw [densely dotted,gray,thick] (FIG2-CUT-SW) -- (FIG2-CUT-SE) -- (FIG2-CUT-NE) -- (FIG2-CUT-NW) -- (FIG2-CUT-SW);
			\def\ptssize{.06} \clip (FIG2-CUT-SW) -- (FIG2-CUT-SE) -- (FIG2-CUT-NE) -- (FIG2-CUT-NW) -- (FIG2-CUT-SW);
			\drawspheres
			\drawgraphs*
			\foreach \ln in {1,...,5} {
				\draw [\colorlar,fill] (R1-\ln) circle (\ptssize);
				\draw [\coloralar,fill] (G1-\ln) circle (\ptssize);
				\draw [\colormid,fill] (R2-\ln) circle (\ptssize);
				\draw [\coloramid,fill] (G2-\ln) circle (\ptssize);
				\draw [\coloragraph,fill] (R2-G2-\ln) circle (\ptssize);
				\draw [\colorsma,fill] (R3-\ln) circle (\ptssize);
				\draw [\colorasma,fill] (G3-\ln) circle (\ptssize);
				\draw [\coloragraph,fill] (R3-G3-\ln) circle (\ptssize);
				\draw [|->,\colorrnu,\stylernu,\pointerth] (R1-\ln) -- (R2-\ln);
				\draw [|->,\colorrnu,\stylernu,\pointerth] (R2-\ln) -- (R3-\ln);
				\draw [|->,\colorarnu,\stylearnu,\pointerth] (G1-\ln) -- (G2-\ln);
				\draw [|->,\colorarnu,\stylearnu,\pointerth] (G2-\ln) -- (G3-\ln);
				\draw [|->,\colorgraph,\stylegraph,\pointerth] (R3-\ln) -- (R3-G3-\ln);
			}%
			\draw [dotted,gray,thick] (FIG2-CUT-SW) -- (FIG2-CUT-SE) -- (FIG2-CUT-NE) -- (FIG2-CUT-NW) -- (FIG2-CUT-SW);
		
		\end{scope}
		\begin{scope}
			\clip (current bounding box.north west) -- (current bounding box.north east) -- (current bounding box.south east) -- (current bounding box.south west)
					-- (FIG1-CUT-SW) -- (FIG1-CUT-SE) -- (FIG1-CUT-NE) -- (FIG1-CUT-NW);
			\clip (current bounding box.north east) -- (current bounding box.north west) -- (current bounding box.south west) -- (current bounding box.south east)
					-- (FIG2-CUT-SE) -- (FIG2-CUT-SW) -- (FIG2-CUT-NW) -- (FIG2-CUT-NE);
			\draw [dotted,gray,thick] (FIG1-CUT-SW) -- (FIG2-CUT-SW);
			\draw [dotted,gray,thick] (FIG1-CUT-SE) -- (FIG2-CUT-SE);
			\draw [dotted,gray,thick] (FIG1-CUT-NW) -- (FIG2-CUT-NW);
			\draw [dotted,gray,thick] (FIG1-CUT-NE) -- (FIG2-CUT-NE);
		\end{scope}
		\begin{scope}
			\clip (FIG1-CUT-SW) -- (FIG1-CUT-SE) -- (FIG1-CUT-NE) -- (FIG1-CUT-NW) -- (FIG1-CUT-SW) -- (FIG2-CUT-SW) -- (FIG2-CUT-SE) -- (FIG2-CUT-NE) -- (FIG2-CUT-NW) -- (FIG2-CUT-SW);
			\draw [loosely dotted,gray!50!white,thick] (FIG1-CUT-SW) -- (FIG2-CUT-SW);
			\draw [loosely dotted,gray!50!white,thick] (FIG1-CUT-SE) -- (FIG2-CUT-SE);
			\draw [loosely dotted,gray!50!white,thick] (FIG1-CUT-NW) -- (FIG2-CUT-NW);
			\draw [loosely dotted,gray!50!white,thick] (FIG1-CUT-NE) -- (FIG2-CUT-NE);
		\end{scope}
		\draw (R1-4)		node [left]	() {\textcolor{\colorlar}{$\roPhi(\varsigma,p)=\Phi(\varsigma,p)$}};
		\draw (R2-4)		node [below]	() {\textcolor{\colormid}{$\roPhi(\Hradius,p)$}};
		\draw (R3-4)		node [left]		() {\textcolor{\colorsma}{$\roPhi(s,p)$}};
		\draw (G2-4)		node [below]	() {\textcolor{\coloramid}{$\Phi(\Hradius,p)$}};
		\draw (G3-4)		node [below]	() {\textcolor{\colorasma}{$\Phi(s,p)$}};
		\draw (R2-G2-4)	node [left]		() {\textcolor{\coloragraph}{$\graphF<\Hradius>(p)$}};
		\draw (R3-G3-4)	node [right]	() {\textcolor{\coloragraph}{$\graphF<s>(p)$}};
		\begin{scope}
			\clip [use path = R1];
			\clip (R1-4) -- (R2-4) -- (R3-4) -- (center3) -- (current bounding box.north east) -- (current bounding box.north west) -- (current bounding box.west);
			\draw (R1-4) circle (.6); \draw [fill] ($(R1-4)!.06!45:(R2-4)$) circle (.03);
			\clip [use path = R2];
			\draw (R2-4) circle (.6); \draw [fill] ($(R2-4)!.05!45:(R3-4)$) circle (.03);
			\clip [use path = R3];
			\draw (R3-4) circle (.6); \draw [fill] ($(R3-4)!.2!45:(R3-G3-4)$) circle (.03);
		\end{scope}
		\begin{scope}
			\clip [use path = G2];
			\clip (R1-4) -- (G2-4) -- (G3-4) -- (current bounding box.south east) -- (current bounding box.south west) -- (current bounding box.west);
			\draw (G2-4) circle (.6); \draw [fill] ($(G2-4)!.06!-45:(G3-4)$) circle (.03);
		\end{scope}
	\end{tikzpicture}
 \par\noindent%
 \hbox{%
	\begin{tikzpicture}[scale=.75]\footnotesize\def\colorlar{black!0!red}			\expandafter\let\csname Rcolor1\endcsname\colorlar
\def\colormid{black!0!blue}			\expandafter\let\csname Rcolor2\endcsname\colormid
\def\colorsma{black!15!green}		\expandafter\let\csname Rcolor3\endcsname\colorsma
\def\coloralar{black!0!red}			\expandafter\let\csname Gcolor1\endcsname\coloralar
\def\coloramid{black!30!blue}		\expandafter\let\csname Gcolor2\endcsname\coloramid
\def\colorasma{black!35!green}	\expandafter\let\csname Gcolor3\endcsname\colorasma
\def\stylelar{}									\expandafter\let\csname Rstyle1\endcsname\stylelar	
\def\stylemid{densely dashed}		\expandafter\let\csname Rstyle2\endcsname\stylemid
\def\stylesma{}									\expandafter\let\csname Rstyle3\endcsname\stylesma
\def\stylealar{}								\expandafter\let\csname Gstyle1\endcsname\stylealar	
\def\styleamid{}								\expandafter\let\csname Gstyle2\endcsname\styleamid
\def\styleasma{dashed}					\expandafter\let\csname Gstyle3\endcsname\styleasma
\def\colorcen{purple}
\def\ptssize{.15}
\def\pointerth{thick}
%
\def\coloramid{red!35!blue}
\def\colorasma{black!50!green}
\def\colorrnu{black}							\def\stylernu{}
\def\colorarnu{orange}						\def\stylearnu{densely dashed}
\def\colorgraph{green}						\def\stylegraph{}
\def\coloragraph{blue}						\def\stylegraph{}
\def\ptssize{.15}
\def\factormid{.9}
\def\factorsma{.75}

	 \begin{scope}
	  \clip (-.27,-.04) -- (-.27,-.18) -- (.27,-.18) -- (.27,-.04);
	  \draw [\coloralar,thick,densely dotted] (0,0) circle (.25);
	 \end{scope}
	  \clip (-.27,.27) -- (-.27,-.05) -- (.27,-.05) -- (.27,.27);
	  \draw [\coloralar,\stylealar,thick] (0,0) circle (.25);
	\end{tikzpicture}}\ :\
 $\Phi(\M<\varsigma>)=\hsphere^2_\varsigma(\centerz<\varsigma>)$\hspace*{1.8em}%
 \hbox{%
	\begin{tikzpicture}[scale=.75]\footnotesize\def\colorlar{black!0!red}			\expandafter\let\csname Rcolor1\endcsname\colorlar
\def\colormid{black!0!blue}			\expandafter\let\csname Rcolor2\endcsname\colormid
\def\colorsma{black!15!green}		\expandafter\let\csname Rcolor3\endcsname\colorsma
\def\coloralar{black!0!red}			\expandafter\let\csname Gcolor1\endcsname\coloralar
\def\coloramid{black!30!blue}		\expandafter\let\csname Gcolor2\endcsname\coloramid
\def\colorasma{black!35!green}	\expandafter\let\csname Gcolor3\endcsname\colorasma
\def\stylelar{}									\expandafter\let\csname Rstyle1\endcsname\stylelar	
\def\stylemid{densely dashed}		\expandafter\let\csname Rstyle2\endcsname\stylemid
\def\stylesma{}									\expandafter\let\csname Rstyle3\endcsname\stylesma
\def\stylealar{}								\expandafter\let\csname Gstyle1\endcsname\stylealar	
\def\styleamid{}								\expandafter\let\csname Gstyle2\endcsname\styleamid
\def\styleasma{dashed}					\expandafter\let\csname Gstyle3\endcsname\styleasma
\def\colorcen{purple}
\def\ptssize{.15}
\def\pointerth{thick}
%
\def\coloramid{red!35!blue}
\def\colorasma{black!50!green}
\def\colorrnu{black}							\def\stylernu{}
\def\colorarnu{orange}						\def\stylearnu{densely dashed}
\def\colorgraph{green}						\def\stylegraph{}
\def\coloragraph{blue}						\def\stylegraph{}
\def\ptssize{.15}
\def\factormid{.9}
\def\factorsma{.75}

	 \begin{scope}
	  \clip (-.27,-.04) -- (-.27,-.18) -- (.27,-.18) -- (.27,-.04);
	  \draw [\coloramid,thick,densely dotted] (0,0) circle (.25);
	 \end{scope}
	  \clip (-.27,.27) -- (-.27,-.05) -- (.27,-.05) -- (.27,.27);
	  \draw [\coloramid,\styleamid,thick] (0,0) circle (.25);
	\end{tikzpicture}}\ :\
 $\Phi(\M<\Hradius>)$\hspace*{1.8em}
 \hbox{%
	\begin{tikzpicture}[scale=.75]\footnotesize\def\colorlar{black!0!red}			\expandafter\let\csname Rcolor1\endcsname\colorlar
\def\colormid{black!0!blue}			\expandafter\let\csname Rcolor2\endcsname\colormid
\def\colorsma{black!15!green}		\expandafter\let\csname Rcolor3\endcsname\colorsma
\def\coloralar{black!0!red}			\expandafter\let\csname Gcolor1\endcsname\coloralar
\def\coloramid{black!30!blue}		\expandafter\let\csname Gcolor2\endcsname\coloramid
\def\colorasma{black!35!green}	\expandafter\let\csname Gcolor3\endcsname\colorasma
\def\stylelar{}									\expandafter\let\csname Rstyle1\endcsname\stylelar	
\def\stylemid{densely dashed}		\expandafter\let\csname Rstyle2\endcsname\stylemid
\def\stylesma{}									\expandafter\let\csname Rstyle3\endcsname\stylesma
\def\stylealar{}								\expandafter\let\csname Gstyle1\endcsname\stylealar	
\def\styleamid{}								\expandafter\let\csname Gstyle2\endcsname\styleamid
\def\styleasma{dashed}					\expandafter\let\csname Gstyle3\endcsname\styleasma
\def\colorcen{purple}
\def\ptssize{.15}
\def\pointerth{thick}
%
\def\coloramid{red!35!blue}
\def\colorasma{black!50!green}
\def\colorrnu{black}							\def\stylernu{}
\def\colorarnu{orange}						\def\stylearnu{densely dashed}
\def\colorgraph{green}						\def\stylegraph{}
\def\coloragraph{blue}						\def\stylegraph{}
\def\ptssize{.15}
\def\factormid{.9}
\def\factorsma{.75}

	 \begin{scope}
	  \clip (-.27,-.04) -- (-.27,-.18) -- (.27,-.18) -- (.27,-.04);
	  \draw [\colorasma,thick,densely dotted] (0,0) circle (.25);
	 \end{scope}
	  \clip (-.27,.27) -- (-.27,-.05) -- (.27,-.05) -- (.27,.27);
	  \draw [\colorasma,\styleasma,thick] (0,0) circle (.25);
	\end{tikzpicture}}\ :\
 $\Phi(\M<s>)$
 \smallskip\par\noindent
 \hbox{%
	\begin{tikzpicture}[scale=.75]\footnotesize\def\colorlar{black!0!red}			\expandafter\let\csname Rcolor1\endcsname\colorlar
\def\colormid{black!0!blue}			\expandafter\let\csname Rcolor2\endcsname\colormid
\def\colorsma{black!15!green}		\expandafter\let\csname Rcolor3\endcsname\colorsma
\def\coloralar{black!0!red}			\expandafter\let\csname Gcolor1\endcsname\coloralar
\def\coloramid{black!30!blue}		\expandafter\let\csname Gcolor2\endcsname\coloramid
\def\colorasma{black!35!green}	\expandafter\let\csname Gcolor3\endcsname\colorasma
\def\stylelar{}									\expandafter\let\csname Rstyle1\endcsname\stylelar	
\def\stylemid{densely dashed}		\expandafter\let\csname Rstyle2\endcsname\stylemid
\def\stylesma{}									\expandafter\let\csname Rstyle3\endcsname\stylesma
\def\stylealar{}								\expandafter\let\csname Gstyle1\endcsname\stylealar	
\def\styleamid{}								\expandafter\let\csname Gstyle2\endcsname\styleamid
\def\styleasma{dashed}					\expandafter\let\csname Gstyle3\endcsname\styleasma
\def\colorcen{purple}
\def\ptssize{.15}
\def\pointerth{thick}
%
\def\coloramid{red!35!blue}
\def\colorasma{black!50!green}
\def\colorrnu{black}							\def\stylernu{}
\def\colorarnu{orange}						\def\stylearnu{densely dashed}
\def\colorgraph{green}						\def\stylegraph{}
\def\coloragraph{blue}						\def\stylegraph{}
\def\ptssize{.15}
\def\factormid{.9}
\def\factorsma{.75}

	 \begin{scope}
	  \clip (-.27,-.04) -- (-.27,-.18) -- (.27,-.18) -- (.27,-.04);
	  \draw [\colorlar,thick,densely dotted] (0,0) circle (.25);
	 \end{scope}
	  \clip (-.27,.27) -- (-.27,-.05) -- (.27,-.05) -- (.27,.27);
	  \draw [\colorlar,\stylelar,thick] (0,0) circle (.25);
	\end{tikzpicture}}\ :\
 $\roPhi(\M<\varsigma>)=\hsphere^2_\varsigma(\centerz<\varsigma>)$\hspace*{1.8em}
 \hbox{%
	\begin{tikzpicture}[scale=.75]\footnotesize\def\colorlar{black!0!red}			\expandafter\let\csname Rcolor1\endcsname\colorlar
\def\colormid{black!0!blue}			\expandafter\let\csname Rcolor2\endcsname\colormid
\def\colorsma{black!15!green}		\expandafter\let\csname Rcolor3\endcsname\colorsma
\def\coloralar{black!0!red}			\expandafter\let\csname Gcolor1\endcsname\coloralar
\def\coloramid{black!30!blue}		\expandafter\let\csname Gcolor2\endcsname\coloramid
\def\colorasma{black!35!green}	\expandafter\let\csname Gcolor3\endcsname\colorasma
\def\stylelar{}									\expandafter\let\csname Rstyle1\endcsname\stylelar	
\def\stylemid{densely dashed}		\expandafter\let\csname Rstyle2\endcsname\stylemid
\def\stylesma{}									\expandafter\let\csname Rstyle3\endcsname\stylesma
\def\stylealar{}								\expandafter\let\csname Gstyle1\endcsname\stylealar	
\def\styleamid{}								\expandafter\let\csname Gstyle2\endcsname\styleamid
\def\styleasma{dashed}					\expandafter\let\csname Gstyle3\endcsname\styleasma
\def\colorcen{purple}
\def\ptssize{.15}
\def\pointerth{thick}
%
\def\coloramid{red!35!blue}
\def\colorasma{black!50!green}
\def\colorrnu{black}							\def\stylernu{}
\def\colorarnu{orange}						\def\stylearnu{densely dashed}
\def\colorgraph{green}						\def\stylegraph{}
\def\coloragraph{blue}						\def\stylegraph{}
\def\ptssize{.15}
\def\factormid{.9}
\def\factorsma{.75}

	 \begin{scope}
	  \clip (-.27,-.04) -- (-.27,-.18) -- (.27,-.18) -- (.27,-.04);
	  \draw [\colormid,thick,densely dotted] (0,0) circle (.25);
	 \end{scope}
	  \clip (-.27,.27) -- (-.27,-.05) -- (.27,-.05) -- (.27,.27);
	  \draw [\colormid,\stylemid,thick] (0,0) circle (.25);
	\end{tikzpicture}}\ :\
 $\roPhi(\M<\Hradius>)=\hsphere^2_\varsigma(\centerz<\Hradius>)$\hspace*{1.8em}%
 \hbox{%
	\begin{tikzpicture}[scale=.75]\footnotesize\def\colorlar{black!0!red}			\expandafter\let\csname Rcolor1\endcsname\colorlar
\def\colormid{black!0!blue}			\expandafter\let\csname Rcolor2\endcsname\colormid
\def\colorsma{black!15!green}		\expandafter\let\csname Rcolor3\endcsname\colorsma
\def\coloralar{black!0!red}			\expandafter\let\csname Gcolor1\endcsname\coloralar
\def\coloramid{black!30!blue}		\expandafter\let\csname Gcolor2\endcsname\coloramid
\def\colorasma{black!35!green}	\expandafter\let\csname Gcolor3\endcsname\colorasma
\def\stylelar{}									\expandafter\let\csname Rstyle1\endcsname\stylelar	
\def\stylemid{densely dashed}		\expandafter\let\csname Rstyle2\endcsname\stylemid
\def\stylesma{}									\expandafter\let\csname Rstyle3\endcsname\stylesma
\def\stylealar{}								\expandafter\let\csname Gstyle1\endcsname\stylealar	
\def\styleamid{}								\expandafter\let\csname Gstyle2\endcsname\styleamid
\def\styleasma{dashed}					\expandafter\let\csname Gstyle3\endcsname\styleasma
\def\colorcen{purple}
\def\ptssize{.15}
\def\pointerth{thick}
%
\def\coloramid{red!35!blue}
\def\colorasma{black!50!green}
\def\colorrnu{black}							\def\stylernu{}
\def\colorarnu{orange}						\def\stylearnu{densely dashed}
\def\colorgraph{green}						\def\stylegraph{}
\def\coloragraph{blue}						\def\stylegraph{}
\def\ptssize{.15}
\def\factormid{.9}
\def\factorsma{.75}

	 \begin{scope}
	  \clip (-.27,-.04) -- (-.27,-.18) -- (.27,-.18) -- (.27,-.04);
	  \draw [\colorsma,thick,densely dotted] (0,0) circle (.25);
	 \end{scope}
	  \clip (-.27,.27) -- (-.27,-.05) -- (.27,-.05) -- (.27,.27);
	  \draw [\colorsma,\stylesma,thick] (0,0) circle (.25);
	\end{tikzpicture}}\ :\
 $\roPhi(\M<s>)=\hsphere^2_s(\centerz<s>)$%
 \smallskip\par\noindent
 \raisebox{.25em}{\hbox{%
	\begin{tikzpicture}[scale=.75]\footnotesize\def\colorlar{black!0!red}			\expandafter\let\csname Rcolor1\endcsname\colorlar
\def\colormid{black!0!blue}			\expandafter\let\csname Rcolor2\endcsname\colormid
\def\colorsma{black!15!green}		\expandafter\let\csname Rcolor3\endcsname\colorsma
\def\coloralar{black!0!red}			\expandafter\let\csname Gcolor1\endcsname\coloralar
\def\coloramid{black!30!blue}		\expandafter\let\csname Gcolor2\endcsname\coloramid
\def\colorasma{black!35!green}	\expandafter\let\csname Gcolor3\endcsname\colorasma
\def\stylelar{}									\expandafter\let\csname Rstyle1\endcsname\stylelar	
\def\stylemid{densely dashed}		\expandafter\let\csname Rstyle2\endcsname\stylemid
\def\stylesma{}									\expandafter\let\csname Rstyle3\endcsname\stylesma
\def\stylealar{}								\expandafter\let\csname Gstyle1\endcsname\stylealar	
\def\styleamid{}								\expandafter\let\csname Gstyle2\endcsname\styleamid
\def\styleasma{dashed}					\expandafter\let\csname Gstyle3\endcsname\styleasma
\def\colorcen{purple}
\def\ptssize{.15}
\def\pointerth{thick}
%
\def\coloramid{red!35!blue}
\def\colorasma{black!50!green}
\def\colorrnu{black}							\def\stylernu{}
\def\colorarnu{orange}						\def\stylearnu{densely dashed}
\def\colorgraph{green}						\def\stylegraph{}
\def\coloragraph{blue}						\def\stylegraph{}
\def\ptssize{.15}
\def\factormid{.9}
\def\factorsma{.75}

	  \draw [|->,\colorrnu,\stylernu,thick] (0,0) -- (1,0);
	\end{tikzpicture}\vspace*{.5em}}}\ :\
	$\partial*_\Hradius\roPhi=(1{\,+\,}\roboost{\rnu})\,\ronu$\hspace*{1.7em}
 \raisebox{.25em}{\hbox{%
	\begin{tikzpicture}[scale=.75]\footnotesize\def\colorlar{black!0!red}			\expandafter\let\csname Rcolor1\endcsname\colorlar
\def\colormid{black!0!blue}			\expandafter\let\csname Rcolor2\endcsname\colormid
\def\colorsma{black!15!green}		\expandafter\let\csname Rcolor3\endcsname\colorsma
\def\coloralar{black!0!red}			\expandafter\let\csname Gcolor1\endcsname\coloralar
\def\coloramid{black!30!blue}		\expandafter\let\csname Gcolor2\endcsname\coloramid
\def\colorasma{black!35!green}	\expandafter\let\csname Gcolor3\endcsname\colorasma
\def\stylelar{}									\expandafter\let\csname Rstyle1\endcsname\stylelar	
\def\stylemid{densely dashed}		\expandafter\let\csname Rstyle2\endcsname\stylemid
\def\stylesma{}									\expandafter\let\csname Rstyle3\endcsname\stylesma
\def\stylealar{}								\expandafter\let\csname Gstyle1\endcsname\stylealar	
\def\styleamid{}								\expandafter\let\csname Gstyle2\endcsname\styleamid
\def\styleasma{dashed}					\expandafter\let\csname Gstyle3\endcsname\styleasma
\def\colorcen{purple}
\def\ptssize{.15}
\def\pointerth{thick}
%
\def\coloramid{red!35!blue}
\def\colorasma{black!50!green}
\def\colorrnu{black}							\def\stylernu{}
\def\colorarnu{orange}						\def\stylearnu{densely dashed}
\def\colorgraph{green}						\def\stylegraph{}
\def\coloragraph{blue}						\def\stylegraph{}
\def\ptssize{.15}
\def\factormid{.9}
\def\factorsma{.75}

	  \draw [|->,\colorarnu,\stylearnu,thick] (0,0) -- (1,0);
	\end{tikzpicture}\vspace*{.5em}}}\ :\
	$\partial*_\Hradius\Phi=\rnu\,\hnu$\hspace*{1.7em}
 \raisebox{.25em}{\hbox{%
	\begin{tikzpicture}[scale=.75]\footnotesize\def\colorlar{black!0!red}			\expandafter\let\csname Rcolor1\endcsname\colorlar
\def\colormid{black!0!blue}			\expandafter\let\csname Rcolor2\endcsname\colormid
\def\colorsma{black!15!green}		\expandafter\let\csname Rcolor3\endcsname\colorsma
\def\coloralar{black!0!red}			\expandafter\let\csname Gcolor1\endcsname\coloralar
\def\coloramid{black!30!blue}		\expandafter\let\csname Gcolor2\endcsname\coloramid
\def\colorasma{black!35!green}	\expandafter\let\csname Gcolor3\endcsname\colorasma
\def\stylelar{}									\expandafter\let\csname Rstyle1\endcsname\stylelar	
\def\stylemid{densely dashed}		\expandafter\let\csname Rstyle2\endcsname\stylemid
\def\stylesma{}									\expandafter\let\csname Rstyle3\endcsname\stylesma
\def\stylealar{}								\expandafter\let\csname Gstyle1\endcsname\stylealar	
\def\styleamid{}								\expandafter\let\csname Gstyle2\endcsname\styleamid
\def\styleasma{dashed}					\expandafter\let\csname Gstyle3\endcsname\styleasma
\def\colorcen{purple}
\def\ptssize{.15}
\def\pointerth{thick}
%
\def\coloramid{red!35!blue}
\def\colorasma{black!50!green}
\def\colorrnu{black}							\def\stylernu{}
\def\colorarnu{orange}						\def\stylearnu{densely dashed}
\def\colorgraph{green}						\def\stylegraph{}
\def\coloragraph{blue}						\def\stylegraph{}
\def\ptssize{.15}
\def\factormid{.9}
\def\factorsma{.75}

	  \draw [|->,\colorgraph,\stylegraph,thick] (0,0) -- (1,0);
	\end{tikzpicture}\vspace*{.5em}}}\ :\
	graph function $\graphf$%
}%
 [Images of the CMC-leaves and the lapse function]%
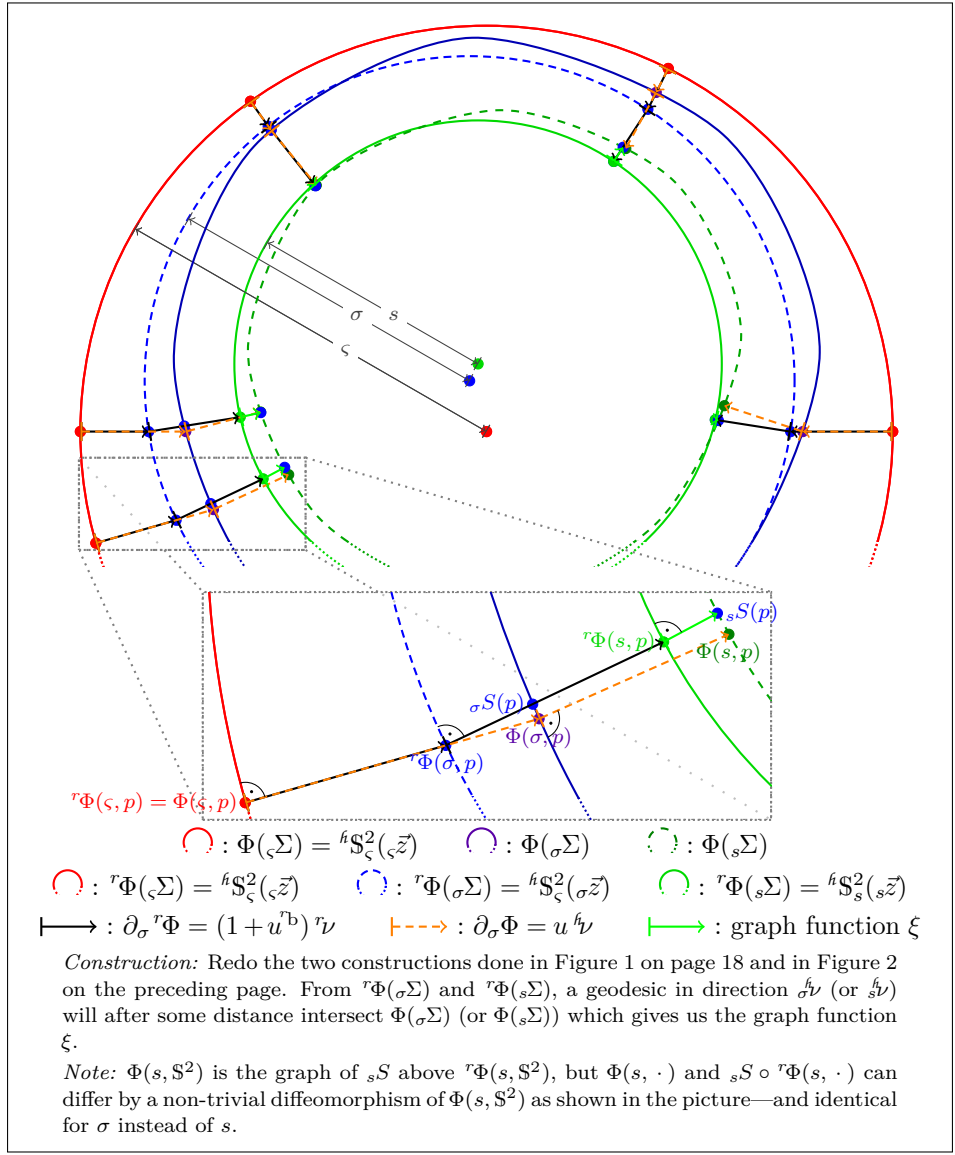
 {$\Phi$- and $\roPhi$-images of the CMC-leaves and the lapse functions}%
	\emph{Construction:} Redo the two constructions done in Figure~\ref{Image__round} \vpageref[above]{Image__round} and in Figure~\ref{Image__non-round} \vpageref[above]{Image__non-round}. From $\roPhi(\M<\Hradius>)$ and $\roPhi(\M<s>)$, a geodesic in direction $\hnu<\Hradius>$ (or $\hnu<s>$) will after some distance intersect $\Phi(\M<
	\Hradius>)$ (or $\Phi(\M<s>)$) which gives us the graph function $\graphf$. \smallskip\par
	\emph{Note:} $\Phi(s,\sphere^2)$ is the graph of $\graphF<s>$ above $\roPhi(s,\sphere^2)$, but $\Phi(s,\emptyarg)$ and $\graphF<s>\circ\roPhi(s,\emptyarg)$ can differ by a non-trivial diffeomorphism of $\Phi(s,\sphere^2)$ as shown in the picture---and identical for $\Hradius$ instead of $s$.
\end{Figure}%
Using \cite[Thm~A.2]{nerz2016HBCMCExistence}, this will give us very sharp estimates on the graph function~$\graphf$ and the second fundamental form of $\Phi(\M<\Hradius>)\hookrightarrow\hyperbolicspace$. Then we prove that this leads to equally good estimates on the graph map $\graphF$ (see above)---note that $\graphF\circ\roPhi$ and $\Phi$ differ by a diffeomorphism of $\Phi(\M<\Hradius>)$, see Figure~\ref{Image__with-graph} \vpageref[above]{Image__with-graph}. Now, the results of step one imply that the metric $\g<\Hradius>$ of $\M<\Hradius>$ is quite close to the pullback $\hg<\Hradius>$ along $\Phi(\Hradius,\emptyarg)$ of the metric induced on $\Phi(\M<\Hradius>)$ by~$\houtg$. An integration of the mentioned \lq very sharp\rq\ control on the second fundamental form, then implies the necessary, even stronger estimates $\outg-\houtg=\mathcal O(\exp({-}\decay\Hradius))$ and $\houtlevi\outg=\mathcal O(\exp((1-\decay)\Hradius))$ on the error of the metrics. This finishes the second step.

\paragraph{The third step}
Now, let us now introduce the $\varsigma$-dependency in the notation and write $\roPhi[\varsigma]$, $\Phi[\varsigma]$, and $\centerz[\varsigma]<\Hradius>$ for the quantities $\roPhi$, $\Phi$, and $\centerz<\Hradius>$ explained above. In the third step, we first prove that the center $\centerz<\Hradius>[\varsigma]$ of a fixed surfaces $\M<\Hradius>$ converges to some center point $\centerz<\Hradius>[\infty]$ as $\varsigma\to\infty$ if we choose simple initial data $\varphi<\varsigma>:\sphere^2_\varsigma(0)\to\M<\varsigma>$ for the charts $\roPhi$ and $\Phi$.

Then, we use this to choose the \lq better\rq\ initial data $\varphi<\varsigma>:\sphere^2_\varsigma(\centerz[\infty]<\varsigma>)\to\M<\varsigma>$ for $\roPhi$ and conclude that then the center points $\centerz[\varsigma]<\Hradius>$ are independent of the chosen large initial mean curvature radius $\varsigma$. Therefore, the two round charts $\roPhi[\varsigma]$ and $\roPhi[\varsigma']$ to two initial mean curvature radii $\varsigma$ and $\varsigma'$ differ only by an one-parameter family of reparametrizations of $\sphere^2$. Removing rotations, we can prove that we can take the limit of $\roPhi[\varsigma](\Hradius,\emptyarg)$ as $\varsigma\to\infty$.

By the estimates on the graph function (see step two), a simple compactness argument proves that $\Phi[\varsigma_n](\Hradius,\emptyarg)$ converges for some sequence $\varsigma_n\to\infty$ and a fixed $\Hradius$. Per construction of $\Phi[\varsigma]$ this proves that it converges for every $\Hradius$ proving the existence of a second (better) global chart.

\paragraph{The forth step}
Finally, we use the assumed decay behavior of $\outric-\houtric=\mathcal O(\exp({-}\decay\Hradius))$ and the already proven ones $\outg-\houtg=\mathcal O(\exp({-}\decay\Hradius))$ and $\houtlevi\outg=\mathcal O(\exp((1-\decay)\Hradius))$ to conclude $\hlaplace(\outg-\houtg)=\mathcal O(\exp({-}\decay\Hradius))$. The regularity of the Laplace operator now proves the theorem.

\begin{proof}<MainTheorem>
To shorten notation, we define the following error terms
\[ \masserr(\Hradius) := C\vert(\masstd<\Hradius>^\oi)_{\oi=1}^3\vert_{\R^3} + C\,\exp({-}\frac\outve2\Hradius), \qquad
	\Masserr(\Hradius) := \int_\Hradius^\infty m(s)\d s + C\,\exp({-}\frac\outve2\Hradius), \]
where the constants $C$ will change from line to line and depend on~$\decay$, $\scdecay$, $\vert\mass\vert$, and~$\eta$, but not on $\varsigma$. We note that we already know
\[ \Vert\rnu-1\Vert_{\Lp^\infty(\M<\Hradius>)} \le \masserr(\Hradius), \qquad
	\masserr**\in\Lp^1(\vphantom{\big|}\interval{\Hradius_0}\infty), \qquad
		\Masserr(\Hradius)\xrightarrow{\Hradius\to\infty} 0. \]
Furthermore, we note that a~posteriori $\masserr(\Hradius)$ (and therefore $\Masserr(\Hradius)$) decays like $\exp({-}\outve\Hradius)$ as $\masstd<\Hradius>$ converges at this order, see \cite{herzlich2015computing}.
 
Let $(\varsigma,\varphi<\varsigma>,\centerz<\varsigma>)$ be some finite initial data, \ie $\varsigma>\Hradius_0$, $\varphi<\varsigma>:\sphere^2_\varsigma(\centerz<\varsigma>)\to\M<\varsigma>$, and $\centerz<\varsigma>$ are some (large) mean curvature radius, a parametrization of the corresponding leaf as in \refth{Regularity_Surfaces}, and some point in the hyperbolicspace. Everything will depend on the chosen finite initial data, but to simplify notation (at least a little bit), we suppress this dependency in the following.
\smallskip

\paragraph{The round map}
As first step, we map each leaf to an exact geodesic sphere in the hyperbolic space by only looking at the boosting part of the lapse function. As we will later construct a chart mapping each leaf to a deformation of these spheres and need to use both at the same time, we use the upperindex $\roundr$ for the former one:
Let $\roPhi:\interval{\Hradius_0}\infty\times\sphere^2\to\hyperbolicspace$ denote the unique smooth map with
\[ \partial[\Hradius]@{\roPhi} = (1+\roboost<\Hradius>{(\rnu<\Hradius>\circ\roPhi<\Hradius>^{{-}1})})\,\ronu<\Hradius>, \qquad \roPhi<\varsigma>=\varphi<\varsigma>^{{-}1}, \]
where $\roPhi<\Hradius>:\sphere^2\to\roPhi(\Hradius,\sphere^2)$, $\ronu<\Hradius>$, and~$\roboost<\Hradius>{(\,{\cdot}\,)}$ denote the diffeomorphism $\roPhi<\Hradius>:=\roPhi(\Hradius,\,{\cdot}\,)$, the outer unit normal of $\roPhi<\Hradius>(\sphere^2)$, and the boosting part of a function with respect to the metric induced on $\roPhi<\Hradius>(\sphere^2)$ by the surrounding hyperbolic metric. By \refth{Boosting_deformations}, we know $\roPhi<\Hradius>(\M<\Hradius>)=\hsphere^2_\Hradius(\centerz<\Hradius>\hspace{.05em})$ for some $\centerz<\Hradius>\in\hyperbolicspace$ and every $\Hradius\in\interval\Hradius_1\infty$. 

Now, let us prove that the metric $\rog<\Hradius>:=\pullback{\roPhi<\Hradius>}\houtg$ is quite close to $\g<\Hradius>$. We assume therefore that for some fixed constants $d>0$ and $p>2$ the radius $\Hradius_1^d\in\interval*\Hradius_0*\varsigma$ is minimal and $\Hradius_2^d\in\interval*\varsigma\infty$ is maximal such that
\begin{equation*}\labeleq{graphf_assumption_round}
 \Vert\g<\Hradius>-\rog<\Hradius>\Vert_{\Wkp^{1,p}(\M<\Hradius>)} \le d\,\exp((\frac2p-\frac12-\outve)\Hradius) + \c\,\exp(\frac2p\Hradius-(\frac12+\outve)\varsigma) \qquad\forall\,\Hradius\in\interval{\Hradius_1^d}{\Hradius_2^d},
\end{equation*}
where $\c$ denotes \emph{twice} the constant $C$ from \refth{Regularity_Surfaces}. Note that the continuity of $\roPhi$ implies that the same is true for $\Hradius=\Hradius_1^d$ and $\Hradius=\Hradius_2^d$ if $\Hradius_1^d\neq\Hradius_0$ and $\Hradius_2^d\neq\infty$, respectively. Furthermore, it implies that $\Hradius_2^d>\varsigma>\Hradius_1^d$ for every $d>0$. We again suppress the index $\varsigma$ and $\Hradius\in\interval*{\Hradius_1^d}*{\Hradius_2^d}$ with $\Hradius_0<\Hradius<\infty$.

In particular, we know
\begin{equation*}
 \Vert\roboost{\rnu} - \boost{\rnu}\Vert_{\Wkp^{2,p}(\M)}
 \le C\,((d+1)\,\exp((\frac2p-\frac12-\outve)\Hradius) + \exp(\frac2p\Hradius-(\frac12+\outve)\varsigma))\masserr*,
\end{equation*}
where we used $\Vert\rnu-1\Vert_{\Wkp^{2,p}(\M)}\le \masserr*$ and $\boost[\mathcal h]1\equiv0$ for every metric $\mathcal h$. By \refth{second_fundamental_form}, we furthermore know
\[ \Vert\zFund-\frac12\roH\:\g\Vert_{\Lp^p(\M)} + \Vert\rolevi\zFund\Vert_{\Lp^p(\M)} \le C\,\exp((\frac2p-\decay)\Hradius), \]
where we used $\roH\equiv\H$ for the mean curvature $\roH$ of $\hsphere^2_\Hradius(\centerz<\Hradius>)$.
As $\roPhi(\Hradius,\sphere^2)$ is a geodesic sphere, we furthermore know $\rozFundtrf\equiv0$. Combined this gives
\[ \Vert\partial[\Hradius]@{\:\rog<\Hradius>'}-\partial[\Hradius]@{\:\g<\Hradius>'}\Vert_{\Wkp^{1,p}(\M<\Hradius>)}
	\le (d\masserr*+C)\,\exp((\frac2p-\frac52-\outve)\Hradius) + \masserr*\exp((\frac2p-2)\Hradius-(\frac12+\outve)\varsigma), \]
where $\rog<\Hradius>':=\sinh(\Hradius)^{{-}2}\rog<\Hradius>$ and $\g<\Hradius>':=\sinh(\Hradius)^{{-}2}\g<\Hradius>$ again denote the rescaled metrics. Thus, an integration proves
\begin{equation*}
 \Vert\g<\Hradius>'-\rog<\Hradius>'\Vert_{\Wkp^{1,p}(\M<\Hradius>)}
 \le \left\{\begin{aligned}
		(d\Masserr*+C_0)\,\exp((\frac2p-\frac52-\outve)\Hradius)
			+ \c\exp((\frac2p-2)\Hradius-(\frac12+\outve)\varsigma)
			&:& \Hradius<\varsigma \\
		(d\Masserr(\varsigma)+C_0)\,\exp((\frac2p-\frac52-\outve)\Hradius)
			+ \c\exp((\frac2p-2)\Hradius-(\frac12+\outve)\varsigma)
			&:& \Hradius>\varsigma
		\end{aligned}\right.
\end{equation*}
for some constant $C_0$ independent of $\varsigma$ and $d$, where we used that for $\Hradius=\varsigma$ the inequality \eqref{graphf_assumption_round} is true for $0$ and $\frac\c2$ instead of $d$ and $\c$, respectively.

For $d=4C_0$ and assuming that $\Hradius_0$ is so large that $\Masserr(\Hradius)<\frac14$ for every $\Hradius>\Hradius_0$, this proves
\begin{align*}\labeleq{graphf_assumption_round_2}
 \Vert\g<\Hradius>-\rog<\Hradius>\Vert_{\Wkp^{1,p}(\M<\Hradius>)}
 \le 2C_0\,\exp((\frac2p-\frac12-\outve)\Hradius) + \c\,\exp(\frac2p\Hradius-(\frac12+\outve)\varsigma) \quad\forall\,\Hradius\in\interval{\Hradius_1^{4C_0}}{\Hradius_2^{4C_0}},\,d\ge4C_0
\end{align*}
where we recall that $C_0$ is independent of $\varsigma$. In partcular, $\Hradius_1^{2C_0}\le\Hradius_1^{4C_0}$ and $\Hradius_2^{2C_0}\ge\Hradius_2^{4C_0}$. However by the continuity of $\roPhi$, we know $\Hradius_1^{4C_0}<\Hradius_1^{2C_0}$ and $\Hradius_2^{4C_0}>\Hradius_2^{2C_0}$ or $\Hradius_1^{2C_0}=\Hradius_0$ and $\Hradius_2^{2C_0}=\infty$, respectively see~\eqref{graphf_assumption_round}. Thus, we have $\Hradius_1^{4C_0}=\Hradius_0$ and $\Hradius_2^{4C_0}=\infty$. Combined this proves the existence of some constants $\Hradius_0$ and $C$ which are independent of $\varsigma$ and satisfy
\begin{equation*}\labeleq{graphf_imp_round}
 \Vert\g<\Hradius>-\rog<\Hradius>\Vert_{\Wkp^{1,p}(\M<\Hradius>)} \le C\,\exp((\frac2p-\frac12-\outve)\Hradius) + C\,\exp(\frac2p\Hradius-(\frac12+\outve)\varsigma) \qquad\forall\,\Hradius\in\interval\Hradius_0\infty
\end{equation*}

\paragraph{The real, finite parametrization}
Let $\Phi[\varsigma]:\interval\Hradius_0*\varsigma\times\sphere^2\to\hyperbolicspace$ be the unique injective smooth map with $\partial*_\Hradius\Phi[\varsigma]=\rnu\,\nu$ and $\Phi[\varsigma](\varsigma,\emptyarg)=\varphi<\varsigma>^{{-}1}$. We denote by $\hg[\varsigma]<\Hradius>$, $\hzFund[\varsigma]<\Hradius>$ \etc the pullback of the metric, second fundamental form, \etc of $\Phi[\varsigma](\Hradius,\sphere^2)\hookrightarrow(\hyperbolicspace,\houtg*)$ along $\Phi[\varsigma](\Hradius,\,{\cdot}\,)$. Again, we suppress the index~$\varsigma$ in the following.

Fix some constants $p>2$ and $d>0$ and let $\Hradius_1^d\in\interval\Hradius_0*\varsigma$ be the minimal mean curvature radius such that for every $\Hradius\in\interval{\Hradius_1^d}*\varsigma$ there exists a function $\graphf<\Hradius>\in\Wkp^{2,p}(\sphere^2_\Hradius(\centerz<\Hradius>))$ such that $\M=\graph\hspace{.05em}\graphf<\Hradius>$ and
\begin{equation*}\labeleq{graphf_assumption}
 \Vert\houtexp(\vphantom{\big|}\graphf<\Hradius>\,\nu)\circ\roPhi<\Hradius> - \Phi<\Hradius>\Vert_{\Wkp^{2,p}(\hsphere^2_\Hradius(\centerz))} \le d\exp((\frac2p-\frac12-\outve)\Hradius) + \c\,\exp({-}(\frac12+\outve)\varsigma+\frac2p\Hradius),
\end{equation*}
where $\c$ the constant $C$ from \refth{Regularity_Surfaces}. Exactly as in the first step, \eqref{graphf_assumption} holds for all $\Hradius$ up to $\varsigma$, \ie $\Hradius_1^d=\Hradius_0$ for some $\varsigma$-independent constant $d$, if there is a $\varsigma$-independent constant $d$ with the following property: if \eqref{graphf_assumption} holds for all $\Hradius\in\interval{\Hradius_1^{d'}}\varsigma$ and some $d'\in\interval0{2d}$, then \eqref{graphf_assumption} holds for all $\Hradius\in\interval{\Hradius_1^{d'}}\varsigma$ and the given $d$, too.

As first step, we note that \eqref{graphf_imp_round} and \eqref{graphf_assumption} combined imply
\begin{align*}
 \Vert\g<\Hradius>-\hg<\Hradius>\Vert_{\Wkp^{1,p}(\M<\Hradius>)} + \Vert\rog<\Hradius>-\hg<\Hradius>\Vert_{\Wkp^{1,p}(\M<\Hradius>)}
  \le{}& C(d+1)\,\exp((\frac2p-\frac12-\outve)\Hradius) + \c\exp(\frac2p\Hradius-(\frac12+\outve)\varsigma), \\
 \Vert\nu-\partial*_{{}_\Hradius r}\Vert_{\Wkp^{1,p}(\Phi(\M))}
  \le{}& C\exp((\frac2p-1)\Hradius)(d\exp({-}(\frac12+\outve)\Hradius) + \c\,\exp({-}(\frac12+\outve)\varsigma))
\end{align*}
for $\Hradius\in\interval{\Hradius_1^d}*\varsigma$, where $\partial*_{{}_\Hradius r}:=\houtlevi(\hdist(\centerz<\Hradius>,\,{\cdot}\,))$ denotes the radial vectorfield to the center point $\centerz<\Hradius>$. Therefore,
\begin{align*}
 \Vert\partial[\Hradius]@{\graphf<\Hradius>}\Vert_{\Wkp^{1,p}(\M<\Hradius>)}
 \le{}& C\,\Vert\rodeform\rnu-1\Vert_{\Wkp^{1,p}(\M<\Hradius>)}
	+ C\,\Vert\roboost\rnu\Vert_{\Wkp^{1,\infty}(\M<\Hradius>)}\,\Vert\partial*_r-\nu\Vert_{\Wkp^{1,p}(\M<\Hradius>)} \\
 \le{}& (\vphantom{\big|}d\masserr*+C)\,\exp((\frac2p-\frac12-\outve)\Hradius) + \masserr*\,\exp({-}(\frac12+\outve)\varsigma+\frac2p\Hradius).
\end{align*}
Therefore, an integration gives 
\begin{align*}\labeleq{graphf_imp_1}\hspace{5em}&\hspace{-5em}
 \Vert\houtexp(\vphantom{\big|}\graphf\,\nu)\circ\roPhi<\Hradius> - \Phi<\Hradius>\Vert_{\Wkp^{1,p}(\hsphere^2_\Hradius(\centerz))} \\
	\le{}& \exp(\frac2p\Hradius)((d\Masserr*+C)\exp({-}(\frac12+\outve)\Hradius) + \exp({-}(\frac12+\outve)\varsigma)\Masserr*).
\end{align*}
Note that we here control the graph function only up to the \emph{first} (and not second) derivative.

Furthermore, \eqref{graphf_assumption} implies
\begin{align*}
 \Vert\hlaplace(\rnu-1)+\frac{2(\rnu-1)}{\sinh(\Hradius)^2}\Vert_{\Lp^p(\M)}
 \le{}& (d\masserr*+C)\exp((\frac2p-\frac52-\outve)\Hradius) + \masserr*\,\exp({-}(\frac12+\outve)\varsigma).
\end{align*}
and
\[ \Vert(\hH^2-\H^2)(\rnu-1)\Vert_{\Lp^q(\M)}
		\le (d\masserr*+C)\exp((\frac2q-\frac52-\outve)\Hradius) + \masserr*\,\exp({-}2\Hradius-(\frac12+\outve)\varsigma). \]
Combined this gives
\[ \Vert\hjacobiext(\rnu-1)\Vert_{\Lp^q(\M)} \le (d\masserr*+C)\exp((\frac2q-\frac52-\outve)\Hradius) + \masserr*\,\exp({-}2\Hradius-(\frac12+\outve)\varsigma). \]
In particular, we have
\[ \Vert\partial[\Hradius]@{\hH}-\frac{\hH^2-4}2\Vert_{\Lp^p(\M)}
		\le (d\masserr*+C)\exp((\frac2p-\frac52-\outve)\Hradius) + \masserr*\,\exp({-}2\Hradius-(\frac12+\outve)\varsigma) \]
and per definition $\hH<\varsigma>\equiv{-}2\frac{\cosh(\varsigma)}{\sinh(\varsigma)}$. Solving this ordinary differential (in-)equality, we get
\begin{equation*}\labeleq{graphf_impication_mc}
\Vert\hH+2\frac{\cosh(\Hradius)}{\sinh(\Hradius)}\Vert_{\Lp^p(\M<\Hradius>)}
		\le \exp((\frac2p-2)\Hradius)((d\Masserr*+C)\exp({-}(\frac12+\outve)\Hradius) + \Masserr*\exp({-}(\frac12+\outve)\varsigma))
\end{equation*}
for every $\Hradius\in\interval{\Hradius_1^d}*\varsigma$, \ie if $d$ is large, then the mean curvature is closer to a constant as we a~priori assumed. As in \cite[Proof of Thm~A.2, Step~2]{nerz2016HBCMCExistence}---using~\eqref{graphf_assumption}---, this implies that there exists a function $\graphf<\Hradius>'\in\Ck^\infty(\hsphere^2_\Hradius(\centerz<\Hradius>)$ with $\rolaplace\graphf<\Hradius>'=\sinh(\Hradius)^{{-}2}(1-\exp(2\graphf<\Hradius>'))$ such that
\[ \Vert\graphf<\Hradius>-\graphf<\Hradius>'\Vert_{\Wkp^{2,p}(\hsphere^2_\Hradius(\centerz<\Hradius>))} \le \exp(\frac2p\Hradius)((d\Masserr*+C)\exp({-}(\frac12+\outve)\Hradius) + \Masserr*\exp({-}(\frac12+\outve)\varsigma)) \]
for every $\Hradius\in\interval{\Hradius_1^d}*\varsigma$. In combination with \eqref{graphf_imp_1}, the regularity of the Laplace operator gives
\begin{align*}
 \Vert\graphf<\Hradius>'\Vert_{\Wkp^{2,p}(\hsphere^2_\Hradius(\centerz<\Hradius>))}
 \le{}& C\,\Vert\graphf<\Hradius>'\Vert_{\Wkp^{1,p}(\hsphere^2_\Hradius(\centerz<\Hradius>))} \\
 \le{}& \exp((\frac2p-2)\Hradius)((d\Masserr*+C)\exp({-}(\frac12+\outve)\Hradius) + \Masserr*\exp({-}(\frac12+\outve)\varsigma))
\end{align*}
for all $\Hradius\in\interval{\Hradius_1^d}*\varsigma$. All in all, this proves
\begin{equation*}\labeleq{graphf_imp}
 \Vert\houtexp(\vphantom{\big|}\graphf\,\nu)\circ\roPhi<\Hradius> - \Phi<\Hradius>\Vert_{\Wkp^{2,p}(\hsphere^2_\Hradius(\centerz))} \le (\frac d4+C)\exp((\frac2p-\frac12-\outve)\Hradius) + \c\,\exp(\frac2p\Hradius-(\frac12+\outve)\varsigma)
\end{equation*}
for these $\Hradius$ if $\Hradius_0$ is sufficiently large. As explained above \eqref{graphf_imp} is true for all $\Hradius\in\interval\Hradius_0*\varsigma$ and $d=C$ (for the constant $C$ in~\eqref{graphf_imp}) independent of $\varsigma$.

\paragraph{Taking the limit}
We see that~\eqref{non_round_map} implies that $\outM$ is $\Ck^0_{\decay,\scdecay}$- and \Ckah^1_{\decay-1,\scdecay} if we can take the limit $\varsigma\to\infty$, \ie if $\Phi[\varsigma]:\interval\Hradius_0\Hradius_1\times\sphere^2\to\hyperbolicspace$ converges for every fixed $\Hradius_1$ to some limit map $\Phi[\infty]:\interval\Hradius_0\Hradius_1\times\sphere^2\to\hyperbolicspace$. Per construction of $\Phi$, this is true if $\Phi[\varsigma]<\Hradius>:\M<\Hradius>\to\hyperbolicspace$ converges for some fixed $s$ to some limit map $\Phi[\infty]<\Hradius>:\M<\Hradius>\to\hyperbolicspace$. By~\eqref{graphf_assumption} and the compactness of the Sobolev embeddings, this is true as $\Ck^1$-map for a subsequence of $\varsigma\to\infty$ if $\roPhi[\varsigma]<\Hradius>$ converges as $\varsigma\to\infty$. As \eqref{graphf_assumption} remains true in this limit, the proves the claim if a subsequence of $\roPhi[\varsigma]$ converges as $\varsigma\to\infty$. Thus, we only have to prove that $\roPhi[\varsigma]<\Hradius>:\M<\Hradius>\to\hyperbolicspace$ converges for one fixed $\Hradius>\Hradius_0$ as $\varsigma\to\infty$.

As first step, we prove that its image $\roPhi[\varsigma]<\Hradius>(\M<\Hradius>)$ is $\varsigma$-independent. A straightforward computation shows
\[ \partial[s]<s=\Hradius>(\hdist(\centerz[\varsigma]<\Hradius>,\centerz[\varsigma]<s>))
 = \Vert\roboost{(\rnu\circ\roPhi^{{-}1})}\Vert_{\Lp^\infty(\M)}
 \le \Vert\rnu-1\Vert_{\Lp^\infty(\M)} \le \masserr*. \]
In particular, if we choose the one-parameter family $(\varsigma,\varphi<\varsigma>,\centerz[\varsigma]<\varsigma>=0)$ of finite initial data and look at the corresponding $\Phi[\varsigma]<\Hradius>:\M<\Hradius>\stackrel\sim\to\sphere^2_\Hradius(\centerz[\varsigma]<\Hradius>)$, then $\centerz[\varsigma]<\Hradius>$ converges with $\varsigma\to\infty$ to some $\centerz[\infty]<\Hradius>\in\hyperbolicspace$. Thus, if we look at the alternative initial data $(\varsigma,\varphi<\varsigma>,\centerz[\varsigma]<\varsigma>:=\centerz[\infty]<\Hradius>)$, then the centers $\centerz<\Hradius>:=\centerz[\varsigma]<\Hradius>$ are independent of $\varsigma$ and satisfy $\centerz<\Hradius>\to0$ as $\Hradius\to\infty$. This proves the first claim.

Now, we see that in the Poincar\'e disc model of the hyperbolic space
\[ \vert\partial[\Hradius]@{\,\roPhi[\varsigma]}\vert_{\eukoutg*}
		= \frac14\,(1-\rad@{\roPhi[\varsigma]}^2)^2\vert\partial[\Hradius]@{\,\roPhi[\varsigma]}\vert_{\houtg*}
		\le \masserr(\Hradius) \]
and therefore $\roPhi[\Hradius]<\varsigma>:\sphere^2\to\hyperbolicspace$ converges to some $\psi[\varsigma]:\sphere^2\to\euksphere^2_1(0)$ for $\Hradius\to\infty$ and the metric $\sinh(\Hradius)^{{-}2}\g[\varsigma]<\Hradius>$ converges to a metric of constant Gau\ss\ curvature. Therefore, we can apply a diffeomorphism $\psi[\varsigma]:\sphere^2\to\sphere^2$ such that if we replace the initial data $(\varsigma,\varphi[\varsigma],\centerz[\varsigma])$ with $(\varsigma,\varphi[\varsigma]\circ\psi[\varsigma]^{{-}1},\centerz[\varsigma])$, then $\sinh(\Hradius)^{{-}2}\g[\varsigma]<\Hradius>$ converges to the standard metric of the Euclidean sphere---by \refth{Local_Estimates_rnu} and \refth{Ricci_mass_well-defined}, \eqref{graphf_assumption_round} remains true. Applying a rotation, we can furthermore assume that the limit $\psi[\varsigma]$ is independent of $\varsigma$. Per construction of $\roPhi$, this proves that $\roPhi[\varsigma]<\Hradius>$ converges as $\varsigma\to\infty$ for every $\Hradius>\Hradius_0$. As explained, this proves that $\Phi[\varsigma]$ converges to a chart of $\bigcup_\Hradius\M<\Hradius>$.

\paragraph{The asymptotic decay rate}
By \refth{second_fundamental_form} and \eqref{graphf_imp}, we know
\[ \Vert\zFundtrf\Vert_{\Lp^p(\M)} + \Vert\hH-\H\Vert_{\Lp^p(\M)} + \Vert\hzFundtrf\Vert_{\Lp^p(\M)} \le C\,\exp((\frac2p-\decay)\Hradius). \]
As $\partial*_\Hradius\,\hg<\Hradius>={-}2\rnu\zFund<\Hradius>$, an integration and the combination of~\eqref{graphf_imp_round} and~\eqref{graphf_imp} prove
\begin{equation*}\labeleq{graphf_imp_metric}\left.\begin{aligned}
 \Vert\g<\Hradius>-\hg<\Hradius>\Vert_{\Lp^p(\M)}
	\le{}& C\,\exp((\frac2p-\decay)\Hradius), \\
 \Vert\hlevi(\g<\Hradius>-\hg<\Hradius>)\Vert_{\Lp^p(\M<\Hradius>)}
	\le{}& C\,\exp((\frac2p-\frac32-\outve)\Hradius) 
 \end{aligned}\qquad\right\}\qquad\forall\,\Hradius\in\interval{\Hradius_0}\infty.
\end{equation*}
Furthermore, $\houtg(\partial*_\Hradius\Phi[\varsigma],\partial*_\Hradius\Phi[\varsigma])=\rnu^2=\houtg(\partial*_\Hradius,\partial*_\Hradius)$ and we therefore have proven
\begin{equation*}
 \int(\exp((\frac12+\outve')\rad)\vert\houtg-\outg\vert_{\houtg*}
		+ \exp((\frac32+\outve')\rad)\vert\houtlevi\outg\vert_{\houtg*})^p\exp({-}2\rad)\d\outmug
		\le C_{\outve'} \qquad\forall\,\outve'\in\interval0\outve
\end{equation*}
In particular, $\vert\outric-\houtric\vert_{\outg*}\le C\exp({-}(\frac52+\outve)\Hradius)$ implies
\begin{equation*}
 \int(\exp((\frac52+\outve')\rad)(\vert\houtlaplace\outg\vert_{\houtg*}+\vert\houtg-\outg\vert_{\houtg*}))^p\exp({-}2\rad)\d\outmug
		\le C_{\outve'} \qquad\forall\,\outve'\in\interval0\outve.
\end{equation*}
Thus, the regularity of the Laplace operator gives
\begin{equation*}
 \int(\exp((\frac52+\outve')\rad)(\vert\houtg-\outg\vert_{\houtg*}+\vert\houtlevi\outg\vert_{\houtg*}+\vert\houtlevi\houtlevi\outg\vert_{\houtg*}))^p\exp({-}2\rad)\d\outmug
		\le C.
\end{equation*}
Using the Sobolev inequality and reapplying the \emph{pointwise} inequality $\vert\outric-\houtric\vert_{\outg*}\le C\exp({-}(\frac52+\outve)\rad)$, we finally get
\begin{equation*}\labeleq{non_round_map}
 \ \sup_{\text{Im}\:\Phi}\lbrace\exp(\decay\rad)(\vert\pushforward{\Phi}\outg-\houtg\vert_{\houtg}+\vert\houtlevi(\pushforward{\Phi}\outg)\vert_{\houtg}+\vert\pushforward{\Phi}\outric-\houtric\vert) + \exp(\scdecay\rad)\vert\outsc+6\vert\rbrace
		\le C.
\end{equation*}
Note that we cannot get pointwise estimates on $\houtlevi\houtlevi(\pushforward\Phi\outg)$.

Here, $\vert\outsc+6\vert$ has to be replaced by $\vert(\outsc+6)^-\vert$ if we only assumed lower bounds on $\outsc+6$, see Remarks~\ref{BoundednessScalarCurvature_coordinates}, \ref{BoundednessScalarCurvature_spheres}, \ref{BoundednessScalarCurvature_stability}, \ref{BoundednessScalarCurvature_rnu}, and~\ref{BoundednessScalarCurvature_mass}.
\end{proof}

\begin{remark}[The Euclidean setting]
On the first glimpse, it may look like we can apply the same proof in the Euclidean setting and can therefore replace the proof done in \cite{nerz2015GeometricCharac}. However, the proof presented above crucial depends on the fact that $\roPhi[\varsigma]$ converges as $\varsigma$ goes to infinity which makes it necessary that $\centerz[\varsigma]<\Hradius>$ converges as $\varsigma\to\infty$. As explained, this is the case if (and only if) the $\Hradius$-derivative of $\centerz[\varsigma]<\Hradius>$ is integrable. However, this is equivalent to convergence of the so called CMC-center of mass. In the hyperbolic setting, this is always true as the CMC-center of mass is always well-defined, see \cite{cederbaum2015center,nerz2016HBCMCExistence}, but in the Euclidean setting this is \emph{not} always true, see \cite{cederbaumnerz2013_examples,nerz2015CMCfoliation}.
This explains why the author did in the Euclidean setting not choose the chart construction \lq by integration along the lapse function and fixing the chart at infinity\rq.

Equivalently, we can not apply the Euclidean construction in the hyperbolic setting as it heavily relies on the linear structure of the Euclidean space which is not given in the hyperbolic space---as can be seen by comparing translations with their hyperbolic counterparts being boosts.
\end{remark}

\section{Replacing the asymptotic end}\label{ReplacingAsymptoticEnd}
As toy application of the above construction, we replace the hyperbolic end of an asymptotically hyperbolic manifold with an Euclidean one.

\begin{proof}<ReplacingTheEnd>
Let $\mathcal M:=\{\M<\Hradius>\}_{\Hradius>\Hradius_0}$ be the CMC-foliation of $\outM$, $\rnu$ be the lapse function of $\mathcal M$, and
\[ \varphi(\Hradius):=6\:\frac{\cosh(\Hradius)^2}{\sinh(\Hradius)^2}\frac{\psi^2(\Hradius)-1}{\psi(\Hradius)^2}-\frac4{\rnu\psi^2}(\frac{\psi(\Hradius)^2-1}{\sinh(\Hradius)^2}-\frac{\psi'(\Hradius)\cosh(\Hradius)}{\psi(\Hradius)\sinh(\Hradius)}), \]
where $\psi\in\Ck^1(\interval0\infty)$ is an arbitrary monotone increasing function such that $\psi|_{\interval0{\Hradius_1'}}\equiv1$ and $\psi(\Hradius)=\cosh(\Hradius)$ for every $\Hradius\in\interval\Hradius_1\infty$ for some arbitrary $\Hradius_1>\Hradius_1'>\Hradius_0$.

Writing the metric in its components orthogonal to the foliation and tangential to it, we know
\[ \outg = \rnu^2\d\outHradius^2 + \g<\Hradius> \quad\text{on }\textrm T\outM|_{\M<\Hradius>}\qquad\forall\,\Hradius>\Hradius_0, \]
where the \lq radial\rq\ function $\outHradius$ is defined by $\outHradius|_{\M<\Hradius>}:=\Hradius$ for every $\Hradius>\Hradius_0$ and the lapse function $\rnu$ can be written as $\rnu=\vert\d\outHradius\vert^{{-}1}$. Now, we define the new metric
\begin{equation*}\labeleq{ReplacingTheEnd_construction} \foutg := \psi(\Hradius)^2\d\outHradius^2 + \g<\Hradius> \quad\text{on }\textrm T\outM|_{\M<\Hradius>}\qquad\forall\,\Hradius>\Hradius_0
\end{equation*}
and get
\begin{equation*}\labeleq{stretching_zFund}
 \fzFund<\Hradius> = \frac{{-}1}{2\frnu<\Hradius>}\lieD{{}^{\,\flatf\!}_\Hradius\!X}\foutg = \psi(\Hradius)^{{-}1}\frac{{-}1}{2\rnu<\Hradius>}\lieD{{}_\Hradius\!X}\outg = \psi(\Hradius)^{{-}1}\:\zFund<\Hradius>,
\end{equation*}
where ${}^\flatf X:=\frnu<\Hradius>\:\fnu<\Hradius>=\rnu<\Hradius>\:\nu<\Hradius>=:X$ and where $\frnu<\Hradius>:=\cosh(\Hradius)\:\rnu<\Hradius>$ denotes the lapse function of $\{\M<\varsigma>\}_\varsigma$ on $\M<\Hradius>$ with respect to $\foutg$ and where $\lieD{X}T$ denotes the Lie derivative of a tensor~$T$ in direction of a vector field~$X$. In particular, $\{\M<\Hradius>\}_\Hradius$ is a CMC-foliation with respect to $\foutg*$, too, where the \lq new\rq\ mean curvature satisfies
\[ \fH<\Hradius> := \fH(\M<\Hradius>) = \psi(\Hradius)^{{-}1}\H<\Hradius> \equiv \frac{{-}2\cosh(\Hradius)}{\psi(\Hradius)\sinh(\Hradius)} =: {-}\frac{2}{\varsigma} \]
Note that the metric $\fg<\Hradius>=\g<\Hradius>$ does not change for any of the CMC-leaves. This proves that the Hawking mass of any surfaces with respect to $\foutg$ is equal to its hyperbolic Hawking mass with respect to $\outg$. From now on, we suppress the index $\Hradius$ and write $\psi$, $\cosh$, and $\sinh$ instead of $\psi(\Hradius)$, $\cosh(\Hradius)$, and $\sinh(\Hradius)$, respectively.

The Codazzi equation and \eqref{stretching_zFund} imply
\[ \foutric(\fnu,\emptyarg) = \div(\fH\:\fg-\fzFund) = \psi^{{-}1}\:\div(\H\g-\zFund) = \psi^{{-}1}\:\outric(\nu,\emptyarg) \]
as $1$-form on any leaf of the foliation, where $\fnu$ and $\nu$ denote the out unit normal of $\M$ with respect to $\foutg$ and $\outg$, respectively. In particular, this proves the decay behavior for $\outric(\nu,\emptyarg)$ we aimed for.

Equation~\eqref{stretching_zFund} furthermore implies
\[ \partial[\Hradius]@{\:\fzFund} = \frac1{\psi}\partial[\Hradius]@{\:\zFund} - \frac{\psi'}{\psi^2}\zFund, \qquad
		\partial[\Hradius]@{\:\fH} = \frac1{\psi}\partial[\Hradius]@{\H} - \frac{\psi'}{\psi^2}\H \]
and by characterization of the stability operator, we know
\[ \frac1\frnu\partial[\Hradius]@{\:\fH} - \foutric(\fnu,\fnu) - \trtr\fzFund\fzFund = \frac{\laplace\frnu}\frnu = \frac{\laplace\rnu}\rnu = \frac1\rnu\partial[\Hradius]@\H - \outric(\nu,\nu) - \trtr\zFund\zFund. \]
Thus, \eqref{stretching_zFund} and $2\:\partial*_\Hradius\H=\H^2-4$ imply
\begin{align*}
 \foutric(\fnu,\fnu)
 ={}& \outric(\nu,\nu)
			- \frac1\rnu\frac{\psi^2-1}{\psi^2}\partial[\Hradius]@\H - \frac{\H\psi'}{\rnu\psi^3}
			+ \frac{\psi^2-1}{\psi^2}\frac{\H^2}2 + \frac{\psi^2-1}{\psi^2}\trtr\zFundtrf\zFundtrf \\
 ={}& \outric(\nu,\nu)
			+ 2\:\frac{\psi^2-1}{\psi^2}\frac{\cosh^2}{\sinh^2}
			+ \frac{\psi^2-1}{\psi^2}\trtr\zFundtrf\zFundtrf \\
		&	- \frac2{\rnu\:\psi^2}(\frac{\psi^2-1}{\sinh^2} - \frac{\psi'\cosh}{\psi\sinh}).
\end{align*}
We note that for $\psi=\cosh$, \ie for large $\Hradius$, this simplifies to
\[ \foutric(\fnu,\fnu) = \outric(\nu,\nu) + 2	+ \frac{\sinh^2}{\cosh^2}\trtr\zFundtrf\zFundtrf. \]
In particular, this proves
\[ \vert\foutric(\fnu,\fnu)\vert \le C\,\sinh(\Hradius)^{{-}\decay} \]
and therefore we have the correct decay behavior for $\foutric(\fnu,\fnu)$, too.

On $T\M^2$, we have
\[ \ric - \frac{\Hess\:\rnu}\rnu
	= \outric + \frac12\H^2\g - 2\trzd\zFundtrf\zFundtrf - \frac1{2\rnu}\partial[\Hradius]@\H\g - \frac1{\rnu}\partial[\Hradius]@\zFundtrf \]
(and equivalent for ${}^\flatf\emptyarg$)
and therefore a calculation as for $\foutric(\fnu,\fnu)$ gives
\begin{align*}
 \foutric
 ={}& \outric + 2\:\frac{\psi^2-1}{\psi^2}\frac{\cosh^2}{\sinh^2}\g - \frac1{\rnu\:\psi^2}(\frac{\psi^2-1}{\sinh^2}-\frac{\psi'\cosh}{\psi\sinh})\g\\
		&- \frac{\psi^2-1}{\psi^2}(2\:\trzd\zFundtrf\zFundtrf+\frac1\rnu\partial[\Hradius]@\zFundtrf) - \frac{\psi'}{\rnu\psi^3}\zFundtrf.
\end{align*}
In particular, for $\psi=\cosh$
\[ \foutric = \outric + 2\g - \frac{\sinh^2}{\cosh^2}(2\:\trzd\zFundtrf\zFundtrf+\frac1\rnu\partial[\Hradius]@\zFundtrf + \frac{\zFundtrf}{\rnu\cosh\sinh}) \]
which implies the decay behavior of $\foutric|_{T\M^2}$ we aimed for. All in all this proves the decay behavior for $\foutric$.

By the above, we know
\[
 \foutsc
	= \outsc + 6\:\frac{\psi^2-1}{\psi^2}\frac{\cosh^2}{\sinh^2} - \frac4{\rnu\:\psi^2}(\frac{\psi^2-1}{\sinh^2}-\frac{\psi'\cosh}{\psi\sinh})
		 + \frac{\psi^2-1}{\psi^2}\trtr\zFundtrf\zFundtrf
\]
where we used $\tr(\partial[\Hradius]\zFundtrf)={-}\partial[\Hradius]@{\g^{\ii\ij}}\zFundtrf_{\ii\ij}={-}2\rnu\trtr\zFundtrf\zFundtrf$. Again for $\psi=\cosh$, this gives
\[ \foutsc = \outsc + 6 + \frac{\sinh^2}{\cosh^2}\trtr\zFundtrf\zFundtrf \ge \outsc + 6 \]
which gives the necessary decay estimate on $\foutsc$, too. By \cite{nerz2015GeometricCharac}, this proves all claims.
\end{proof}

\begin{proof}<HawkingMassSmallerTotalMass>
Fix any $\Hradius>\Hradius_1$ and choose a function $\psi\in\Ck^1(\interval0\infty)$ as in \refth{ReplacingTheEnd} with $\psi(s)=\cosh(s)$ for every $s>\frac{\Hradius_1+\Hradius}2$. By \refth{ReplacingTheEnd}, we can transform $\outg$ to an asymptotically Euclidean metric $\foutg$ such that the hyperbolic Hawking mass of $\M=\M<\Hradius>$ with respect to $\outg$ is the same as the Hawking mass of it with respect to $\foutg$. By the Gau\ss\ equation and \cite{herzlich2015computing}, we know
\[ \mass*_{\text{ADM}}(\outM,\foutg) \xleftarrow{\Hradius\to\infty} \mHaw(\M<\Hradius>,\foutg*) = \HmHaw(\M<\Hradius>,\outg*) \xrightarrow{\Hradius\to\infty} \mass*(\outM,\outg*) = \vert\mass\vert_{\R^{3,1}}. \]
Per definition, we know that $\M<\Hradius>$ has constant mean curvature ${-}\frac2\Hradius$ with respect to $\foutg*$. By \cite{bray1997mon}, $\mHaw(\M<\Hradius>,\foutg*)$ is therefore monotone increasing in $\Hradius$. All but the last claim.

If $\HmHaw(\M<\Hradius>)=\mass*$ for one $\Hradius$, then (by the above) $\mHaw(\M<s>,\foutg*)=\mass*_{\text{ADM}}(\outM,\foutg)$ for every $s\ge\Hradius$. In particular, the Hawking mass is constant along the inverse mean curvature flow starting at $\M<\Hradius>$ and therefore $(\outM,\foutg*)$ is outside of $\M<\Hradius>$ identical to the Schwarzschild solution, see~\cite{huisken2001inverse}. By \cite{brendle2013constant,nerz2015CMCfoliation}, this proves $\M<s>=\sphere^2_{\rradius(s)}(0)$ implying $\sc<s>\equiv\text{const}(s)$, $\fzFundtrf<s>\equiv0$, $\foutric(\nu,\nu)|_{\M<s>}\equiv\text{const}(s)$, $\foutric(\fnu<s>,\emptyarg)|_{T\M<s>}\equiv0$, and $\foutric|_{T\M<s>^2}=\text{const}(s)\:\fg<s>$ for every $s\ge\Hradius$. Revisiting the above equations for $\foutric$ \etc, this proves the last claim.
\end{proof}

\appendix
\section{Boosting a sphere}
As boosts are isometries of the hyperbolic space, it is quite obvious that if we have a one-parameter family $\outPhi[t]:\hyperbolicspace\to\hyperbolicspace$ of boosts of the hyperbolic space with $\outPhi[0]=\id_{\hyperbolicspace}$, then it maps $\houtg$-geodesic spheres to $\houtg$-geodesic spheres (of the same radius), \ie $\outPhi[t](\hsphere^2_\Hradius(\centerz[0]))=\hsphere^2(\centerz[t])$. In particular, $\hH[t]:=\hH(\outPhi[t](\hsphere^2_\Hradius(\centerz[0])))$ is a $t$-independent constant and therefore
\begin{align*}
 0 \equiv{}& \partial[t]@{\,\hH[t]} = \jacobiext*\houtg(\hnu[t],\partial[t]@{\,\outPhi[t]})
	= \laplace \houtg(\hnu[t],\partial[t]@{\,\outPhi[t]}) + (\outric(\nu,\nu)+\trtr{\zFund[t]}{\zFund[t]})\houtg(\hnu[t],\partial[t]@{\,\outPhi[t]}) \\
	={}& \laplace \houtg(\hnu[t],\partial[t]@{\,\outPhi[t]}) + \frac2{\sinh(\Hradius)^2}\houtg(\hnu[t],\partial[t]@{\,\outPhi[t]}),
\end{align*}
\ie $\houtg(\hnu[t],\partial[t]@{\,\outPhi[t]})$ is an eigenfunction of the (negative) Laplace operator with eigenvalue $2\,\sinh(\Hradius)^{{-}2}$.

In the proof of the main theorem, we need the exact reverse implication, \ie if $\outPhi[t]|_{\sphere^2}$ is a one parameter family of embeddings of $\sphere^2$ to the hyperbolic space such that at each time the lapse function is an eigenfunction of the (negative) Laplace operator with eigenvalue $8\pi\volume{\outPhi[t](\sphere^2)}^{{-}1}$ and $\outPhi[0](\sphere^2)$ is a geodesic sphere, then $\outPhi[t](\sphere^2)$ is always a sphere. Although this result is not new and the proof is quite straightforward, we give it for the sake of completeness. To adept to the setting of the main theorem, we allow the spheres to also be rescaled along $\outPhi[t]$.
\begin{lemma}[Deformations which are boostings]\labelth{Boosting_deformations}
Let $T>0$ be a constant and $\Phi:\interval*0*T\times\sphere^2\to\hyperbolicspace:(t,p)\mapsto\Phi(t,p)$ be a smooth deformation of a geodesic sphere which does not collapse. If the lapse function is the combination of rescalings and linearized boosts for all deformation indexes~$t$, then $\Phi(t,\sphere^2)$ is a geodesic sphere for all deformation indexes~$t$.

This means if $\Phi$ is smooth, $\Phi(0,\sphere^2)=\hsphere_r(\centerz)$ for some $r>0$ and $\centerz\in\hyperbolicspace$, $\Phi(\emptyarg,\sphere^2)$ does not collapse to a single point, and the lapse function $\hrnu<t>:=\houtg(\partial*_t\Phi,\hnu<t>)$ satisfies
\[ \forall\,t\in\interval*0*T\qquad \hlaplace<t>\:\rnu<t> = {-}\frac{8\pi}{\hvolume{\Phi(t,\sphere^2)}}(\hrnu<t>-\fint\rnu<t>\d\:\hmug<t>), \]
then
\[ \forall\,t\in\interval*0*T\quad\exists\,\centerz<t>\in\hyperbolicspace,\ r(t):=r+\int_0^t\fint\hrnu<t>\d\:\hmug<t>\d s\qquad\Phi(t,\sphere)=\hsphere_{r(t)}(\centerz<t>). \]
Here, $\hlaplace<t>$, $\hnu<t>$, and $\hmug<t>$ denote the Laplace operator, the outer unit normal, and the volume measure of $\Phi(t,\sphere^2)\hookrightarrow\hyperbolicspace$, respectively.
\end{lemma}
\begin{proof}
Let us first assume that $r(t):=r+\int_0^t\fint\rnu<t>\d\:\mug<t>>0$ for every $t\in\interval*0*T$. Using the definition of the stability operator, we know
\begin{equation*}
 \partial[t]@{\:\hH<t>}
 = \frac12(\hH<t>^2 - 4 + 2\vert\hzFundtrf<t>\vert_{\hg<t>}^2)\hrnu<t> - \frac{8\pi}{\hvolume{\Phi(t,\sphere^2)}}\hrnu<t>.
\end{equation*}
Furthermore, $\Phi(t,\sphere^2)$ is a extrinsic round sphere, \ie $\zFundtrf\equiv0$, if and only if $\hH<t>\equiv{-}2\frac{\cosh(R)}{\sinh(R)}$ is constant for some $R>0$ and then $\hvolume{\Phi(t,\sphere^2)}=4\pi\sinh(R)^2$, see \cite{brendle2013constant}. Therefore, if $\hH<t>\equiv{-}2\,\frac{\cosh(r(t))}{\sinh(r(t))}$ for some $t\in\interval*0*T$, then
\begin{equation*}
 \partial[t]@{\:\hH<t>}
 \equiv 2\,\sinh(r(t))^{{-}2}\,\fint\rnu<t>\d\:\hmug<t>
 = \partial[t]({-}2\frac{\cosh(r(t))}{\sinh(r(t))}),
\end{equation*}
for this $t$. With $\hH<0>\equiv{-}2\,\frac{\cosh(r(0))}{\sinh(r(0))}$, this implies that $\hH<t>\equiv{-}2\,\frac{\cosh(r(t))}{\sinh(r(t))}$ for all $t\in\interval*0*T$. Again by \cite{brendle2013constant}, this proves the claim.

Now, assume that $T'\in\interval0*T$ is minimal with $r(T')=0$. By the above, $\Phi(t,\sphere^2)=\hsphere^2_{r(t)}(\centerz<t>)$ for some $\centerz<t>$ for every $t<T'$. Furthermore, $r(t)\to0$ for $t\to T'$  and therefore the diameter of $\Phi(t,\sphere^2)$ goes to $0$ as $t\to T'$ which contradicts the assumption that $\Phi$ is  non-collapsing.
\end{proof}%

\clearpage																							%
\bibliography{bib}																			%
\makeatletter
\def\bibindent{10em}
\let\old@biblabel\@biblabel
\def\@biblabel#1{\old@biblabel{#1}\kern\bibindent}
\makeatother
\vfill\bibliographystyle{alpha}\vfill
\end{document}